\documentclass[10pt]{article}

\usepackage{amsmath, amssymb, amsthm,amsfonts}
\usepackage{centernot}
\usepackage{mathpazo}
\usepackage[margin=2.5cm]{geometry}

\usepackage{color, graphicx, soul}
\usepackage{subfig}
\usepackage{adjustbox}
\usepackage{dsfont} 
\usepackage{resizegather}

\usepackage{xargs}
\usepackage{cite,microtype}
\usepackage{dsfont} 

\usepackage[colorinlistoftodos,prependcaption,textsize=footnotesize]{todonotes}
\newcommandx{\attn}[2][1=]{\todo[linecolor=red,backgroundcolor=blue!25,bordercolor=red,#1]{#2}}
\newcommandx{\other}[2][1=]{\todo[linecolor=OliveGreen,backgroundcolor=OliveGreen!25,bordercolor=OliveGreen,#1]{#2}}
\newcommandx{\thiswillnotshow}[2][1=]{\todo[disable,#1]{#2}}

\usepackage{hyperref}
\hypersetup{
	colorlinks,
	linkcolor={red!50!black},
	citecolor={blue!50!black},
	urlcolor={blue!80!black}
}

\usepackage{empheq}
\definecolor{myblue}{rgb}{.9, .9, 1}

\makeatletter
\def\namedlabel#1#2{\begingroup
   \def\@currentlabel{#2}%
   \label{#1}\endgroup
}
\makeatother

\makeatletter
\def\th@plain{%
  \thm@notefont{}
  \itshape 
}
\def\th@definition{%
  \thm@notefont{}
  \normalfont 
}
\makeatother

\newtheorem{theorem}{Theorem}[section]
\newtheorem{lemma}[theorem]{Lemma}
\newtheorem{corollary}[theorem]{Corollary}
\newtheorem{proposition}[theorem]{Proposition}

\newtheorem{assumption}[theorem]{Assumption}
\theoremstyle{definition}
\newtheorem{definition}[theorem]{Definition}
\theoremstyle{definition}
\newtheorem{example}[theorem]{Example}
\theoremstyle{definition}
\newtheorem{remark}[theorem]{Remark}

\newcommand{\distP}{\ell _{\text{poly}}}
\usepackage[shortlabels]{enumitem}
\setlist[enumerate]{nosep}

\allowdisplaybreaks

\newcommand{\dpp}{d_{\text{PPS}}}

\newcommand{\abd}{\ensuremath{\mathbf a}}
\newcommand{\xx}{\ensuremath{\mathbf x}}
\newcommand{\yy}{\ensuremath{\mathbf y}}
\newcommand{\zz}{\ensuremath{\mathbf z}}
\newcommand{\ww}{\ensuremath{\mathbf w}}
\newcommand{\vv}{\ensuremath{\mathbf v}}
\newcommand{\pp}{\ensuremath{\mathbf p}}
\newcommand{\qq}{\ensuremath{\mathbf q}}
\newcommand{\uu}{\ensuremath{\mathbf u}}

\newcommand{\Fhat}{\ensuremath{{\mathcal F}}}

\newcommand{\fraks}{\ensuremath{\mathfrak{s}}}

\newcommand{\frakg}{\ensuremath{\mathfrak{g}}}

\newcommand{\Ws}{\ensuremath{\mathcal W}_{-1}}
\newcommand{\NN}{\ensuremath{\mathbb N}}

\newcommand{\RR}{\ensuremath{\mathbb R}}
\newcommand{\RP}{\ensuremath{\mathbb{R}_+}}

\newcommand{\argmin}{\ensuremath{\operatorname*{argmin}}}

\newcommand{\bd}{\partial} 

\newcommand{\cone}{\ensuremath{\operatorname{cone}}}

\newcommand{\dom}{\ensuremath{\operatorname{dom}}}

\newcommand{\dist}{\ensuremath{\operatorname{d}}}
\newcommand{\wdist}{{\widehat\dist}}

\newcommand{\fenv}[2][]%
{\ensuremath{\,\overrightarrow{\operatorname{env}}_{#2}}^{#1}}
\newcommand{\benv}[2][]%
{\ensuremath{\,\overleftarrow{\operatorname{env}}_{#2}^{#1}}}
\newcommand{\env}[2][]%
{\ensuremath{{\operatorname{env}}_{#2}^{#1}}}
\newcommand{\W}{\mathcal W}
\newcommand{\bP}[2][]%
{\ensuremath{\,\overleftarrow{\operatorname{P}}_{#2}^{#1}}}

\newcommand{\fP}[2][]%
{\ensuremath{\,\overrightarrow{\operatorname{P}}_{#2}^{#1}}}

\newcommand{\proj}[1]{{\operatorname{P}\thinspace}_%
	{\negthinspace\negthinspace #1}}

\usepackage{xcolor}
\usepackage{tikz}
\usetikzlibrary{calc}
\usetikzlibrary{decorations.pathmorphing}

\newcommand{\norm}[1]{\|#1\|}
\newcommand{\stdFace}{ \ensuremath{\mathcal{F}}}

\newcommand{\stdCone}{ \ensuremath{\mathcal{K}}}
\newcommand{\spanVec}{\ensuremath{\mathrm{span}\,}}
\newcommand{\stdSpace}{ \ensuremath{\mathcal{L}}}

\newcommand{\ambSpace}{\ensuremath{\mathcal{E}}}
\newcommand{\reInt}{\ensuremath{\mathrm{ri}\,}}
\newcommand{\expCone}{\ensuremath{K_{\exp}}}
\newcommand{\face}{\mathrel{\unlhd}}
\newcommand{\inProd}[2]{\langle #1 , #2 \rangle }
\newcommand{\comp}{\diamondsuit}

\numberwithin{equation}{section}

\newcommand{\hzz}{\ensuremath{\widehat{\mathbf z}}}
\newcommand{\hff}{\ensuremath{\widehat{\mathbf f}}}
\newcommand{\hpp}{\ensuremath{\widehat{\mathbf p}}}
\newcommand{\tzz}{\ensuremath{\widetilde{\mathbf z}}}
\newcommand{\tff}{\ensuremath{\widetilde{\mathbf f}}}
\newcommand{\tpp}{\ensuremath{\widetilde{\mathbf p}}}
\newcommand{\bv}{\ensuremath{\bar{\vv}}}
\newcommand{\amf}{\ensuremath{\mathfrak{g}}}
\newcommand{\samf}{\hat\amf}
\newcommand{\oneFRF}{{one-step facial residual function}}
\newcommand{\oneFRFs}{{$\mathds{1}$-FRF}}
\newcommand{\OneFRF}{{One-step facial residual function}}

\begin{document}

\title{Error bounds, facial residual functions and applications to the exponential cone}

\author{
Scott B.\ Lindstrom\thanks{
School of Electrical Engineering, Computing and Mathematical Sciences, Faculty of Science and Engineering,
Curtin University, Australia.
E-mail: \href{scott.lindstrom@curtin.edu.au}{scott.lindstrom@curtin.edu.au}.}
\and
Bruno F. Louren\c{c}o\thanks{Department of Statistical Inference and Mathematics, Institute of Statistical Mathematics, Japan.
	This author was supported partly by JSPS Grant-in-Aid for Early-Career Scientists 19K20217 and the Grant-in-Aid for Scientific Research (B)18H03206.
	Email: \href{bruno@ism.ac.jp}{bruno@ism.ac.jp}}
\and
Ting Kei Pong\thanks{
Department of Applied Mathematics, the Hong Kong Polytechnic University, Hong Kong, People's Republic of China.
This author was supported partly by Hong Kong Research Grants Council PolyU153003/19p.
E-mail: \href{tk.pong@polyu.edu.hk}{tk.pong@polyu.edu.hk}.
}
}

\date{\today}

\maketitle

\begin{abstract} \noindent
We construct a general framework for deriving error bounds for conic feasibility problems. In particular, our approach allows one to work with cones that fail to be amenable or even to have computable projections, two previously challenging  barriers. For the purpose, we first show how error bounds may be constructed using objects called \textit{one-step facial residual functions}. Then, we develop several tools to compute these facial residual functions even in the absence of closed form expressions for the  projections onto the cones.
We demonstrate the use and power of our results by computing tight error bounds for the exponential cone feasibility problem.
Interestingly, we discover a natural example for which
the tightest error bound is related to the Boltzmann-Shannon entropy. We were also able to produce an example of sets for which a H\"{o}lderian error bound holds but the supremum of the set of admissible exponents is not itself an admissible exponent.
\end{abstract}

{\small

}

\section{Introduction}

Our main object of interest is the following convex conic feasibility problem:
\begin{align}
\text{find} & \quad \xx \in (\stdSpace + \abd) \cap \stdCone \label{eq:feas}\tag{Feas},
\end{align}
where $\stdSpace$ is a subspace contained in some finite-dimensional real Euclidean space $\ambSpace$, $\abd \in \ambSpace$ and $\stdCone \subseteq \ambSpace$ is a closed convex cone. For a discussion of some applications and algorithms for \eqref{eq:feas} see \cite{HM11}. See also \cite{BB96} for a broader analysis of convex feasibility problems.
We also recall that a \emph{conic linear program} (CLP) is the problem of minimizing a linear function subject to a constraint of the form described in \eqref{eq:feas}.
In addition, when the optimal set of a CLP is non-empty it can  be written as the intersection of a cone with an affine set.
This provides yet another motivation for analyzing \eqref{eq:feas}: to better understand feasible regions and optimal sets of conic linear programs.
Here, our main interest is in obtaining \emph{error bounds} for \eqref{eq:feas}. That is, assuming $(\stdSpace+\abd)\cap \stdCone\neq \emptyset$, we want an inequality that, given some arbitrary $\xx \in \ambSpace$, relates the \emph{individual distances} $\dist(\xx, \stdSpace+\abd), \dist(\xx,\stdCone)$ to the \emph{distance to the intersection} $\dist(\xx, (\stdSpace+\abd)\cap \stdCone)$. Considering that
$\ambSpace$ is equipped with some norm $\norm{\cdot}$ induced by some inner product $\inProd{\cdot}{\cdot}$, we recall that the \emph{distance function to a convex set $C$} is defined as follows:
\begin{equation*}
\dist(\xx, C) \coloneqq \inf _{\yy \in C} \norm{\xx-\yy}.
\end{equation*}

When $\stdCone$ is a polyhedral cone, the classical Hoffman's error bound \cite{HF52} gives a relatively complete picture of the way that the individual distances relate to the distance to the intersection.
If $\stdCone$ is not polyhedral, but $\stdSpace+\abd$ intersects $\stdCone$ in a sufficiently well-behaved fashion (say, for example, when $\stdSpace+\abd$ intersects $\reInt \stdCone$, the relative interior of $\stdCone$; see Proposition~\ref{prop:cq_er}), we may still expect ``good'' error bounds to hold, e.g., \cite[Corollary~3]{BBL99}. However, checking whether  $\stdSpace+\abd$ intersects $\reInt \stdCone$ is not necessarily a trivial task; and, in general, $(\stdSpace+\abd)\cap \reInt\stdCone$ can be void.

Here, we focus on error bound results that \emph{do not} require any assumption on the way that the affine space $\stdSpace+\abd$ intersects $\stdCone$. So, for example, we want results that are valid even if, say,
$\stdSpace+\abd$ fails to intersect the relative interior of $\stdCone$. 
Inspired by Sturm's pioneering work on error bounds for positive semidefinite systems \cite{St00}, the class of \emph{amenable cones} was proposed in \cite{L17} and it was shown that the following three ingredients can be used to obtain general error bounds for \eqref{eq:feas}: (i) amenable cones, (ii) facial reduction \cite{BW81,Pa13,WM13} and (iii) the so-called facial residual functions (FRFs) \cite[Definition~16]{L17}.

In this paper, we will show that, in fact, it is possible to obtain error bounds for \eqref{eq:feas} by using the so-called \emph{{\oneFRF}s} directly in combination with facial reduction.
It is fair to say that computing the facial residual functions is the most critical step in obtaining error bounds for \eqref{eq:feas}. We will demonstrate techniques that are readily adaptable for the purpose. 

All the techniques discussed here will be showcased with error bounds for the so-called \emph{exponential cone} which is defined as follows\footnote{Our notation for the exponential cone coincides with the one in \cite{OCPB16}. However, the $x,y,z$ variables might appear permuted in other papers, e.g., \cite{SY15,MC2020,PY21}. }:
\begin{align*}
\expCone:=&\left \{(x,y,z)\in \RR^3\;|\;y>0,z\geq ye^{x/y}\right \} \cup \left \{(x,y,z)\;|\; x \leq 0, z\geq 0, y=0  \right \}.
\end{align*}
Put succinctly,  the exponential cone is the closure of the epigraph of the perspective function of $z=e^x$. It is quite useful in entropy optimization, see \cite{CS17}. Furthermore, it is also
implemented in the MOSEK package, see \cite{DE21}, \cite[Chapter 5]{MC2020}, and the many modelling examples in Section~5.4 therein. There are several other solvers that either support the exponential cone or convex sets closely related to it \cite{OCPB16,PY21,KT19,CKV21}. See also \cite{Fr21} for an algorithm for projecting onto the exponential cone. So convex optimization with exponential cones is widely available even if, as of this writing, it is not as widespread as, say, semidefinite programming.

The exponential cone $\expCone$ appears, at a glance, to be simple. However, it possesses a very intricate geometric structure that illustrates a wide range of challenges practitioners may face in computing error bounds. First of all, being non-facially-exposed, it is not amenable, so the theory developed in \cite{L17} does not directly apply to it. Another difficulty is that not many analytical tools have been developed to deal with the projection operator onto $\expCone$ (compared with, for example, the projection operator onto PSD cones) which is only implicitly specified.
Until now, these issues have made it
challenging the establishment of error bounds for objects like $\expCone$, many of which are of growing interest in the mathematical programming community.

Our research is at the intersection of two topics: \emph{error bounds} and the \emph{facial structure of cones}. General information on the former can be found, for example,  in \cite{Pang97,LP98}. Classically, there seems to be a focus on the so-called \emph{H\"olderian error bounds} (see also \cite{Li10,Li13,LMP15}) but we will see in this paper that non H\"olderian behavior can still appear even  in relatively natural settings such as conic feasibility problems associated to the exponential cone.

Facts on the facial structure of convex cones can be found, for example, in \cite{Ba73,Ba81,Pa00}. We recall that a cone is said to be \emph{facially exposed} if each face arises as the intersection of the whole cone with some supporting hyperplane. Stronger forms of facial exposedness have also been studied to some extent, here are some examples: projectional exposedness~\cite{BW81,ST90}, niceness~\cite{Pa13_2,Ve14}, tangential exposedness~\cite{RT19}, amenability \cite{L17}.
See also \cite{LRS20} for a comparison between a few different types of facial exposedness.
These notions are useful in many topics, e.g.:
regularization of convex programs and extended duals~\cite{BW81,Pa13,LiuPat18}, studying the closure of certain linear images~\cite{Pa07,LiuPat18}, lifts of convex sets~\cite{GPR13} and error bounds \cite{L17}.
However, as can be seen in Figure~\ref{fig:cone}, the exponential cone is not even a facially exposed cone, so none of the aforementioned notions apply (in particular, the face $\stdFace_{ne}:=\{(0,0,z)\;|\; z \geq 0\}$ is not exposed).
This was one of the motivations for looking beyond facial exposedness and developing a framework for deriving error bounds for feasibility problems associated to general closed convex cones.

\begin{figure}
	\begin{center}
		\begin{tikzpicture}
		[scale=0.76]
		\node[anchor=south west,inner sep=0] (image) at (0,0) {\includegraphics[width=.4\linewidth]{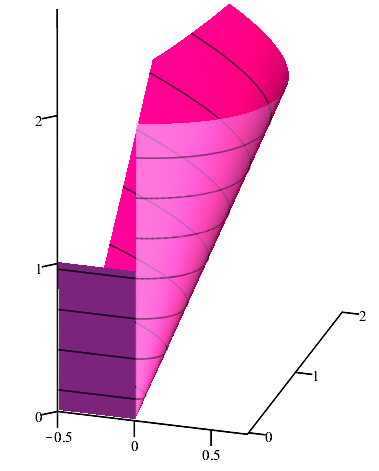}};
		\begin{scope}[x={(image.south east)},y={(image.north west)}]
		
		\node[black] at (0.65,0.025) {$x$};
		
		\draw[->,black] (0,0.29) -- (0.085,0.29);
		
		\node[black,left] at (0,0.29) {$\{(x,y,z) \;|\; x\leq 0,z\geq0,y=0 \}$};
		
		\node[black] at (0.9,0.37) {$y$};
		
		\draw[<-,black] (0.665,0.5) -- (0.75,0.5);
		
		\node[black,right] at (0.75,0.5) {$\{(x,y,z)\;|\; y>0, z \geq ye^{\frac{x}{y}} \} $};
		
		\node[black] at (0.06,0.9) {$z$};
		
		\end{scope}
		\end{tikzpicture}
	\end{center}
\caption{The exponential cone is the union of the two labelled sets}\label{fig:cone}
\end{figure}

\subsection{Outline and results}

The goal of this paper is to build a robust framework that may be used to obtain error bounds for previously inaccessible cones, and to demonstrate the use of this framework by applying it to fully describe error bounds for \eqref{eq:feas} with $\stdCone = \expCone$.

In Section~\ref{sec:preliminaries}, we recall preliminaries. New contributions begin in Section~\ref{sec:frf}. We first recall some rules for chains of faces and the diamond composition. Then we show how  error bounds may be constructed using objects known as {\oneFRF}. In Section~\ref{sec:frf_comp}, we build our general framework for constructing {\oneFRF}s. Our key result, Theorem~\ref{thm:1dfacesmain}, obviates the need of computing explicitly the projection onto the cone. 
Instead, we make use of the parametrization of the boundary of the cone and projections onto the proper \emph{faces} of a cone: thus, our approach is advantageous when these projections are easier to analyze than the projection onto the whole cone itself. We emphasize that \textit{all} of the results of Section~\ref{sec:frf} are applicable to a \textit{general closed convex cone}.

In Section~\ref{sec:exp_cone}, we use our new framework to fully describe error bounds for \eqref{eq:feas} with $\stdCone = \expCone$. This was previously a problem lacking a clear strategy, because all projections onto $\expCone$ are implicitly specified. However, having obviated the need to project onto $\expCone$, we successfully obtain all the necessary FRFs, partly because it is easier to project onto the \emph{proper faces} of $\expCone$ than to project onto $\expCone$ itself. Surprisingly, we discover that different collections of faces and exposing hyperplanes admit very different FRFs. In Section~\ref{sec:freas_2d}, we show that for the unique 2-dimensional face, any exponent in $\left(0,1\right)$ may be used to build a valid FRF, while the supremum over all the admissible exponents \textit{cannot} be.
Furthermore, a better FRF for the 2D face can be obtained  if we go beyond H\"olderian error bounds and consider a so-called \emph{entropic error bound} which uses a modified Boltzmann-Shannon entropy function, see Theorem~\ref{thm:entropic}.
The curious discoveries continue; for infinitely many 1-dimensional faces, the FRF, and the final error bound, feature exponent $1/2$. For the final outstanding 1-dimensional exposed face, the FRF, and the final error bound, are Lipschitzian for all exposing hyperplanes except exactly one, for which \textit{no exponent} will suffice. However, for this exceptional case, our framework \textit{still} successfully finds an FRF, which is logarithmic in character (Corollary~\ref{col:1dfaces_infty}). Consequentially, the system consisting of $\{(0,0,1)\}^\perp$ and $\expCone$ possesses a kind of ``logarithmic error bound" (see Example~\ref{ex:non_hold}) instead of a H\"olderian error bound. In Theorems~\ref{thm:main_err} and \ref{theo:sane}, we give explicit error bounds by using our FRFs and the  suite of tools we developed in Section~\ref{sec:frf}. We also show that the error bound given in Theorem~\ref{thm:main_err} is tight, see Remark~\ref{rem:opt}.

These findings about the exponential cone are surprising,  since we are not aware of other objects having this litany of odd behaviour hidden in its structure all at once.\footnote{To be fair, the exponential function is the classical example of a non-semialgebraic analytic function. Given that semialgebraicity is connected to the KL property (which is related to error bounds), one may argue that it is not \emph{that} surprising that the exponential cone has its share of quirks. Nevertheless, given how natural the exponential cone is, the  amount of quirks   is still somewhat surprising.}
One possible reason for the absence of previous reports on these phenomena might have been the sheer absence of tools for obtaining error bounds for general cones.
In this sense, we believe that the machinery developed in Section~\ref{sec:frf} might be a reasonable first step towards filling this gap. 
In Section~\ref{sec:odd}, we document additional odd consequences and connections to other concepts, with particular relevance to the \emph{Kurdyka-{\L}ojasiewicz (KL) property} \cite{BDL07,BDLS07,ABS13,ABRS10,BNPS17,LP18}. In particular, we have two sets satisfying a H\"olderian error bound for every $\gamma \in \left(0,1\right)$ but the supremum of allowable exponents is not allowable. Consequently, one obtains a function with the $KL$ property with exponent $\alpha$ for any $\alpha \in \left(1/2,1\right)$ at the origin, but not for $\alpha = 1/2$. We conclude in Section~\ref{sec:conclusion}.

\section{Preliminaries}\label{sec:preliminaries}
We recall that $\ambSpace$ denotes an arbitrary
finite-dimensional real Euclidean space. We will adopt the following convention, vectors will be boldfaced while scalars will use normal typeface. For example, if $\pp \in \RR^3$, we write $\pp = (p_x,p_y,p_z)$, where
$p_x,p_y,p_z \in \RR$.

We denote by $B(\eta)$ the closed ball of radius $\eta$ centered at the origin, i.e., $B(\eta) = \{\xx\in\ambSpace \mid \norm{\xx} \leq \eta\}$.
Let $C\subseteq \ambSpace$ be a convex
set. We denote the relative interior and the linear span of $C$ by $\reInt C$ and $\spanVec C$, respectively. We also denote the boundary of $C$ by $\partial C$, and $\mathrm{cl}\, C$ is the closure of $C$.
We denote the projection operator onto $C$ by $P_C$, so that
$P_C(\xx) = \argmin_{\yy \in C} \norm{\xx-\yy}$.
Given closed convex sets $C_1,C_2 \subseteq \ambSpace$, we note the following properties of the projection operator
\begin{align}
\dist(\xx,C_1) &\leq \dist(\xx,C_2) + \dist(P_{C_2}(\xx),C_1) \label{proj:p1}\\
\dist(P_{C_2}(\xx),C_1) &\leq \dist(\xx,C_2) + \dist(\xx,C_1). \label{proj:p2}
\end{align}

\subsection{Cones and their faces}
Let $\stdCone$ be a closed convex cone. We say that $\stdCone$ is \emph{pointed} if $\stdCone \cap - \stdCone = \{\mathbf{0}\}$. The dimension of $\stdCone$ is denoted by $\dim(\stdCone)$ and is the dimension of the linear subspace spanned by $\stdCone$.
A \emph{face} of $\stdCone$ is a closed convex cone $\stdFace$ satisfying $\stdFace \subseteq \stdCone$ and the following property
\[
\xx,\yy \in \stdCone, \xx+\yy \in \stdFace \Rightarrow \xx,\yy \in \stdFace.
\]
In this case, we write $\stdFace \face \stdCone$.
We say that $\stdFace$ is \emph{proper} if $\stdFace \neq \stdCone$.
A face is said to be \emph{nontrivial} if $\stdFace \neq \stdCone$ and  $\stdFace \neq \stdCone \cap - \stdCone$. In particular, if $\stdCone$ is pointed (as is the case of the exponential cone), a nontrivial face is neither $\stdCone$ nor $\{\mathbf{0}\}$.
Next, let $\stdCone^*$ denote the dual cone of $\stdCone$, i.e., $\stdCone^* = \{\zz \in \ambSpace \mid \inProd{\xx}{\zz} \geq 0, \forall \xx \in \stdCone \}$. We say that $\stdFace$ is an \emph{exposed face} if there exists $\zz \in \stdCone^*$ such that
$\stdFace = \stdCone \cap \{\zz\}^\perp$.

A \emph{chain of faces} of $\stdCone$ is a sequence of faces satisfying $\stdFace_\ell \subsetneq \cdots \subsetneq \stdFace_{1}$ such that each $\stdFace_{i}$ is a face of $\stdCone$ and the inclusions $\stdFace _{i+1} \subsetneq \stdFace_{i}$ are all proper. The length of the chain is defined to be $\ell$.
With that, we define the \emph{distance to polyhedrality of $\stdCone$} as the length {\em minus one} of the longest chain of faces of $\stdCone$ such that $\stdFace _{\ell}$ is polyhedral and $\stdFace_{i}$ is not polyhedral for $i < \ell$, see \cite[Section~5.1]{LMT18}. We denote the distance to polyhedrality by $\distP(\stdCone)$.

\subsection{Lipschitzian and H\"olderian error bounds}
In this subsection, suppose that $C_1,\ldots, C_{\ell} \subseteq \ambSpace$ are convex sets with nonempty intersection.
We recall the following definitions.
\begin{definition}[H\"olderian and Lipschitzian error bounds]\label{def:hold}
	The sets $C_1,\ldots, C_\ell$ are said to satisfy a \emph{H\"olderian error bound} if for every bounded set $B \subseteq \ambSpace$ there exist some $\kappa_B > 0$ and an exponent $\gamma _B\in(0, 1]$ such that
	\begin{equation*}
	\dist(\xx, \cap _{i=1}^\ell C_i) \le \kappa_B\max_{1\le i\le \ell}\dist(\xx, \, C_i)^{\gamma_B}, \qquad \forall\ \xx\in B.
	\end{equation*}
	If we can take the same $\gamma _B = \gamma \in (0,1]$ for all $B$, then we say
	that the bound is \emph{uniform}. If the bound is uniform with $\gamma = 1$, we call it a \emph{Lipschitzian error bound}.
\end{definition}
We note that the concepts in Definition~\ref{def:hold} also have different names throughout the literature. When $C_1,\ldots, C_\ell$ satisfy a H\"olderian error bound it is said that they satisfy \emph{bounded H\"older regularity}, e.g., see \cite[Definition~2.2]{BLT17}. When a Lipschitzian error bound holds, $C_1,\ldots, C_\ell$ are said to satisfy \emph{bounded linear regularity}, see \cite[Section~5]{BB96} or \cite{BBL99}. Bounded linear regularity is also closely related to the notion of \emph{subtransversality} \cite[Definition~7.5]{Io17}.

H\"olderian and Lipschitzian error bounds will appear frequently in our results, but we also encounter non-H\"olderian bounds as in Theorem~\ref{thm:entropic} and Theorem~\ref{thm:nonzerogammasec72}.
Next, we recall the following result which ensures
a Lipschitzian error bound  holds between families of
convex sets when a constraint qualification is satisfied.

\begin{proposition}[\!{\cite[Corollary~3]{BBL99}}]\label{prop:cq_er}
Let $C_1,\ldots, C_{\ell} \subseteq \ambSpace$ be convex sets such that
$C_{1},\ldots, C_{k}$ are polyhedral. If
\begin{equation*}
\left(\bigcap _{i=1}^k C_i\right) \bigcap \left(\bigcap _{j=k+1}^\ell \reInt C_j\right) \neq \emptyset,
\end{equation*}
then for every bounded set $B$ there exists $\kappa _B>0$ such that
\[
\dist(\xx, \cap _{i=1}^\ell C_i) \leq \kappa _B(\max_{1 \leq i \leq \ell} \dist(\xx, C_i)), \qquad \forall \xx \in B.
\]
\end{proposition}

In view of \eqref{eq:feas}, we say that \emph{Slater's condition} is satisfied
if $(\stdSpace + \abd) \cap \reInt\stdCone \neq \emptyset$.
If $\stdCone$ can be written as $\stdCone^1 \times \stdCone^2\subseteq \ambSpace^1 \times \ambSpace^2$, where $\ambSpace^1$ and $\ambSpace^2$ are real Euclidean spaces and
$\stdCone^1\subseteq \ambSpace^1$ is polyhedral, we say that the \emph{partial polyhedral Slater's (PPS) condition} is satisfied if
\begin{equation}\label{eq:ppsc}
(\stdSpace + \abd) \cap (\stdCone^1 \times (\reInt\stdCone^2) ) \neq \emptyset.
\end{equation}
Adding a dummy coordinate, if necessary, we can see Slater's condition as a particular case of the PPS condition.
By convention, we consider that the PPS condition is satisfied for \eqref{eq:feas} if one of the following is satisfied:
 1) $\stdSpace+\abd$ intersects $\reInt\stdCone$; 2) $(\stdSpace+\abd)\cap \stdCone\neq \emptyset$ and $\stdCone$ is polyhedral; or 3) $\stdCone$ can be written as a direct product $\stdCone^1 \times \stdCone^2$   where $\stdCone^1$ is polyhedral and \eqref{eq:ppsc} is satisfied.

Noting that $(\stdSpace + \abd) \cap (\stdCone^1 \times (\reInt\stdCone^2) ) = (\stdSpace + \abd)\cap (\stdCone^1 \times \ambSpace^2) \cap (\ambSpace^1 \times (\reInt\stdCone^2) )$, we deduce the following result from Proposition~\ref{prop:cq_er}.
\begin{proposition}[Error bound under PPS condition]\label{prop:pps_er}
Suppose that \eqref{eq:feas} satisfies the 	\emph{partial polyhedral Slater's condition}. Then, for every bounded set $B$ there exists $\kappa _B>0$ such that
\[
\dist(\xx, (\stdSpace+\abd)\cap \stdCone) \leq \kappa _B \max\{\dist(\xx,\stdCone),\dist(\xx,\stdSpace+\abd)\}, \qquad \forall \xx \in B.
\]
\end{proposition}
We recall that for $a,b \in \RR_+$ we have $a+b \leq 2\max\{a,b\} \leq 2(a+b)$, so Propositions~\ref{prop:cq_er} and \ref{prop:pps_er} can also be equivalently stated in terms of sums of distances.

\section{Facial residual functions and error bounds}\label{sec:frf}
In this section, we discuss a strategy for obtaining error
bounds for the conic linear system \eqref{eq:feas} based on
the so-called \emph{facial residual functions} that
were introduced in \cite{L17}. In contrast to \cite{L17}, we will not require that $\stdCone$ be amenable.

The motivation for our approach is as follows.
If it were the case that \eqref{eq:feas} satisfies some constraint qualification, we would have a Lipschitizian error bound per Proposition~\ref{prop:pps_er}, see also \cite{BBL99} for other sufficient conditions. Unfortunately, this does not happen in general. However, as long as
\eqref{eq:feas} is feasible, there is always a face of $\stdCone$ that contains the feasible region of \eqref{eq:feas} and for which a constraint qualification holds. The error bound computation essentially boils down to understanding how to compute the distance to this special face.
The first result towards our goal is the following.

\begin{proposition}[An error bound when a face satisfying a CQ is known]\label{prop:err_cq2}
	Suppose that \eqref{eq:feas} is feasible and let $\stdFace \face \stdCone$ be a face such that
	\begin{enumerate}[{\rm (a)}]
		\item $\stdFace$ contains $\stdCone \cap (\stdSpace+\abd)$.
		\item $\{\stdFace, \stdSpace+\abd\}$ satisfies the PPS condition.\footnote{As a reminder, the PPS condition is, by convention, a shorthand for three closely related conditions, see remarks after \eqref{eq:ppsc}.}
	\end{enumerate}
	Then, for every bounded set $B$, there exists $\kappa_B > 0$ such that
	\[
	\dist(\xx, \stdCone \cap (\stdSpace + \abd)) \leq \kappa_B(\dist(\xx, \stdFace) + \dist(\xx, \stdSpace+\abd)), \qquad \forall \xx \in B.
	\]
\end{proposition}
\begin{proof}
	Since $\stdFace$ is a face of $\stdCone$, assumption (a) implies
	$\stdCone \cap (\stdSpace + \abd) = \stdFace \cap (\stdSpace+\abd)$. Then,
	the result follows from  assumption (b) and Proposition~\ref{prop:pps_er}.
\end{proof}
From Proposition~\ref{prop:err_cq2} we see that the key to obtaining
an error bound for the system \eqref{eq:feas} is to find a face $\stdFace \face \stdCone$ satisfying (a), (b) \emph{and} we must know how to
estimate the quantity $\dist(\xx, \stdFace)$ from the available information
$\dist(\xx, \stdCone)$ and $\dist(\xx, \stdSpace+\abd)$.

This is where we will make use of facial reduction and facial residual functions. The former will help us find $\stdFace$ and the latter will be instrumental in upper bounding $\dist(\xx, \stdFace)$.
First, we recall below a result that follows from the analysis of the FRA-poly facial reduction algorithm developed in \cite{LMT18}.
\begin{proposition}[{\cite[Proposition~5]{L17}}\footnote{Although {\cite[Proposition~5]{L17}} was originally stated for pointed cones, it holds for general closed convex cones. Indeed, its proof only relies on {\cite[Proposition~8]{LMT18}}, which holds for general closed convex cones.}]\label{prop:fra_poly}
	Let $\stdCone = \stdCone^1\times \cdots \times \stdCone^s$, where
	each $\stdCone^i$ is a closed convex cone. Suppose \eqref{eq:feas} is feasible.
	Then there is a chain of faces
	\begin{equation}\label{eq:chain}
	\stdFace _{\ell}  \subsetneq \cdots \subsetneq \stdFace_1 = \stdCone
	\end{equation}
	of length $\ell$ and vectors $\{\zz_1,\ldots, \zz_{\ell-1}\}$ satisfying the following properties.
	\begin{enumerate}[{\rm (i)}]
		\item \label{prop:fra_poly:1} $\ell -1\leq  \sum _{i=1}^{s} \distP(\stdCone ^i)  \leq \dim{\stdCone}$.
		\item \label{prop:fra_poly:2} For all $i \in \{1,\ldots, \ell -1\}$, we have
		\begin{flalign*}%
		\zz_i \in \stdFace _i^* \cap \stdSpace^\perp \cap \{\abd\}^\perp \ \ \ {and}\ \ \
		\stdFace _{i+1} = \stdFace _{i} \cap \{\zz_i\}^\perp.
		\end{flalign*}		
		\item \label{prop:fra_poly:3} $\stdFace _{\ell} \cap (\stdSpace+\abd) = \stdCone \cap (\stdSpace + \abd)$ and 	$\{\stdFace _{\ell},\stdSpace+\abd\}$ satisfies the PPS condition.
	\end{enumerate}
\end{proposition}
In view of Proposition~\ref{prop:fra_poly}, we define the \emph{distance to the PPS condition} $\dpp(\stdCone,\stdSpace+\abd)$ as the length \emph{minus one} of the shortest chain of faces (as in \eqref{eq:chain}) satisfying item (iii) in Proposition~\ref{prop:fra_poly}. For example, if \eqref{eq:feas} satisfies the PPS condition, we have
$\dpp(\stdCone,\stdSpace+\abd) = 0$.

Next, we recall the definition of facial residual functions from \cite[Definition~16]{L17}.

\begin{definition}[Facial residual function\footnote{The only difference between Definition~\ref{def:ebtp} and the definition of facial residual functions in {\cite[Definition~16]{L17}} is that we added the ``with respect to $\stdCone$'' part, to emphasize the dependency on $\stdCone$.}]\label{def:ebtp}
	Let $\stdCone$ be a closed convex cone, $\stdFace \face \stdCone$ be a face, and let $\zz \in \stdFace^*$.
	Suppose that $\psi_{\stdFace,\zz} : \RP \times \RP \to \RP$ satisfies the following properties:
	\begin{enumerate}[label=({\roman*})]
		\item $\psi_{\stdFace,\zz}$ is nonnegative, monotone nondecreasing in each argument and $\psi_{\stdFace,\zz}(0,t) = 0$ for every $t \in \RP$.
		\item The following implication holds for any $\xx \in \spanVec \stdCone$ and any $\epsilon \geq 0$:
\[
\dist(\xx,\stdCone) \leq \epsilon, \quad \inProd{\xx}{\zz} \leq \epsilon, \quad \dist(\xx, \spanVec \stdFace ) \leq \epsilon \quad \Rightarrow \quad \dist(\xx,  \stdFace \cap \{\zz\}^{\perp})  \leq \psi_{\stdFace,\zz} (\epsilon, \norm{\xx}).
\]
	\end{enumerate}
	Then, $\psi_{\stdFace,\zz}$ is said to be \emph{a facial residual function for $\stdFace$ and $\zz$ with respect to $\stdCone$}.
\end{definition}
Definition~\ref{def:ebtp}, in its most general form, represents ``two-steps'' along the facial structure of a cone: we have a cone $\stdCone$, a face $\stdFace$ (which could be different from $\stdCone$) and a third face defined by $\stdFace \cap \{\zz\}^\perp$.
In this work, however, we  will be focused on the following special case of Definition~\ref{def:ebtp}.
\begin{definition}[{\OneFRF} ({\oneFRFs})]\label{def:onefrf}
Let $\stdCone$ be a closed convex cone and $\zz \in \stdCone^*$.
A function $\psi _{\stdCone,\zz}:\RP \times \RP \to \RP$ is called a \emph{{\oneFRF} (\oneFRFs) for $\stdCone$ and $\zz$} if it is a facial residual function of $\stdCone$ and $\zz$ with respect to $\stdCone$. That is, $\psi _{\stdCone,\zz}$ satisfies item (i) of Definition~\ref{def:ebtp} and for every $\xx \in \spanVec \stdCone$ and any $\epsilon \geq 0$:
\[
\dist(\xx,\stdCone) \leq \epsilon, \quad \inProd{\xx}{\zz} \leq \epsilon \quad \Rightarrow \quad \dist(\xx,  \stdCone \cap \{\zz\}^{\perp})  \leq \psi_{\stdCone,\zz} (\epsilon, \norm{\xx}).
\]

\end{definition}

\begin{remark}[Concerning the implication in {Definition~\ref{def:onefrf}}]
In view of the monotonicity of $\psi_{\stdCone,\zz}$, the implication in Definition~\ref{def:onefrf} can be equivalently and more succinctly written as
  \[
  \dist(\xx, \stdCone\cap\{\zz\}^\perp) \le \psi_{\stdCone,\zz}(\max\{\dist(\xx,\stdCone), \inProd{\xx}{\zz}\}, \norm{\xx}),\ \ \forall \xx \in \spanVec \stdCone.
  \]
  The {\em unfolded} form presented in Definition~\ref{def:onefrf} is more handy in our discussions and analysis below.
\end{remark}

Facial residual functions always exist (see \cite[Section~3.2]{L17}), but their computation is often nontrivial. Next, we review a few examples.
\begin{example}[Examples of facial residual functions]\label{ex:frf}
If $\stdCone$ is a symmetric cone (i.e., a self-dual homogeneous cone, see \cite{FK94,FB08}), then given $\stdFace \face \stdCone$ and  $\zz \in \stdFace^*$, there exists a $\kappa > 0$ such that $\psi _{\stdFace,\zz}(\epsilon,t) \coloneqq \kappa \epsilon + \kappa \sqrt{\epsilon t}$ is
a {\oneFRF} for $\stdFace$ and $\zz$, see \cite[Theorem~35]{L17}.

If $\stdCone$ is a polyhedral cone, the
function $\psi _{\stdFace,\zz}(\epsilon,t) \coloneqq \kappa \epsilon$ can be taken instead, with no dependency on $t$, see \cite[Proposition~18]{L17}.
\end{example}

Moving on, we say that
a function $\tilde \psi_{\stdFace,\zz}$ is a \emph{positively rescaled shift of $\psi_{\stdFace,\zz}$} if there are positive constants $M_1,M_2,M_3$ and nonnegative constant $M_4$ such that
\begin{equation}\label{eq:pos_rescale}
\tilde \psi _{\stdFace,\zz}(\epsilon,t) = M_3\psi_{\stdFace,\zz} (M_1\epsilon,M_2t) + M_4\epsilon.
\end{equation}
This is a generalization of the notion of positive rescaling in \cite{L17}, which sets $M_4 = 0$.
We also need to compose facial residual functions in a special manner.
Let $f:\RP\times \RP \to \RP$ and $g:\RP\times \RP \to \RP$ be functions.
We define the \emph{diamond composition} $f\comp g$ to be the function satisfying
\begin{equation}\label{eq:comp}
(f\comp g)(a,b) = f(a+g(a,b),b), \qquad \forall a,b \in \RP.
\end{equation}
Note that the above composition is not associative in general. When we have functions $f_i:\RP\times \RP \to \RP$, $i = 1,\ldots,m$ with $m\ge 3$, we define $f_m\comp \cdots \comp f_1$ inductively as the function $\varphi _m$ such that
\begin{align*}
\varphi_i &\coloneqq f_i \comp \varphi_{i-1},\qquad i \in \{2,\ldots,m\}\\
\varphi_1 &\coloneqq f_1.
\end{align*}
With that, we have $
f_m\comp f_{m-1} \comp \cdots \comp f_2 \comp f_1 \coloneqq f_m\comp (f_{m-1}\comp(\cdots \comp (f_2\comp f_1)))$.

The following lemma, which holds for a general closed convex cone $\stdCone$, shows how (positively rescaled shifts of) {\oneFRF s} for the \emph{faces of $\stdCone$} can be combined via the diamond composition to derive useful bounds on the distance to faces. A version of it was proved in {\cite[Lemma~22]{L17}}, which required the cones to be pointed and made use of general (i.e., not necessarily one-step) facial residual functions with respect to $\stdCone$. This is a subtle, but very crucial difference which will allows us to relax the assumptions in  \cite{L17}. 
\begin{lemma}[Diamond composing facial residual functions]\label{lem:chain}
	Suppose \eqref{eq:feas} is feasible and let
	\[
	\stdFace _{\ell}  \subsetneq \cdots \subsetneq \stdFace_1 = \stdCone
	\]
	be a chain of faces of $\stdCone$ together with $\zz_i \in \stdFace _i^*\cap\stdSpace^\perp \cap \{\abd\}^\perp$ such that
	$\stdFace_{i+1} = \stdFace _i\cap \{\zz_i\}^\perp$, for $i = 1,\ldots, \ell - 1$.
	For each $i$, let $\psi _{i}$ be a {\oneFRFs} for $\stdFace_i$ and $\zz_i$.
	Then, there is a positively rescaled shift of $\psi_i$ (still denoted as $\psi_i$ by an abuse of notation) so that
	for every $\xx \in \ambSpace$ and $\epsilon \geq 0$:
	\[
	 \quad \dist(\xx,\stdCone) \leq \epsilon, \quad \dist(\xx,\stdSpace + \abd) \leq \epsilon \quad \Rightarrow\quad
	\dist(\xx,  \stdFace _{\ell})  \leq
	\varphi (\epsilon,\norm{\xx}),
	\]
	where $\varphi = \psi _{{\ell-1}}\comp \cdots \comp \psi_{{1}}$, if $\ell \geq 2$. If $\ell = 1$, we let $\varphi$ be the function satisfying $\varphi(\epsilon, t) = \epsilon$.
\end{lemma}
\begin{proof}
For $\ell = 1$, we have $\stdFace_{\ell} = \stdCone$, so the lemma follows immediately. Now, we consider the case $\ell \geq 2$.
 First we note that $\stdSpace + \abd$ is contained in all the $\{\zz_{i}\}^\perp$ for $i = 1, \ldots, \ell-1$. Since the distance of $\xx \in \ambSpace$ to $\{\zz_{i}\}^\perp$ is given by $\frac{|\inProd{\xx}{\zz_i}|}{\norm{\zz_i}}$, we have the following chain of implications
\begin{equation}\label{eq:hyper}
\dist(\xx,\stdSpace + \abd) \leq \epsilon\quad \Rightarrow\quad \dist(\xx,\{\zz_i\}^\perp) \leq \epsilon \quad \Rightarrow\quad  \inProd{\xx}{\zz_i} \leq  \epsilon \norm{\zz_i}.
\end{equation}
Next, we proceed by induction. If $\ell = 2$, we
have that $\psi_1$ is a  {\oneFRF} for
$\stdCone$ and $\zz_1$. By
Definition~\ref{def:onefrf}, we have
\[
\yy \in \spanVec \stdCone, \quad \dist(\yy,\stdCone) \leq \epsilon, \quad \inProd{\yy}{\zz_1} \leq \epsilon \quad \Rightarrow \quad \dist(\yy,\stdFace_{2})  \leq
\psi_{1} (\epsilon, \norm{\yy}).
\]
In view of \eqref{eq:hyper} and the monotonicity of $\psi_1$, we see further that
\begin{equation}\label{eq:ind_base}
\yy \in \spanVec \stdCone, \, \dist(\yy ,\stdCone) \leq \epsilon, \, \dist(\yy ,\stdSpace + \abd) \leq \epsilon \, \Rightarrow \, \dist(\yy ,\stdFace_{2})  \leq
\psi_{1} (\epsilon(1+\norm{\zz_{1}}), \norm{\yy}).
\end{equation}
Now, suppose that $\xx \in \ambSpace$ and
$\epsilon \geq 0$
are such that $\dist(\xx ,\stdCone) \leq \epsilon$ and $\dist(\xx ,\stdSpace + \abd) \leq \epsilon$.
Let $\hat\xx := P_{\spanVec\stdCone}(\xx)$.
Since $\stdCone \subseteq \spanVec \stdCone$,
we have $\dist(\xx, \spanVec \stdCone)\leq \dist(\xx, \stdCone)$ and, in view of \eqref{proj:p2}, we have that
\begin{equation}\label{eq:ind_base2}
\begin{aligned}
\dist(\hat\xx,\stdCone) &\leq \dist(\xx,\spanVec\stdCone) + \dist(\xx,\stdCone)\le  2\epsilon,\\
\dist(\hat \xx,\stdSpace + \abd) &\leq \dist(\xx,\spanVec\stdCone) + \dist(\xx,\stdSpace + \abd)\le 2\epsilon.
\end{aligned}
\end{equation}
From \eqref{proj:p1}, \eqref{eq:ind_base} and
\eqref{eq:ind_base2} we obtain
\begin{equation}\notag
\begin{aligned}
\dist(\xx, \stdFace_2) \leq \dist(\xx, \spanVec \stdCone) + \dist(\hat \xx, \stdFace_{2}) &\leq \epsilon + \psi_{{1}}(2\epsilon(1+\norm{\zz_{1}}),\norm{\hat \xx})\\
&\leq \epsilon + \psi_{{1}}(2\epsilon(1+\norm{\zz_{1}}),\norm{ \xx}),\\
\end{aligned}
\end{equation}
where the last inequality follows from the monotonicity of $\psi_{{1}}$ and the fact
that $\norm{\hat \xx} \leq \norm{\xx}$.
This proves the lemma for chains of length $\ell = 2$ because the function mapping $(\epsilon,t)$ to $\epsilon + \psi_{{1}}(2\epsilon(1+\norm{\zz_{1}}),t)$ is a positively rescaled shift of $\psi_{{1}}$.

 Now, suppose that the lemma holds for chains of length $\hat \ell$ and consider a chain of length $\hat \ell + 1$. By the induction hypothesis, we have
\begin{equation}\label{eq:inductive}
\dist(\xx,\stdCone) \leq \epsilon, \quad \dist(\xx,\stdSpace + \abd) \leq \epsilon \quad \Rightarrow \quad \dist(\xx,\stdFace_{\hat \ell})  \leq
\varphi (\epsilon, \norm{\xx}),
\end{equation}
where $\varphi = \psi _{{\hat\ell-1}}\comp \cdots \comp \psi_1$ and the $\psi_i$ are (positively rescaled shifts of) \oneFRF s.
By the definition of $\psi _{{\hat\ell}}$ as a {\oneFRF} and using \eqref{eq:hyper}, we may positively rescale $\psi_{{\hat\ell}}$ (still denoted as $\psi_{{\hat\ell}}$ by an abuse of notation) so that for $\yy \in \spanVec \stdFace_{\hat\ell}$
and $\hat \epsilon \ge 0$, the following implication holds:
\begin{equation}\label{eq:eps_hat_ell}
\dist(\yy,\stdFace_{\hat\ell}) \leq \hat\epsilon, \quad \dist(\yy,\stdSpace + \abd) \leq \hat\epsilon \quad \Rightarrow \quad \dist(\yy,\stdFace_{\hat \ell+1})  \leq
\psi_{\hat\ell}(\hat \epsilon, \norm{\yy}).
\end{equation}
Now, suppose that $\xx \in \ambSpace$ and $\epsilon \ge 0$ satisfy $\dist(\xx,\stdCone) \leq \epsilon$ and $\dist(\xx,\stdSpace + \abd) \leq \epsilon$. Let $\hat\xx := P_{\spanVec\stdFace_{\hat \ell}}(\xx)$.
As before, since $\stdFace_{\hat\ell} \subseteq \spanVec \stdFace_{\hat\ell}$,
we have $\dist(\xx, \spanVec \stdFace_{\hat\ell})\leq \dist(\xx, \stdFace_{\hat\ell})$ and, in view of \eqref{proj:p2}, we have
\begin{equation}\label{eq:ind}
\begin{aligned}
  \dist(\hat\xx,\stdFace_{\hat\ell}) &\leq \dist(\xx,\spanVec\stdFace_{\hat\ell}) + \dist(\xx,\stdFace_{\hat\ell})\le  2\dist(\xx,\stdFace_{\hat \ell}) \le  2\dist(\xx,\stdFace_{\hat \ell}) + \epsilon,\\
  \dist(\hat \xx,\stdSpace + \abd) &\leq \dist(\xx,\spanVec\stdFace_{\hat\ell}) + \dist(\xx,\stdSpace + \abd)\le \dist(\xx,\stdFace_{\hat \ell})+\epsilon \le  2\dist(\xx,\stdFace_{\hat \ell}) + \epsilon.
\end{aligned}
\end{equation}
Let $\hat \psi_{\hat\ell}$ be such that
$\hat \psi_{\hat\ell}(s,t)\coloneqq s + \psi_{\hat\ell}(2s,t)$, so that $\hat \psi_{\hat\ell}$ is a positively rescaled shift of $\psi_{\hat\ell}$.
Then, \eqref{eq:ind}  together with \eqref{eq:eps_hat_ell} and \eqref{proj:p1} gives
\begin{align*}
\dist(\xx,\stdFace_{\hat \ell+1}) &\le \dist(\xx,\spanVec\stdFace_{\hat\ell}) + \dist(\hat\xx,\stdFace_{\hat \ell+1})\le \dist(\xx,\stdFace_{\hat\ell}) + \dist(\hat\xx,\stdFace_{\hat \ell+1})\\
 &\leq \dist(\xx,\stdFace_{\hat\ell}) + \psi_{\hat\ell}(\epsilon+2\dist(\xx,\stdFace_{\hat \ell}) , \norm{\hat\xx}) \overset{\rm (a)}\leq \dist(\xx,\stdFace_{\hat\ell}) + \psi_{\hat\ell}(2\epsilon+2\dist(\xx,\stdFace_{\hat \ell}) , \norm{\xx})\\
 & \le \hat \psi_{\hat\ell}(\epsilon+\dist(\xx,\stdFace_{\hat \ell}), \norm{\xx})\overset{\rm (b)}\leq\hat \psi_{\hat\ell}(\epsilon+\varphi(\epsilon, \norm{\xx}), \norm{\xx})=(\hat \psi_{\hat\ell} \comp \varphi)(\epsilon,\norm{\xx}),
\end{align*}
where (a) follows from the monotonicity of $\psi_{\hat\ell}$ and the fact that $\|\hat \xx\|\le \|\xx\|$, and (b) follows from \eqref{eq:inductive} and the monotonicity of $\hat\psi_{\hat\ell}$. This completes the proof.
\end{proof}

We now have all the pieces to state an error bound result for \eqref{eq:feas} that does not require any constraint qualifications.
\begin{theorem}[Error bound based on {\oneFRFs}s]\label{theo:err}
	Suppose \eqref{eq:feas} is feasible and let
	\[
	\stdFace _{\ell}  \subsetneq \cdots \subsetneq \stdFace_1 = \stdCone
	\]
	be a chain of faces of $\stdCone$  together with $\zz_i \in \stdFace _i^*\cap\stdSpace^\perp \cap \{\abd\}^\perp$ such that
	$\{\stdFace _{\ell}, \stdSpace+\abd\}$ satisfies the
	PPS condition and $\stdFace_{i+1} = \stdFace _i\cap \{\zz_i\}^\perp$ for every $i$.
	For $i = 1,\ldots, \ell - 1$, let $\psi _{i}$ be a {\oneFRFs} for $\stdFace_{i}$ and $\zz_i$.

	Then, there is a suitable positively rescaled shift of the $\psi _{i}$ (still denoted as $\psi_i$ by an abuse of notation) such that for any bounded set $B$ there is a  positive constant $\kappa_B$ (depending on $B, \stdSpace, \abd, \stdFace _{\ell}$) such that
	\[
	\xx \in B, \quad \dist(\xx,\stdCone) \leq \epsilon, \quad \dist(\xx,\stdSpace + \abd) \leq \epsilon\quad \Rightarrow \quad \dist\left(\xx, (\stdSpace + \abd) \cap \stdCone\right) \leq \kappa _B (\epsilon+\varphi(\epsilon,M)),
	\]
	where $M = \sup _{\xx \in B} \norm{\xx}$,
	$\varphi = \psi _{{\ell-1}}\comp \cdots \comp \psi_{{1}}$, if $\ell \geq 2$. If $\ell = 1$, we let $\varphi$ be the function satisfying $\varphi(\epsilon, M) = \epsilon$.
\end{theorem}
\begin{proof}
	The case $\ell = 1$ follows from Proposition \ref{prop:err_cq2}, by
	taking $\stdFace = \stdFace_1$.
	Now, suppose $\ell \geq 2$.
	We apply Lemma \ref{lem:chain}, which tells us that, after positively rescaling and shifting the $\psi_i$,
	we have:
	\[
	\dist(\xx,\stdCone) \leq \epsilon, \quad \dist(\xx,\stdSpace + \abd) \leq \epsilon \implies
	\dist(\xx, \stdFace _{\ell})  \leq
	\varphi(\epsilon,\norm{\xx}),
	\]
	where $\varphi = \psi _{{\ell-1}}\comp \cdots \comp \psi_{{1}} $.
	In particular, since $\norm{\xx} \leq M$ for $\xx \in B$ we have
	\begin{equation}\label{eq:aam_aux1}
	\quad \dist(\xx,\stdCone) \leq \epsilon, \quad \dist(\xx,\stdSpace + \abd) \leq \epsilon \implies
	\dist(\xx, \stdFace _{\ell})  \leq
	\varphi(\epsilon,M), \qquad \forall \xx \in B
	\end{equation}	
	By assumption, $\{\stdFace _{\ell}, \stdSpace+\abd\}$ satisfies the
	PPS condition. We invoke Proposition~\ref{prop:err_cq2} to find $ \kappa _B > 0$ such that
	\begin{equation}\label{eq:aam_aux2}
\dist(\xx, \stdCone \cap (\stdSpace + \abd)) \leq \kappa_B(\dist(\xx, \stdFace_{\ell}) + \dist(\xx, \stdSpace+\abd)), \qquad \forall \xx \in B.
	\end{equation}	
	Combining \eqref{eq:aam_aux1}, \eqref{eq:aam_aux2}, we
	conclude that if $\xx \in B$ and $\epsilon\ge 0$ satisfy $\dist(\xx,\stdCone) \leq \epsilon$ and $\dist(\xx,\stdSpace + \abd) \leq \epsilon$,
	then we have $\dist\left(\xx, (\stdSpace + \abd) \cap \stdCone\right) \leq \kappa_B(\epsilon+\varphi(\epsilon,M))$. This completes the proof.
\end{proof}
Theorem~\ref{theo:err} is an improvement over \cite[Theorem~23]{L17} because it removes the amenability assumption. Furthermore, it shows that it is enough to determine the \oneFRF s for $\stdCone$ and its faces, whereas \cite[Theorem~23]{L17}  may require all possible facial residual functions related to $\stdCone$ and its faces.
Nevertheless, Theorem~\ref{theo:err} is still an abstract error bound result; whether some concrete inequality can be written down depends on obtaining a formula for the $\varphi$ function. To do so, it would require finding expressions for the {\oneFRF}s.
In the next subsections, we will address this challenge.

\subsection{How to compute {\oneFRF}s?}\label{sec:frf_comp}
In this section, we present some general tools for computing {\oneFRF}s.
\begin{lemma}[{\oneFRFs} from error bound]\label{lem:facialresidualsbeta}
	Suppose that $\stdCone$ is a closed convex cone and let $\zz\in \stdCone^*$ be such that $\stdFace = \{\zz \}^\perp\cap \stdCone$ is a proper face of $\stdCone$.
	Let $\frakg:\RR_+\to\RR_+$ be monotone nondecreasing with $\frakg(0)=0$, and let $\kappa_{\zz,\fraks}$ be a finite monotone nondecreasing nonnegative function in $\fraks\in \RR_+$ such that
	\begin{equation}\label{assumption:q}
	\dist(\qq,\stdFace) \leq \kappa_{\zz,\norm{\qq}} \frakg(\dist(\qq,\stdCone))\ \ \mbox{whenever}\ \ \qq \in \{\zz\}^\perp.
	\end{equation}
	Define the function $\psi_{\stdCone,\zz}:\RP\times \RP\to \RP$ by
	\begin{equation*}
	\psi_{\stdCone,\zz}(s,t) := \max \left\{s,s/\|\zz\| \right\} + \kappa_{\zz,t}\frakg \left(s +\max \left\{s,s/\|\zz\| \right\} \right).
	\end{equation*}
	Then we have
	\begin{equation}\label{haha}
	\dist(\pp,\stdFace) \leq \psi_{\stdCone,\zz}(\epsilon,\norm{\pp}) \mbox{\ \ \ \  whenever\ \ \ \ $\dist(\pp,\stdCone) \leq \epsilon$\ \ and\ \ $\inProd{\pp}{\zz} \leq \epsilon$.}
	\end{equation}
	Moreover, $\psi_{\stdCone,\zz}$ is a {\oneFRFs} for $\stdCone$ and $\zz$.
\end{lemma}
\begin{proof}
	Suppose that $\dist(\pp,\stdCone) \leq \epsilon$ and $\inProd{\pp}{\zz} \leq \epsilon$. We first claim that
	\begin{equation}\label{dpzperp}
	\dist(\pp,\{\zz \}^\perp) \leq \max \left\{\epsilon,\epsilon/\|\zz\| \right\}.
	\end{equation}
	This can be shown as follows. Since $\zz \in \stdCone^*$, we have
	$\inProd{\pp+P_{\stdCone}(\pp)-\pp}{\zz} \geq 0$ and
	\[
	\inProd{\pp}{\zz} \geq - \inProd{P_{\stdCone}(\pp)-\pp}{\zz} \geq -\epsilon \norm{\zz}.
	\]
	We conclude that $|\langle \pp,\zz\rangle| \leq \max\{\epsilon \norm{\zz},\epsilon\}$.
	This, in combination with $\dist(\pp,\{\zz \}^\perp) = |\langle \pp,\zz\rangle|/\|\zz\|$, leads to \eqref{dpzperp}.

	
	Next, let $\qq:=P_{\{\zz \}^\perp}\pp$.
	Then we have that
	\begin{equation*}
	\begin{aligned}
	\dist(\pp,\stdFace) &\leq \|\pp - \qq\|+\dist(\qq,\stdFace)\overset{\rm (a)}\leq \max \left\{\epsilon,\epsilon/\|\zz\| \right\} + \dist(\qq,\stdFace)\\
	&\overset{\rm (b)}\leq \max \left\{\epsilon,\epsilon/\|\zz\| \right\} + \kappa_{\zz,\norm{\qq}} \frakg\left(\dist(\qq,\stdCone)\right)\\
	&\overset{\rm (c)}\leq \max \left\{\epsilon,\epsilon/\|\zz\| \right\} + \kappa_{\zz,\norm{\pp}} \frakg\left(\dist(\qq,\stdCone)\right)\\
	&\overset{\rm (d)}\leq \max \left\{\epsilon,\epsilon/\|\zz\| \right\} + \kappa_{\zz,\norm{\pp}} \frakg\left(\epsilon +\max \left\{\epsilon,\epsilon/\|\zz\| \right\} \right),
	\end{aligned}
	\end{equation*}
	where (a) follows from \eqref{dpzperp}, (b) is a consequence of \eqref{assumption:q}, (c) holds because $\|\qq\| = \|P_{\{\zz \}^\perp}\pp\| \leq \|\pp\|$ so that $\kappa_{\zz,\|\qq\|}\leq \kappa_{\zz,\|\pp\|}$, and (d) holds because $\frakg$ is monotone nondecreasing and
	$$
	\dist(\qq,\stdCone) \leq \dist(\pp,\stdCone) + \|\qq - \pp\| \leq \epsilon + \max \left\{\epsilon,\epsilon/\|\zz\| \right\};
	$$
	here, the second inequality follows from \eqref{dpzperp} and the assumption that $\dist(\pp,\stdCone)\le \epsilon$. This proves \eqref{haha}. Finally, notice that $\psi_{\stdCone,\zz}$ is nonnegative, monotone nondecreasing in each argument, and that $\psi_{\stdCone,\zz}(0,t)=0$ for every $t \in \RR_+$. Hence, $\psi_{\stdCone,\zz}$ is a {\oneFRF} for $\stdCone$ and $\zz$.
\end{proof}

\begin{figure}
	\begin{center}
		\begin{tikzpicture}[scale=6]
		\draw[->,gray] (0,0) -- (1.1,0) node[right] {$x$-axis};
		\draw[->,gray] (0,0) -- (0,1.0) node[above] {$y$-axis};	
		
		\draw[scale=1,domain=0:1,smooth,variable=\t,red] plot ({\t*exp(-\t)},{exp(-\t)});
		
		\draw[scale=1,domain=0.001:1,smooth,variable=\u,red] plot ({exp(-1/\u)/\u},{exp(-1/\u)});
		
		\node[right,red] at (0.08,0.94) {$K_{\exp} \cap \{(x,y,z)\;|\; z =1\}$};
		
		\draw[scale=1,domain=0.1613:0.35,smooth,variable=\u,black] plot ({\u},{1.213*\u+-0.1956});
		
		\draw[scale=1,domain=0.35:.45,smooth,variable=\u,black,dotted] plot ({\u},{1.213*\u+-0.1956});
		
		\draw[scale=1,domain=0.45:.6,smooth,variable=\u,black] plot ({\u},{1.213*\u+-0.1956});
		
		\draw[scale=1,domain=0.6:.8,smooth,variable=\u,black,dotted] plot ({\u},{1.213*\u+-0.1956});
		
		\draw[scale=1,domain=0.8:1,smooth,variable=\u,black] plot ({\u},{1.213*\u+-0.1956});
		
		\draw[scale=1,domain=0.0:0.5,smooth,variable=\u,purple,dashed] plot ({\u},{(1.61/2.94)*\u+0});
		
		\shade[shading=ball, ball color=purple] (.294,.161) circle (0.015) node  [above left,purple] {$\uu = P_{\Fhat}\ww\;$};

		\shade[shading=ball, ball color=black] (0.294,0.6239) circle (0.015) node  [left,black] {$\vv=P_{\stdCone}\qq\;$};
		
		\shade[shading=ball, ball color=black] (0.9699,0.9809) circle (0.015) node [right,black] {$\;\;\qq \in \{\zz\}^\perp\cap B(\eta)\backslash \stdFace$};
		
		
		\shade[shading=ball, ball color=black] (0.52137,0.43676) circle (0.015) node  [below right,black] {$\ww=P_{\{\mathbf{z} \}^\perp}\vv$};
		
		
		\draw[->,thick,blue] (0.915828,0.952340) -- (0.348072,0.652460);
		
		\draw[->,thick,blue] (0.339474,0.586472) -- (0.475896,0.474188);
		
		\draw[<-,thick,blue] (0.3271055 +0.025,0.2023640-0.025) -- (0.4872645+0.025,0.395390 -0.025);

		\end{tikzpicture}
	\end{center}
	\caption{Theorem~\ref{thm:1dfacesmain} shows that we may replace the problem of showing that $\frakg(\|\qq-P_{\stdCone}\qq\|)/\|\qq-P_{\stdFace}\qq \|$ is uniformly bounded away from zero for all $\qq \in \{\zz \}^\perp \cap B(\eta)\setminus \stdFace$, with the equivalent problem of showing that $\frakg(\|\vv-P_{\{\zz \}^\perp}\vv \|)/\|P_{\{\zz \}^\perp}\vv - P_{\stdFace}\circ P_{\{\zz \}^\perp}\vv\|$ is uniformly bounded away from zero for all $\vv \in \bd \stdCone \cap B(\eta)\setminus \stdFace$ with $P_{\{\zz \}^\perp}\vv \neq P_{\stdFace}\circ P_{\{\zz \}^\perp}\vv$. This second problem can sometimes be easier to deal with, because it obviates the need to project onto $\stdCone$ and projects onto the nontrivial exposed face $\stdFace$ instead. For describing a possibly higher dimensional problem in 2D, we represent $\{\zz\}^\perp$ with a line, $\mathcal{K}$ with a 2D slice, and $\Fhat$ with a dot; of course, this is an oversimplification, since $\qq,\ww,\uu$ are \textit{not} generically colinear, nor would any of the points necessarily lie in the same horizontal slice. The scenario shown is meant to suggest intuition, but it is not a plausible configuration of points.}\label{fig:uvw}	
\end{figure}

In view of Lemma~\ref{lem:facialresidualsbeta}, one may construct {\oneFRF}s after establishing the error bound \eqref{assumption:q}. In the next theorem, we present a characterization for the existence of such an error bound. Our result is based on the quantity \eqref{gammabetaeta} defined below being {\em nonzero}. Note that this quantity {\em does not} explicitly involve projections onto $\stdCone$; this enables us to work with the exponential cone later, whose projections do not seem to have simple expressions. Figure~\ref{fig:uvw} provides a geometric interpretation of \eqref{gammabetaeta}.

\begin{theorem}[Characterization of the existence of error bounds]\label{thm:1dfacesmain}
	Suppose that $\stdCone$ is a closed convex cone and let $\zz\in \stdCone^*$ be such that $\stdFace = \{\zz \}^\perp \cap \stdCone$ is a nontrivial exposed face of $\stdCone$.
	Let $\eta \ge 0$, $\alpha \in (0,1]$ and let $\frakg:\RR_+\to \RR_+$ be monotone nondecreasing with $\frakg(0) = 0$ and $\frakg \geq |\cdot|^\alpha$. Define
	\begin{equation}\label{gammabetaeta}
	\gamma_{\zz,\eta} := \inf_{\vv} \left\{\frac{\frakg(\|\ww-\vv\|)}{\|\ww-\uu\|}\;\bigg|\; \vv\in \bd \stdCone\cap B(\eta)\backslash\stdFace,\ \ww = P_{\{\zz \}^\perp}\vv,\  \uu = P_{\stdFace}\ww,\ \ww\neq \uu\right\}. 
	\end{equation}
	Then the following statements hold.
	\begin{enumerate}[{\rm (i)}]
		\item\label{thm:1dfacesmaini} If $\gamma_{\zz,\eta} \in (0,\infty]$, then it holds that
		\begin{equation}\label{haha2}
		\dist(\qq,\stdFace) \leq \kappa_{\zz,\eta} \frakg(\dist(\qq,\stdCone))\ \ \mbox{whenever\ $\qq \in \{\zz \}^\perp \cap B(\eta)$},
		\end{equation}
		where $\kappa_{\zz,\eta} := \max \left \{2\eta^{1-\alpha}, 2\gamma_{\zz,\eta}^{-1}   \right \} < \infty$.
		\item\label{thm:1dfacesmainii} If there exists $\kappa_{_B} \in (0,\infty)$ so that
		\begin{equation}\label{hahaww2}
		\dist(\qq,\stdFace) \leq \kappa_{_B} \frakg(\dist(\qq,\stdCone))\ \ \mbox{whenever\ $\qq \in \{\zz \}^\perp \cap B(\eta)$},
		\end{equation}
		then $\gamma_{\zz,\eta} \in (0,\infty]$.
	\end{enumerate}
	
\end{theorem}
\begin{proof}
	We first consider item (i).
	If $\eta = 0$ or $\qq \in \stdFace$, the result is vacuously true, so let $\eta > 0$ and $\qq \in \{\zz \}^\perp \cap B(\eta) \backslash \stdFace$. Then $\qq\notin \stdCone$ because $\stdFace = \{\zz\}^\perp\cap \stdCone$.
	Define
	\[
	\vv=P_{\stdCone}\qq,\quad \ww = P_{\{\zz \}^\perp}\vv,\quad \text{and}\quad \uu = P_{\stdFace}\ww.
	\]
	Then $\vv\in \bd \stdCone\cap B(\eta)$ because $\qq\notin \stdCone$ and $\|\qq\|\le \eta$. If $\vv\in \stdFace$, then we have $\dist(\qq,\stdFace) = \dist(\qq,\stdCone)$ and hence
	\[
	\dist(\qq,\stdFace) = \dist(\qq,\stdCone)^{1-\alpha}\cdot \dist(\qq,\stdCone)^\alpha\le \eta^{1-\alpha}\dist(\qq,\stdCone)^\alpha\le \kappa_{\zz,\eta} \frakg(\dist(\qq,\stdCone)),
	\]
	where the first inequality holds because $\norm{\qq}\le \eta$, and the last inequality follows from the definitions of $\frakg$ and $\kappa_{\zz,\eta}$.
	Thus, from now on, we assume that $\vv\in \bd \stdCone\cap B(\eta)\backslash\stdFace$.
	
	Next, since $\ww=P_{\{\zz\}^\perp}\vv$, it holds that $\vv-\ww\in \{\zz\}^{\perp\perp}$ and hence
	$\|\qq-\vv\|^2 = \|\qq-\ww\|^2+\|\ww-\vv\|^2$. In particular, we have
	\begin{equation}\label{cinequality_infty}
	\dist(\qq,\stdCone) = \|\qq-\vv\| \geq \max\{\|\vv-\ww\|,\|\qq-\ww\|\},
	\end{equation}
	where the equality follows from the definition of $\vv$.
	Now, to establish \eqref{haha2}, we consider two cases.
	\begin{enumerate}[(I)]
		\item\label{thMbetafacescase1_infty} $\dist(\qq,\stdFace) \leq 2\dist(\ww,\stdFace)$;
		\item\label{thMbetafacescase2_infty} $\dist(\qq,\stdFace) > 2\dist(\ww,\stdFace)$.
	\end{enumerate}
	\ref{thMbetafacescase1_infty}: In this case, we have from $\uu = P_{\stdFace}\ww$ and $\qq\notin \stdFace$ that
	\begin{equation}\label{haha3}
	2\|\ww - \uu\|= 2\dist(\ww,\stdFace) \ge \dist(\qq,\stdFace) > 0,
	\end{equation}
	where the first inequality follows from the assumption in this case \ref{thMbetafacescase1_infty}.
	Hence,
	\begin{equation}\label{biginequalitybetageneral_infty}
	\frac{1}{\kappa_{\zz,\eta}} \stackrel{\rm (a)}{\le} \frac{1}{2}\gamma_{\zz,\eta} \stackrel{\rm (b)}{\leq} \frac{\frakg(\|\ww-\vv\|)}{2\|\ww-\uu\|} \stackrel{\rm (c)}{\leq} \frac{\frakg(\dist(\qq,\stdCone))}{2\|\ww-\uu\|}\stackrel{\rm (d)}{\leq} \frac{\frakg(\dist(\qq,\stdCone))}{\dist(\qq,\stdFace)},
	\end{equation}
	where (a) is true by the definition of $\kappa_{\zz,\eta}$, (b) uses the condition that $\vv\in \bd \stdCone\cap B(\eta)\backslash\stdFace$, \eqref{haha3} and the definition of $\gamma_{\zz,\eta}$, (c) is true by \eqref{cinequality_infty} and the monotonicity of $\frakg$, and (d) follows from \eqref{haha3}. This concludes case \ref{thMbetafacescase1_infty}.\footnote{In particular, in view of \eqref{biginequalitybetageneral_infty}, we see that this case only happens when $\gamma_{\zz,\eta} < \infty$.}
	
	\noindent\ref{thMbetafacescase2_infty}: Using the triangle inequality, we have
	\begin{equation*}
	2\dist(\qq,\stdFace) \leq 2\|\qq-\ww\|+2\dist(\ww,\stdFace)< 2\|\qq-\ww\|+\dist(\qq,\stdFace),
	\end{equation*}
	where the strict inequality follows from the condition for this case \ref{thMbetafacescase2_infty}. Consequently, we have $\dist(\qq,\stdFace) \leq 2\|\qq-\ww\|$. Combining this with \eqref{cinequality_infty}, we deduce further that
	\[
	\begin{aligned}
	\dist(\qq,\stdFace) &\le 2\|\qq-\ww\|\le 2 \max\{\|\vv-\ww\|,\|\qq-\ww\|\} \le 2\dist(\qq,\stdCone) = 2\dist(\qq,\stdCone)^{1-\alpha}\cdot \dist(\qq,\stdCone)^\alpha\\
	& \le 2\eta^{1-\alpha}\dist(\qq,\stdCone)^\alpha\le \kappa_{\zz,\eta} \frakg(\dist(\qq,\stdCone)),
	\end{aligned}
	\]
	where the fourth inequality holds because $\norm{\qq}\le \eta$, and the last inequality follows from the definitions of $\frakg$ and $\kappa_{\zz,\eta}$. This proves item (i).
	
	We next consider item (ii). Again, the result is vacuously true if $\eta = 0$, so let $\eta > 0$. Let $\vv\in \bd \stdCone\cap B(\eta)\backslash\stdFace$, $\ww = P_{\{\zz \}^\perp}\vv$ and $\uu = P_{\stdFace}\ww$ with $\ww\neq \uu$. Then $\ww\in B(\eta)$, and we have in view of \eqref{hahaww2} that
	\[
	\|\ww-\uu\| \overset{\rm (a)}= \dist(\ww,\stdFace) \overset{\rm (b)}\le \kappa_{_B}\frakg(\dist(\ww,\stdCone)) \overset{\rm (c)}\le \kappa_{_B}\frakg(\|\ww - \vv\|),
	\]
	where (a) holds because $\uu = P_{\stdFace}\ww$, (b) holds because of \eqref{hahaww2}, $\ww\in \{\zz\}^\perp$ and $\|\ww\|\le \eta$, and (c) is true because $\frakg$ is monotone nondecreasing and $\vv\in \stdCone$. Thus, we have $\gamma_{\zz,\eta} \ge 1/\kappa_{_B} > 0$. This completes the proof.
\end{proof}
\begin{remark}[About $\kappa_{\zz,\eta}$ and $\gamma_{\zz,\eta}^{-1}$]\label{rem:kappa}
As $\eta$ increases, the infimum in \eqref{gammabetaeta} is taken over a larger region, so 	$\gamma_{\zz,\eta}$ does not increase.
Accordingly, $\gamma_{\zz,\eta}^{-1}$ does not decrease when $\eta$ increases. Therefore, the $\kappa_{\zz,\eta}$ and $\gamma_{\zz,\eta}^{-1}$ considered in Theorem~\ref{thm:1dfacesmain} are monotone nondecreasing as functions of $\eta$ when $\zz$ is fixed. We are also using the convention that $1/\infty = 0$ so that $\kappa_{\zz,\eta} = 2\eta^{1-\alpha}$ when $\gamma_{\zz,\eta} = \infty$.
\end{remark}

Thus, to establish an error bound as in \eqref{haha2}, it suffices to show that $\gamma_{\zz,\eta} \in (0,\infty]$ for the choice of $\frakg$ and $\eta\ge 0$. Clearly, $\gamma_{\zz,0} = \infty$. The next lemma allows us to check whether $\gamma_{\zz,\eta} \in (0,\infty]$ for an $\eta > 0$ by considering {\em convergent} sequences.
\begin{lemma}\label{lem:infratio}
	Suppose that $\stdCone$ is a closed convex cone and let $\zz\in \stdCone^*$ be such that $\stdFace = \{\zz \}^\perp \cap \stdCone$ is a nontrivial exposed face of $\stdCone$.
	Let $\eta > 0$, $\alpha \in (0,1]$ and let $\frakg:\RR_+\to \RR_+$ be monotone nondecreasing with $\frakg(0) = 0$ and $\frakg \geq |\cdot|^\alpha$. Let $\gamma_{\zz,\eta}$ be defined as in \eqref{gammabetaeta}. If $\gamma_{\zz,\eta} = 0$, then there exist
	$\bv \in \stdFace$ and a sequence $\{\vv^k\}\subset \bd \stdCone\cap B(\eta) \backslash \stdFace$ such that
	\begin{subequations}\label{infratiohk}
	\begin{align}
	\underset{k \rightarrow \infty}{\lim}\vv^k = \underset{k \rightarrow \infty}{\lim}\ww^k &= \bv \label{infratiohka} \\
	{\rm and}\ \ \ \lim_{k\rightarrow\infty} \frac{\frakg(\|\ww^k - \vv^k\|)}{\|\ww^k- \uu^k\|} &= 0, \label{infratiohkb}
	\end{align}
	\end{subequations}
	where $\ww^k = P_{\{\zz\}^\perp}\vv^k$, $\uu^k = P_{\stdFace}\ww^k$ and $\ww^k\neq \uu^k$.
\end{lemma}	
\begin{proof}
	Suppose that $\gamma_{\zz,\eta} = 0$. Then, by the definition of infimum, there exists a sequence $\{\vv^k\}\subset\bd \stdCone\cap B(\eta) \backslash \stdFace$ such that
	\begin{equation}\label{infratio1}
	\underset{k \rightarrow \infty}{\lim} \frac{\frakg(\|\ww^k - \vv^k\|)}{\|\ww^k - \uu^k\|} = 0,
	\end{equation}
	where $\ww^k = P_{\{\zz\}^\perp}\vv^k$, $\uu^k = P_{\stdFace}\ww^k$ and $\ww^k\neq \uu^k$.
	Since $\{\vv^k\}\subset B(\eta)$, by passing to a convergent subsequence if necessary, we may assume without loss of generality that
	\begin{equation}\label{infratio4}
	\underset{k \rightarrow \infty}{\lim}\vv^k = \bv
	\end{equation}
	for some $\bv\in \stdCone\cap B(\eta)$. In addition, since
	$0\in \stdFace\subseteq \{\zz\}^\perp$, and projections onto closed convex sets are nonexpansive, we see that $\{\ww^k\}\subset B(\eta)$ and $\{\uu^k\}\subset B(\eta)$,
	and hence the sequence $\{\|\ww^k-\uu^k\|\}$ is bounded. Then we can conclude from \eqref{infratio1} and the assumption $\frakg \geq |\cdot|^\alpha$ that
	\begin{equation}\label{infratio3}
	\underset{k \rightarrow \infty}{\lim}\|\ww^k-\vv^k\| =0.
	\end{equation}
	Now \eqref{infratio3}, \eqref{infratio4}, and the triangle inequality give $\ww^k\to \bv$. Since $\{\ww^k\}\subset \{\zz\}^\perp$, it then follows that $\bv \in \{\zz \}^\perp$. Thus, $\bv\in \{\zz \}^\perp \cap \stdCone=\stdFace$.
	This completes the proof.
\end{proof}

Let $\stdCone$ be a closed convex cone.
Lemma~\ref{lem:facialresidualsbeta}, Theorem~\ref{thm:1dfacesmain} and Lemma~\ref{lem:infratio} are tools to obtain {\oneFRF}s for $\stdCone$. These are exactly the kind of facial residual functions needed in the abstract error bound result, Theorem~\ref{theo:err}. We conclude this subsection with a result that connects the {\oneFRF}s of a product cone and those of its constituent cones, which is useful for deriving error bounds for product cones.

\begin{proposition}[{\oneFRFs} for products]\label{prop:frf_prod}
Let $\stdCone^i \subseteq \ambSpace^i$ be closed convex cones for every $i \in \{1,\ldots,m\}$ and let $\stdCone = \stdCone^1 \times \cdots \times \stdCone^m$.
Let
$\stdFace \face \stdCone$, $\zz \in \stdFace^*$ and suppose that
$\stdFace = \stdFace^1\times \cdots \times \stdFace^m$ with $\stdFace^i \face \stdCone^i$ for every $i \in \{1,\ldots,m\}$.
Write $\zz = (\zz_1,\ldots,\zz_m)$ with $\zz_i \in (\stdFace^i)^*$.

For every $i$, let $\psi_{\stdFace^i,\zz_i}$ be a {\oneFRFs} for $\stdFace^i$ and $\zz_i$.
Then, there exists a $\kappa > 0$ such that  the function $\psi _{\stdFace,\zz}$ satisfying
\[
\psi _{\stdFace,\zz}(\epsilon,t) = \sum _{i=1}^m \psi_{\stdFace^i,\zz_i}(\kappa\epsilon,t)
\]
is a {\oneFRFs} for $\stdFace$ and $\zz$.


\end{proposition}
\begin{proof}
Suppose that $\xx \in \spanVec \stdFace$ and $\epsilon\ge 0$ satisfy the inequalities
\begin{equation*}
\dist(\xx,\stdFace) \leq \epsilon, \quad \inProd{\xx}{\zz} \leq \epsilon.
\end{equation*}
We note that
\[
\stdFace\cap \{\zz\}^{\perp} = (\stdFace^{1} \cap\{\zz_1\}^\perp) \times \cdots \times (\stdFace^{m} \cap\{\zz_m\}^\perp),
\]
and that for every $i \in \{1,\ldots,m\}$,
\begin{align}\label{eq:f1}
\dist(\xx_i,\stdFace^i)\le \dist(\xx,\stdFace)\le \epsilon.
\end{align}
Since $\zz_i \in (\stdFace^i)^*$, we have from \eqref{eq:f1} that
\begin{equation}\label{eq:p_f1}
0 \leq \inProd{\zz_i}{P_{\stdFace^{i}}({\xx}_i)} = \inProd{\zz_i}{P_{\stdFace^{i}}({\xx}_i) - {\xx}_i + {\xx}_i } \leq \epsilon\norm{\zz_i} + \inProd{\zz_i}{{\xx}_i}.
\end{equation}
Using \eqref{eq:p_f1} for all $i$ and
recalling that $\inProd{\zz}{\xx}\leq \epsilon$, we have
\begin{align}
\inProd{(\zz_1,\ldots,\zz_m)}{(P_{\stdFace^{1}}({\xx}_1),\ldots,P_{\stdFace^{m}}({\xx}_m)) } \leq  \sum _{i=1}^m\left[\epsilon\norm{\zz_i} + \inProd{\zz_i}{\xx_i}\right] \leq \hat\kappa \epsilon, \label{eq:p_f}
\end{align}
where $\hat\kappa = 1+\sum_{i=1}^m\norm{\zz_i}$.
Since $\inProd{\zz_i}{P_{\stdFace^{i}}(\xx_i)} \geq 0$ for $i \in \{1,\ldots,m\}$, from \eqref{eq:p_f} we obtain
\begin{equation}\label{eq:z_pf}
\inProd{\zz_i}{P_{\stdFace^{i}}(\xx_i)} \leq \hat\kappa \epsilon, \qquad i \in \{1,\ldots,m\}.
\end{equation}
This implies that for $i \in\{1,\ldots,m\}$ we have
\begin{equation}\label{eq:revisef1}
\inProd{\zz_i}{\xx_i} = \inProd{\zz_i}{\xx_i -P_{\stdFace^{i}}(\xx_i) + P_{\stdFace^{i}}(\xx_i) } \leq \epsilon\norm{\zz} + \hat\kappa \epsilon,
\end{equation}
where the inequality follows from \eqref{eq:f1} and \eqref{eq:z_pf}.
Now, recapitulating, the facial
residual function $\psi_{\stdFace^i,\zz_i}$ has the property that if $\gamma_1,\gamma_2 \in \RR_+$ then  the relations
\begin{align*}
\yy_i\in \spanVec\stdFace^{i},\quad \dist(\yy_i,\stdFace^{i}) \leq \gamma_1,\quad \inProd{\yy_i}{\zz_i} \leq \gamma_2
\end{align*}
imply $\dist(\yy_i,\stdFace^i\cap \{\zz_i\}^\perp) \leq  \psi_{\stdFace^i,\zz_i}(\max_{1\leq j \leq 2}\{\gamma_j \},\norm{\yy_i})$.
Therefore, from \eqref{eq:f1}, \eqref{eq:revisef1} and the monotonicity of
$\psi_{\stdFace^i,\zz_i}$, we have upon recalling $\xx\in \spanVec\stdFace$ that
\begin{equation}\label{eq:x_df}
\dist(\xx_i,\stdFace^i\cap \{\zz_i\}^\perp) \leq  \psi_{\stdFace^i,\zz_i}(\max\{1,\hat\kappa+\|\zz\|\}\epsilon,\norm{\xx_i}).
\end{equation}
Finally, from \eqref{eq:x_df}, we conclude that
\begin{align}\notag
\dist(\xx, \stdFace\cap \{\zz\}^\perp) \le
\sum _{i=1}^m{\dist({\xx}_i, \stdFace^i \cap \{\zz_i\}^\perp)} \leq \sum _{i=1}^m\psi_{\stdFace^i,\zz_i}(\max\{1,\hat\kappa+\|\zz\|\}\epsilon,\norm{\xx}),
\end{align}
where we also used the monotonicity of $\psi_{\stdFace^i,\zz_i}$ for the last inequality.
This completes the proof.
%
\end{proof}

\section{The exponential cone}\label{sec:exp_cone}
In this section, we will use all the techniques developed so far to obtain error bounds for the 3D exponential cone $\expCone$. We will start with a study of its facial structure in Section~\ref{sec:facial_structure}, then we will compute its {\oneFRF}s in Section~\ref{sec:exp_frf}.
Finally, error bounds will be presented in Section~\ref{sec:exp_err}. In Section~\ref{sec:odd}, we summarize odd behaviour found in the facial structure of the exponential cone.

\subsection{Facial structure}\label{sec:facial_structure}

Recall that the exponential cone is defined as follows:
\begin{align}\label{d:Kexp}
K_{\exp}:=&\left \{(x,y,z)\;|\;y>0,z\geq ye^{x/y}\right \} \cup \left \{(x,y,z)\;|\; x \leq 0, z\geq 0, y=0  \right \}.
\end{align}

\begin{figure}
	\begin{center}
		\begin{tikzpicture}
		[scale=0.76]
		\node[anchor=south west,inner sep=0] (image) at (0,0) {\includegraphics[width=.75\linewidth]{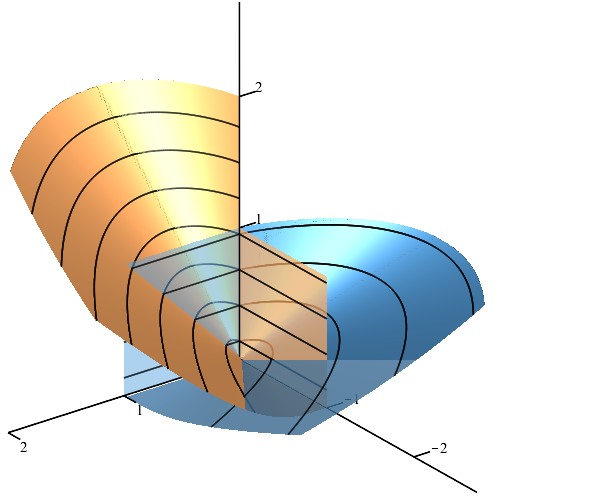}};
		\begin{scope}[x={(image.south east)},y={(image.north west)}]
		
		
		\node [left,brown] at (0,0.67) {$\mathbf{K_{{\bf \rm exp}}}$};
		
		
		\node [above] at (0.12,.88) {$\bigcup_{\beta \in \RR}\Fhat_{\beta}$};
		\draw [->,thick] (0.12,.88) -- (0.17,0.86);
		\draw [->,thick] (0.12,.88) -- (0.09,0.81);
		
		
		\node [below] at (0.46,.10) {$\Fhat_{\infty}$};
		\draw [->] (0.46,.10) -- (0.46,0.21);
		
		\node [above] at (0.48,.6) {$\Fhat_{-\infty}$};
		\draw [->,thick] (0.48,.6) -- (0.48,0.5);
		
		
		
		\node [blue] at (0.88,.4) {$\mathbf{K}_{{\rm \bf exp}}^*$};
		
		\node [right] at (0.70,.19) {$\bigcup_{\beta \in \RR}\zz_{\beta}$};
		\draw [->,thick] (0.70,.19) -- (0.63,0.20);
		\draw [->,thick] (0.70,.19) -- (0.72,0.25);
		
		
		\node [below] at (0.27,.10) {$\zz_{\beta =-\infty}$};
		\draw [->,thick] (0.27,.10) -- (0.27,0.21);
		
		\node [right] at (0.45,.72) {$\zz_{\beta =\infty}, \, \cone\{\zz_{\infty}\} = \Fhat_{ne}$};
		\draw [->,thick] (0.45,.72) -- (0.42,0.72);
		
		\node at (0.05,0.18) {$y$-axis};
		\node at (0.85,0.05) {$x$-axis};
		\node at (0.45,0.95) {$z$-axis};
		
		\end{scope}
		\end{tikzpicture}
	\end{center}
	\caption{The exponential cone and its dual, with faces and exposing vectors labeled according to our index $\beta$.}\label{fig:exp}
\end{figure}


\noindent Its dual cone is given by
\begin{align*}
K_{\exp}^*:=&\left \{(x,y,z)\;|\;x<0, ez\geq -xe^{y/x}\right \} \cup \left \{(x,y,z)\;|\; x = 0, z\geq 0, y\geq0  \right \}.
\end{align*}
It may therefore be readily seen that $K_{\exp}^*$ is a scaled and rotated version of $K_{\exp}$. 
In this subsection, we will describe the
nontrivial faces of  $K_{\exp}$; see Figure~\ref{fig:exp}.
We will show that we have the following types of nontrivial faces:
\begin{enumerate}[{\rm (a)}]
	\item infinitely many exposed extreme rays (1D faces) parametrized by $\beta \in \RR$ as follows:
	\begin{equation}\label{eq:exp_1d_beta}
	\Fhat_\beta  \coloneqq
	\left\{\left(-\beta y+y,y,e^{1-\beta}y \right)\;\bigg|\;y \in [0,\infty )\right\}.
	\end{equation}
	\item a single ``exceptional'' exposed extreme ray denoted by $\Fhat_{\infty}$:
	\begin{equation}\label{eq:exp_1d_exc}
	\Fhat_{\infty}\coloneqq \{(x,0,0)\;|\;x\le0 \}.
	\end{equation}
	\item a single non-exposed extreme ray denoted by
	$\Fhat _{ne} $:
		\begin{equation}\label{eq:exp_1d_ne}
	\Fhat _{ne}\coloneqq \{(0,0,z)\;|\;z\ge0 \}.
	\end{equation}
	\item a single 2D exposed face denoted by $\Fhat_{-\infty}$:
	\begin{equation}\label{eq:exp_2d}
	\Fhat_{-\infty}\coloneqq \{(x,y,z)\;|\;x\le0, z\ge0, y=0\},
	\end{equation}
	where we note that $\Fhat_{\infty}$ and  $\Fhat _{ne}$ are the extreme rays of $\Fhat_{-\infty}$.
\end{enumerate}
Notice that except for the case (c), all faces are
exposed and thus arise as an intersection $\{\zz\}^\perp\cap K_{\exp}$ for some $\zz \in K_{\exp}^*$.
To establish the above characterization, we start by examining how the components of $\zz$ determine
the corresponding exposed face.

\subsubsection{Exposed faces}\label{sec:exposed}
Let $\zz\in \expCone^*$ be such that $\{\zz\}^\perp\cap \expCone$ is a nontrivial face of $\expCone$. Then $\zz\neq 0$ and $\zz\in \partial K_{\exp}^*$. We consider the following cases.

\noindent \underline{$z_x < 0$}:
Since $\zz\in \partial K_{\exp}^*$, we must have $z_z e = -z_x e^{\frac{z_y}{z_x}}$ and hence
\begin{equation}\label{d:pz}
\zz=(z_x,z_y,-z_x e^{\frac{z_y}{z_x}-1}).
\end{equation}
Since $z_x \neq 0$, we see that $\qq \in \{\zz\}^\perp$ if and only if
\begin{equation}\label{qdotp1}
q_x+ q_y\left(\frac{z_y}{z_x}\right)-q_z e^{\frac{z_y}{z_x}-1} = 0.
\end{equation}
Solving \eqref{qdotp1} for $q_z$ and letting $\beta:=\frac{z_y}{z_x}$ to simplify the exposition, we have
\begin{equation}\label{qdotp4}
q_z = e^{1-\frac{z_y}{z_x}}\left(q_x+q_y\cdot \frac{z_y}{z_x}  \right) = e^{1-\beta}\left(q_x+q_y\beta  \right) \;\; \text{with}\;\;\beta:=\frac{z_y}{z_x}\in (-\infty,\infty).
\end{equation}	
Thus,
we obtain that $\{\zz\}^\perp = \left \{ \left(x,y, e^{1-\beta}\left(x+y\beta  \right) \right)\;\big|\; x,y \in \RR  \right \}$.
Combining this with the definition of $K_{\exp}$ and the fact that $\{\zz\}^\perp$ is a supporting hyperplane (so that $K_{\exp} \cap \{\zz\}^\perp = \partial K_{\exp}\cap \{\zz\}^\perp$) yields
\begin{equation}\label{capperp}
\begin{aligned}
&K_{\exp} \cap \{\zz\}^\perp = \partial K_{\exp}\cap \{\zz\}^\perp \\
=& \left \{ \left(x,y, e^{1-\beta}\left(x+y\beta  \right) \right) \;\;\big|\;\; e^{1-\beta}\left(x+y\beta  \right) = ye^{\frac{x}{y}},y>0   \right \} \\
& \bigcup \left \{ \left(x,y, e^{1-\beta}\left(x+y\beta  \right) \right) \;\;\big|\;\; x \leq 0, e^{1-\beta}\left(x+y\beta  \right) \geq 0,y=0   \right \} \\
=&\left \{ \left(x,y, e^{1-\beta}\left(x+y\beta  \right) \right) \;\;\big|\;\; e^{1-\beta}\left(x+y\beta  \right) = ye^{\frac{x}{y}},y>0   \right \} \cup \{0\}.
\end{aligned}
\end{equation}
We now refine the above characterization in the next proposition.
\begin{proposition}[Characterization of $\Fhat_\beta$, $\beta\in \RR$]\label{d:FcapbdKexp}
	Let $\zz \in K_{\exp}^*$ satisfy $\zz=(z_x,z_y,z_z)$, where $z_z e = -z_x e^{\frac{z_y}{z_x}}$ and $z_x <0$. Define $\beta=\frac{z_y}{z_x}$ as in \eqref{qdotp4} and let $\Fhat_\beta:= K_{\exp}\cap \{\zz\}^\perp$. Then
	\begin{align*}
	\Fhat_\beta  =
	\left\{\left(-\beta y+y,y,e^{1-\beta}y \right)\;\bigg|\;y \in [0,\infty )\right\}.
	\end{align*}
\end{proposition}
\begin{proof}
	Let $\Omega := \left\{\left(-\beta y+y,y,e^{1-\beta}y \right)\;\big|\;y \in [0,\infty )\right\}$. In view of \eqref{capperp}, we can check that $\Omega\subseteq \Fhat_\beta$. To prove the converse inclusion, pick any $\qq=\left(x,y,e^{1-\beta}(x+y\beta)\right) \in \Fhat_\beta$. We need to show that $\qq\in \Omega$.
	
	To this end, we note from \eqref{capperp} that if $y = 0$, then necessarily $\qq = \mathbf{0}$ and consequently $\qq\in \Omega$. On the other hand, if $y > 0$, then \eqref{capperp} gives $ye^{x/y}= (x+\beta y)e^{1-\beta}$. Then we have the following chain of equivalences:
	\begin{equation}\label{xneq0case}
	\begin{array}{crl}
	& ye^{x/y}&\!= (x+\beta y)e^{1-\beta} \\
	\iff & -e^{-1} &\!=-(x/y+\beta)e^{-(x/y+\beta)} \\
	\overset{\rm (a)}\iff & -x/y-\beta&\!=-1 \\
	\iff & x&\!=y-y\beta,
	\end{array}
	\end{equation}
	where (a) follows from the fact that the function $t\mapsto te^t$ is strictly increasing on $[-1,\infty)$.
	Plugging the last expression back into $\qq$, we may compute
	\begin{equation}\label{xneq0case6}
	q_z = e^{1-\beta}(x+y\beta) = e^{1-\beta}(y-y\beta+y\beta)=ye^{1-\beta}.
	\end{equation}
	Altogether, \eqref{xneq0case}, \eqref{xneq0case6} together with $y>0$ yield
	\begin{equation*}
	\qq=\left(y-\beta y ,y,ye^{1-\beta} \right)\in \Omega.
	\end{equation*}
	This completes the proof
\end{proof}
Next, we move on to the two remaining cases.

\noindent{\underline{$z_x = 0$, $z_z > 0$}}: Notice that $\qq\in \expCone$ means that $q_y \ge 0$ and $q_z\ge 0$. Since $z_z > 0$ and $z_y \ge 0$, in order to have $\qq\in \{\zz\}^\perp$, we must have $q_z = 0$. The the definition of $K_{\exp}$ also forces $q_y = 0$ and hence
\begin{equation}\label{Finfinity}
\{\zz\}^\perp\cap K_{\exp}=\{(x,0,0)\;|\; x \leq 0\} =:\Fhat_{\infty}.
\end{equation}
This one-dimensional face is exposed by any hyperplane with normal vectors coming from the set $\{(0,z_y,z_z):\; z_y\ge 0, z_z > 0)\}$.

\noindent{\underline{$z_x = 0$, $z_z = 0$}}: In this case, we have $z_y > 0$. In order to have $\qq\in \{\zz\}^\perp$, we must have $q_y = 0$. Thus
\begin{equation}\label{Fneginfinity}
\{\zz\}^\perp\cap K_{\exp} = \{(x,y,z)\;|\;x\le0, z\ge0, y=0\} =: \Fhat_{-\infty},
\end{equation}
which is the unique two-dimensional face of $K_{\exp}$.

\subsubsection{The single non-exposed face and completeness of the classification}
The face $\Fhat_{ne}$ is non-exposed because, as shown in
Proposition~\ref{d:FcapbdKexp}, \eqref{Finfinity} and \eqref{Fneginfinity}, it never arises as an intersection
of the form $\{\zz\}^\perp\cap K_{\exp}$, for $\zz \in K_{\exp}^*$.

We now show that all nontrivial faces of $K_{\exp}$ were
accounted for in \eqref{eq:exp_1d_beta}, \eqref{eq:exp_1d_exc}, \eqref{eq:exp_1d_ne}, \eqref{eq:exp_2d}.
First of all, by the discussion in Section~\ref{sec:exposed}, all nontrivial exposed faces must be among the ones in \eqref{eq:exp_1d_beta}, \eqref{eq:exp_1d_exc} and  \eqref{eq:exp_2d}.
So, let $\stdFace$ be a non-exposed face of $K_{\exp}$.
Then, it must be contained in a nontrivial exposed face of $K_{\exp}$.\footnote{This is a general fact. A proper face of a closed convex cone is contained in a proper exposed face, e.g., \cite[Proposition~3.6]{BW81}.}
	Therefore, $\stdFace$ must be a proper face of the unique 2D face \eqref{eq:exp_2d}.
This implies that  $\stdFace$ is one of the extreme rays of
\eqref{eq:exp_2d}: $\Fhat_{\infty}$ or $\Fhat_{ne}$. By
assumption, $\stdFace$ is non-exposed, so it must be
$\Fhat_{ne}$.

\subsection{{\OneFRF}s}\label{sec:exp_frf}
In this subsection, we will use the machinery developed in Section~\ref{sec:frf} to obtain the {\oneFRF}s for $K_{\exp}$.

Let us first discuss how the discoveries were originally made, and how that process motivated the development of the framework we built in Section~\ref{sec:frf}. The FRFs proven here were initially found by using the characterizations of Theorem~\ref{thm:1dfacesmain} and Lemma~\ref{lem:infratio} together with numerical experiments. Specifically, we used \emph{Maple}\footnote{Though one could use any suitable computer algebra package.} to numerically evaluate limits of relevant sequences \eqref{infratiohk}, as well as plotting lower dimensional slices of the function $\vv \mapsto \frakg(\|\vv-\ww\|)/\|\ww-\uu\|$, where $\ww$ and $\uu$ are defined similarly as in \eqref{gammabetaeta}.
	
A natural question is whether it might be simpler to change coordinates and work with the nearly equivalent $\ww \mapsto \frakg(\|\vv-\ww\|)/\|\ww-\uu\|$, since $\ww \in \{\zz \}^\perp$. However, $P_{\{\zz\}^\perp}^{-1}\{\ww\}\cap \partial \stdCone$ may contain multiple points, which creates many challenges. We encountered an example of this when working with the exponential cone, where the change of coordinates from $\vv$ to $\ww$ necessitates the introduction of the two real branches of the Lambert $\W$ function (see, for example, \cite{BL2016,bauschke2018proximal,burachik2019generalized} or \cite{Se15} for the closely related Wright Omega function). With terrible effort, one can use such a parametrization to prove the FRFs for $\Fhat_{\beta}, \beta \in \left[-\infty,\infty \right]\setminus \{\hat{\beta}:=-\W_{{\rm principal}}(2e^{-2})/2 \}$. However, the change of branches inhibits proving the result for the exceptional number $\hat{\beta}$. The change of variables to $\vv$ cures this problem by obviating the need for a branch function in the analysis; see \cite{WWJD} for additional details. This is why we present Theorem~\ref{thm:1dfacesmain} in terms of $\vv$. Computational investigation also pointed to the path of proof, though the proof we present may be understood without the aid of a computer.

\subsubsection{$\Fhat_{-\infty}$: the unique 2D face}
\label{sec:freas_2d}

Recall the unique 2D face of $K_{\exp}$:
$$
\stdFace_{-\infty}:=\{(x,y,z)\;|\; x\leq 0, z\geq 0, y=0 \}.
$$
Define the piecewise modified Boltzmann--Shannon entropy $\frakg_{-\infty}:\RP\to \RP$ as follows:
\begin{equation}\label{d:entropy}
\frakg_{-\infty}(t) := \begin{cases}
0 & \text{if}\;\; t=0,\\
-t \ln(t) & \text{if}\;\; t\in \left(0,1/e^2\right],\\
t+ \frac{1}{e^2} & \text{if}\;\; t>1/e^2.
\end{cases}
\end{equation}
For more on its usefulness in optimization, see, for example, \cite{bauschke2018proximal,burachik2019generalized}. We note that $\frakg_{-\infty}$ is monotone increasing and there exists $L\ge 1$ such that the following inequalities hold for every $t \in \RP$ and $M > 0$:\footnote{The third relation in \eqref{d:entropy_p} is derived from the second relation and the monotonicity of $\frakg_{-\infty}$ as follows: $\frakg_{-\infty}(Mt) = \frakg_{-\infty}(2^{\log_2 M}t )\le \frakg_{-\infty}(2^{\lceil |\log_2 M|\rceil}t )\le L^{\lceil |\log_2 M|\rceil}\frakg_{-\infty}(t)\le L^{1 +|\log_2 M|}\frakg_{-\infty}(t)$.}
\begin{equation}\label{d:entropy_p}
|t| \leq \frakg_{-\infty}(t), \quad \frakg_{-\infty}(2t) \leq L\frakg_{-\infty}(t),\quad \frakg_{-\infty}(Mt) \leq L^{1+|\log_2(M)|}\frakg_{-\infty}(t).
\end{equation}
With that, we prove in the next theorem that $\gamma_{\zz,\eta}$  is positive for $\Fhat_{-\infty}$, which implies that an \emph{entropic error bound} holds.

\begin{theorem}[Entropic error bound concerning $\Fhat_{-\infty}$]\label{thm:entropic}
	Let $\zz\in K_{\exp}^*$ with $z_x=z_z=0$ and $z_y > 0$ so that $\{\zz \}^\perp \cap K_{\exp}=\Fhat_{-\infty}$ is the two-dimensional face of $K_{\exp}$.
	Let $\eta > 0$ and let $\gamma_{\zz,\eta}$ be defined as in \eqref{gammabetaeta} with $\frakg = \frakg_{-\infty}$ in \eqref{d:entropy}. Then $\gamma_{\zz,\eta} \in (0,\infty]$ and
	\begin{equation}\label{eq:entropic}
	\dist(\qq,\Fhat_{-\infty})\le \max\{2,2\gamma_{\zz,\eta}^{-1}\}\cdot\frakg_{-\infty}(\dist(\qq,K_{\exp}))\ \ \ \mbox{whenever $\qq\in \{\zz\}^\perp\cap B(\eta)$.}
	\end{equation}
\end{theorem}
\begin{proof}
    In view of Lemma~\ref{lem:infratio}, take any
	$\bv \in \Fhat_{-\infty}$ and a sequence $\{\vv^k\}\subset \bd K_{\exp}\cap B(\eta) \backslash \Fhat_{-\infty}$ such that
	\begin{equation}\label{infratiohk_contradiction_neg_infty}
	\underset{k \rightarrow \infty}{\lim}\vv^k = \underset{k \rightarrow \infty}{\lim}\ww^k =\bv,
	\end{equation}
	where $\ww^k = P_{\{\zz\}^\perp}\vv^k$, $\uu^k = P_{\Fhat_{-\infty}}\ww^k$, and $\ww^k\neq \uu^k$. We will show that \eqref{infratiohkb} does not hold for $\frakg = \frakg_{-\infty}$.
	
	Since $\vv^k \notin \Fhat_{-\infty}$, in view of \eqref{d:Kexp} and \eqref{Fneginfinity}, we have $v^k_y>0$ and
	\begin{equation}\label{formulav_entropic}
	\vv^k = (v^k_x,v^k_y,v^k_ye^{v^k_x/v^k_y}) = (v^k_y\ln(v^k_z/v^k_y),v^k_y,v^k_z),
	\end{equation}
	where the second representation is obtained by solving for $v^k_x$ from $v^k_z = v^k_ye^{v^k_x/v^k_y} > 0$. Using the second representation in \eqref{formulav_entropic}, we then have
	\begin{equation}\label{formulawu_entropic}
	\ww^k = (v^k_y\ln(v^k_z/v^k_y),0,v^k_z)\ \ {\rm and}\ \ \uu^k = (0,0,v^k_z);
	\end{equation}
	here, we made use of the fact that $\ww^k\neq \uu^k$, which implies that $v^k_y\ln(v^k_z/v^k_y) > 0$ and thus the resulting formula for $\uu^k$. In addition, we also note from $v^k_y\ln(v^k_z/v^k_y) > 0$ (and $v^k_y>0$) that
	\begin{equation}\label{inequality_entropic}
	v^k_z > v^k_y > 0.
	\end{equation}
    Furthermore, since $\bar \vv \in \Fhat_{-\infty}$, we see from \eqref{Fneginfinity} and \eqref{infratiohk_contradiction_neg_infty} that
    \begin{equation}\label{limitvyzero}
      \lim_{k\to\infty} v_y^k = 0.
    \end{equation}
	Now, using \eqref{formulav_entropic}, \eqref{formulawu_entropic}, \eqref{limitvyzero} and the definition of $\frakg_{-\infty}$, we see that for $k$ sufficiently large,
	\begin{equation}\label{e:entropic}
	\frac{\frakg_{-\infty}(\|\vv^k-\ww^k\|)}{\|\uu^k-\ww^k\|} = \frac{-v_y^k \ln(v_y^k)}{v_y^k\ln(v_z^k/v_y^k)} = \frac{-\ln(v_y^k)}{\ln(v_z^k)-\ln(v_y^k)}.
	\end{equation}
	We will show that \eqref{infratiohkb} does not hold for $\frakg = \frakg_{-\infty}$ in each of the following cases.
	\begin{enumerate}[(I)]
		\item\label{entropiccase1} $\bar{v}_z >0$.
		\item\label{entropiccaseinfty} $\bar{v}_z =0$.
	\end{enumerate}
	
    \noindent \ref{entropiccase1}: In this case, we deduce from \eqref{limitvyzero} and \eqref{e:entropic} that
	\begin{equation*}
	\lim_{k\rightarrow\infty} \frac{\frakg_{-\infty}(\|\vv^k-\ww^k\|)}{\|\uu^k-\ww^k\|} = 1.
	\end{equation*}
	Thus \eqref{infratiohkb} does not hold for $\frakg = \frakg_{-\infty}$.
	
	\noindent \ref{entropiccaseinfty}: By passing to a subsequence if necessary, we may assume that $v^k_z < 1$ for all $k$. This together with \eqref{inequality_entropic} gives $\frac{\ln(v_z^k)}{\ln(v_y^k)} \in (0,1)$ for all $k$. Thus, we conclude from \eqref{e:entropic} that for all $k$,
	\begin{align*}
	\frac{\frakg_{-\infty}(\|\vv^k-\ww^k\|)}{\|\uu^k-\ww^k\|} & = \frac{-\ln(v_y^k)}{\ln(v_z^k)-\ln(v_y^k)} = \frac{1}{1 - \frac{\ln(v_z^k)}{\ln(v_y^k)}}> 1.
	\end{align*}
	Consequently, \eqref{infratiohkb} also fails for $\frakg = \frakg_{-\infty}$ in this case.
	
	Having shown that \eqref{infratiohkb} does not hold for $\frakg = \frakg_{-\infty}$ in any case, we conclude by Lemma~\ref{lem:infratio} that $\gamma_{\zz,\eta} \in \left(0,\infty \right]$.
	With that, \eqref{eq:entropic} follows from Theorem~\ref{thm:1dfacesmain} and \eqref{d:entropy_p}.
\end{proof}
Using Theorem~\ref{thm:entropic}, we can also show weaker H\"olderian error bounds.

\begin{corollary}\label{col:2d_hold}
Let $\zz\in K_{\exp}^*$ with $z_x=z_z=0$ and $z_y > 0$ so that $\{\zz \}^\perp \cap \expCone=\Fhat_{-\infty}$ is the two-dimensional face of $\expCone$.
Let $\eta>0$, $\alpha \in (0,1)$, and $\gamma_{\zz,\eta}$ be as in \eqref{gammabetaeta} with $\frakg = |\cdot|^\alpha$.
Then $\gamma_{\zz,\eta} \in (0,\infty]$ and
\begin{equation*}
	\dist(\qq,\Fhat_{-\infty})\le \max\{2\eta^{1-\alpha},2\gamma_{\zz,\eta}^{-1}\}\cdot\dist(\qq,K_{\exp})^\alpha\ \ \ \mbox{whenever $\qq\in \{\zz\}^\perp\cap B(\eta)$.}
\end{equation*}
\end{corollary}
\begin{proof}
Suppose that $\gamma_{\zz,\eta} = 0$ and let sequences $\{\vv^k\},\{\ww^k\},\{\uu^k\}$ be as in Lemma~\ref{lem:infratio}. Then $\vv^k\neq \ww^k$ for all $k$ because $\{\vv^k\}\subset \expCone\backslash \Fhat_{-\infty}$, $\{\ww^k\}\subset\{\zz\}^\perp$, and $\Fhat_{-\infty} = \expCone\cap \{\zz\}^\perp$.
Since $\frakg_{-\infty}(t)/|t|^{\alpha} \downarrow 0 $ as $t \downarrow 0$ we have
\begin{equation*}
\liminf_{k\rightarrow\infty} \frac{\frakg_{-\infty}(\|\ww^k - \vv^k\|)}{\|\ww^k- \uu^k\|} =  \liminf_{k\rightarrow\infty} \frac{\frakg_{-\infty}(\|\ww^k - \vv^k\|)}{\|\ww^k- \vv^k\|^\alpha} \frac{\|\ww^k - \vv^k\|^{\alpha}}{\|\ww^k- \uu^k\|} = 0,
\end{equation*}
which contradicts Theorem~\ref{thm:entropic} because the quantity in \eqref{gammabetaeta} should be positive for $\frakg = \frakg _{-\infty}$.	
\end{proof}

Recalling \eqref{d:entropy_p}, we obtain {\oneFRF}s using Theorem~\ref{thm:entropic} and Corollary~\ref{col:2d_hold} in combination with
Theorem~\ref{thm:1dfacesmain}, Remark~\ref{rem:kappa} and Lemma~\ref{lem:facialresidualsbeta}.
\begin{corollary}[{\oneFRFs} concerning $\Fhat_{-\infty}$]\label{col:frf_2dface_entropic}
	Let $\zz\in K_{\exp}^*$ be such that $\{\zz \}^\perp \cap K_{\exp}=\Fhat_{-\infty}$ is the two-dimensional face of $K_{\exp}$.
	Let $\frakg = \frakg_{-\infty}$ in \eqref{d:entropy} or
	$\frakg = |\cdot|^\alpha$ for $\alpha\in (0,1)$.
	
	Let $\kappa_{\zz,t}$ be defined as in \eqref{haha2}. Then the function $\psi_{\stdCone,\zz}:\RR_+\times\RR_+\to \RP$ given by
	\begin{equation*}
	\psi_{\stdCone,\zz}(\epsilon,t):=\max \left\{\epsilon,\epsilon/\|\zz\| \right\} + \kappa_{\zz,t}\frakg(\epsilon +\max \left\{\epsilon,\epsilon/\|\zz\| \right\} )
	\end{equation*}
	is a {\oneFRFs} for $K_{\exp}$ and $\zz$. In particular,
	there exist $\kappa > 0$ and a nonnegative monotone nondecreasing function $\rho:\RR_+ \to \RR_+$ such that the function $\hat \psi_{\stdCone,\zz}$ given by $\hat \psi_{\stdCone,\zz}(\epsilon,t) \coloneqq \kappa \epsilon + \rho(t)\frakg(\epsilon)$ is a {\oneFRFs}   for $K_{\exp}$ and $\zz$.
\end{corollary}

\subsubsection{$\Fhat_\beta$: the family of one-dimensional faces $\beta \in \RR$}

Recall from Proposition~\ref{d:FcapbdKexp} that for each $\beta \in \RR$,
\begin{align*}
	\Fhat_\beta :=
	\left\{\left(-\beta y+y,y,e^{1-\beta}y \right)\;\bigg|\;y \in [0,\infty )\right\}
\end{align*}
is a one-dimensional face of $K_{\exp}$. We will now show that for $\Fhat_\beta$, $\beta\in \RR$, the $\gamma_{\zz,\eta}$ defined in Theorem~\ref{thm:1dfacesmain} is positive when $\frakg = |\cdot|^\frac12$.
Our discussion will be centered around the following quantities, which were also defined and used in the proof of Theorem~\ref{thm:1dfacesmain}. Specifically, for $\zz \in \expCone^*$ such that $\Fhat_{\beta} = \expCone\cap \{\zz\}^\perp$, we let $\vv\in \bd K_{\exp}\cap B(\eta)\backslash \Fhat_\beta$ and define
\begin{align} \label{d:w}
\ww:=P_{\{\zz \}^\perp }\vv\ \ \ {\rm and}\ \ \
\uu:=P_{\Fhat_\beta}\ww.
\end{align}

We first note the following three important vectors:
\begin{equation}\label{veczfp}
\hzz := \begin{bmatrix}
1\\ \beta\\ -e^{\beta-1}
\end{bmatrix},\ \ \ \hff = \begin{bmatrix}
1-\beta \\ 1\\ e^{1-\beta}
\end{bmatrix},\ \ \ \hpp = \begin{bmatrix}
\beta e^{1-\beta} + e^{\beta-1}\\
-e^{1-\beta} - (1-\beta)e^{\beta-1}\\
\beta^2-\beta+1
\end{bmatrix}.
\end{equation}
Note that $\hzz$ is parallel to $\zz$ in \eqref{d:pz} (recall that $z_x < 0$ for $\Fhat_\beta$, where $\beta := \frac{z_y}{z_x}\in \RR$), $\Fhat_\beta$ is the conic hull of $\{\hff\}$ according to Proposition~\ref{d:FcapbdKexp}, $\langle\hzz,\hff\rangle=0$ and $\hpp = \hzz\times\hff\neq {\bf 0}$. These three {\em nonzero} vectors form a mutually orthogonal set. The next lemma represents $\|\uu - \ww\|$ and $\|\ww - \vv\|$ in terms of inner products of $\vv$ with these vectors, whenever possible.

\begin{lemma}\label{lem:distance}
	Let $\beta\in \RR$ and $\zz\in K_{\exp}^*$ with $z_x<0$ such that $\Fhat_\beta = \{\zz \}^\perp \cap K_{\exp}$ is a one-dimensional face of $K_{\exp}$. Let $\eta >0$, $\vv\in \bd K_{\exp}\cap B(\eta)\backslash \Fhat_\beta$ and define $\ww$ and $\uu$ as in \eqref{d:w}. Let $\{\hzz,\hff,\hpp\}$ be as in \eqref{veczfp}. Then
	\[
	\|\ww - \vv\| = \frac{|\langle \hzz,\vv\rangle|}{\|\hzz\|}\ \ \ {\rm and}\ \ \
	\|\ww - \uu\| =\begin{cases}
	\frac{|\langle \hpp,\vv\rangle|}{\|\hpp\|} & {\rm if}\ \langle\hff,\vv\rangle\ge 0,\\
	\|\ww\| & {\rm otherwise}.
	\end{cases}
	\]
	Moreover, when $\langle\hff,\vv\rangle\ge 0$, we have $\uu = P_{{\rm span}\Fhat_\beta}\ww$.
\end{lemma}
\begin{proof}
	Since $\{\hzz,\hff,\hpp\}$ is orthogonal, one can decompose $\vv$ as
	\begin{equation}\label{eq:vdecompose}
	\vv = \lambda_1 \hzz + \lambda_2 \hff + \lambda_3 \hpp,
	\end{equation}
	with
	\begin{equation}\label{eq:lambda}
	\lambda_1 = \langle \hzz,\vv\rangle/\|\hzz\|^2, \ \  \lambda_2 = \langle \hff,\vv\rangle/\|\hff\|^2\ \ {\rm and}\ \ \lambda_3 = \langle \hpp,\vv\rangle/\|\hpp\|^2.
	\end{equation}
	Also, since $\hzz$ is parallel to $\zz$, we must have $\ww = \lambda_2 \hff + \lambda_3 \hpp$. Thus, it holds that $\|\ww-\vv\| = |\lambda_1|\|\hzz\|$ and the first conclusion follows from this and \eqref{eq:lambda}.
	
	Next, we have $\uu = \hat t \,\hff$, where
	\[
	\hat t= \argmin_{t\ge 0}\|\ww - t\hff\| = \begin{cases}
	\frac{\langle \ww,\hff\rangle}{\|\hff\|^2} & {\rm if}\ \langle \hff,\ww\rangle\ge 0,\\
	0 & {\rm otherwise}.
	\end{cases}
	\]
	Moreover, observe from \eqref{eq:vdecompose} that $\langle \hff,\ww\rangle = \langle \hff,\vv - \lambda_1\hzz\rangle = \langle \hff,\vv\rangle$. These mean that when $\langle \hff,\vv\rangle< 0$, we have $\uu = 0$, while when $\langle \hff,\vv\rangle\ge 0$, we have $\uu = \frac{\langle \ww,\hff\rangle}{\|\hff\|^2}\hff=P_{{\rm span}\Fhat_\beta}\ww$ and
	\[
	\|\ww - \uu\| = \left\|\ww - \frac{\langle \ww,\hff\rangle}{\|\hff\|^2}\hff\right\| = |\lambda_3|\|\hpp\| = |\langle\hpp,\vv\rangle|/\|\hpp\|,
	\]
	where the second and the third equalities follow from \eqref{eq:vdecompose}, \eqref{eq:lambda}, and the fact that $\ww = \lambda_2 \hff + \lambda_3 \hpp$.
	This completes the proof.
\end{proof}

We now prove our main theorem in this section.
\begin{theorem}[H\"{o}lderian error bound concerning $\Fhat_\beta$, $\beta\in \RR$]\label{thm:nonzerogamma}
	Let $\beta\in \RR$ and $\zz\in K_{\exp}^*$ with $z_x<0$ such that $\Fhat_\beta = \{\zz \}^\perp \cap K_{\exp}$ is a one-dimensional face of $K_{\exp}$.
	Let $\eta > 0$ and let $\gamma_{\zz,\eta}$ be defined as in \eqref{gammabetaeta} with $\frakg = |\cdot|^\frac12$. Then $\gamma_{\zz,\eta} \in (0,\infty]$ and
	\[
	\dist(\qq,\Fhat_\beta)\le \max\{2\sqrt{\eta},2\gamma_{\zz,\eta}^{-1}\}\cdot\dist(\qq,\expCone)^\frac12\ \ \ \mbox{whenever $\qq\in \{\zz\}^\perp\cap B(\eta)$.}
	\]
\end{theorem}
\begin{proof}
	In view of Lemma~\ref{lem:infratio}, take any
	$\bv \in \Fhat_\beta$ and a sequence $\{\vv^k\}\subset \bd K_{\exp}\cap B(\eta) \backslash \Fhat_\beta$ such that
	\begin{equation}\label{infratiohk_contradiction}
	\underset{k \rightarrow \infty}{\lim}\vv^k = \underset{k \rightarrow \infty}{\lim}\ww^k =\bv,
	\end{equation}
	where $\ww^k = P_{\{\zz\}^\perp}\vv^k$, $\uu^k = P_{\Fhat_\beta}\ww^k$, and $\ww^k\neq \uu^k$. We will show that \eqref{infratiohkb} does not hold for $\frakg = |\cdot|^\frac12$.
	
	We first suppose that $\vv^k\in \Fhat_{-\infty}$ infinitely often. By extracting a subsequence if necessary, we may assume that $\vv^k\in \Fhat_{-\infty}$ for all $k$. From the definition of $\Fhat_{-\infty}$ in \eqref{Fneginfinity}, we have $v_x^k\le 0$, $v_y^k=0$ and $v_z^k \ge 0$. Thus, recalling the definition of $\hzz$ in \eqref{veczfp}, it holds that
	\begin{equation}\label{easycase1}
	\begin{aligned}
	|\langle\hzz,\vv^k\rangle| &= |v^k_x - v_z^ke^{\beta-1}| = -v^k_x + v_z^ke^{\beta-1} = |v^k_x| + e^{\beta-1}|v_z^k| \\
	&\ge \min\{1,e^{\beta-1}\}\sqrt{|v^k_x|^2 + |v_z^k|^2} = \min\{1,e^{\beta-1}\}\|\vv^k\|,
	\end{aligned}
	\end{equation}
	where the last equality holds because $v_y^k=0$. Next, using properties of projections onto subspaces, we have $\|\ww^k\|  \le \|\vv^k\|$.
	This together with Lemma~\ref{lem:distance} and \eqref{easycase1} shows that
	\[
	\|\ww^k - \uu^k\| = \dist(\ww^k,\Fhat_\beta)\le \|\ww^k\| \le \|\vv^k\|\le \frac{1}{\min\{1,e^{\beta-1}\}}|\langle\hzz,\vv^k\rangle| = \frac{\|\hzz\|}{\min\{1,e^{\beta-1}\}}\|\ww^k - \vv^k\|.
	\]
	Since $\ww^k \to \bv$ and $\vv^k \to \bv$, the above display shows that \eqref{infratiohkb} does not hold for $\frakg = |\cdot|^\frac12$ in this case.
	
	Next, suppose that $\vv^k\notin \Fhat_{-\infty}$ for all large $k$ instead. By passing to a subsequence, we assume that $\vv^k\notin \Fhat_{-\infty}$ for all $k$. In view of \eqref{d:Kexp} and \eqref{Fneginfinity}, this means in particular that
	\begin{equation}\label{formulavstar}
	\vv^k = (v^k_x,v^k_y,v^k_ye^{v^k_x/v^k_y})\ \ {\rm and}\ \ v^k_y>0\ \mbox{ for all}\ k.
	\end{equation}
	
	We consider two cases  and show that \eqref{infratiohkb} does not hold for $\frakg = |\cdot|^\frac12$ in either of them:
	\begin{enumerate}[(I)]
		\item\label{ebsubcasev0} $\langle \hff,\vv^k\rangle\ge 0$ infinitely often;
		\item\label{ebsubcasevneq0} $\langle \hff,\vv^k\rangle< 0$ for all large $k$.
	\end{enumerate}
	
	\ref{ebsubcasev0}: Since $\langle \hff,\vv^k\rangle\ge 0$ infinitely often, by extracting a subsequence if necessary, we assume that $\langle \hff,\vv^k\rangle\ge 0$ for all $k$. Now, consider the following functions:
	\[
	\begin{aligned}
	h_1(\zeta)&:= \zeta + \beta - e^{\beta+\zeta-1},\\
	h_2(\zeta)&:= (\beta e^{1-\beta} + e^{\beta-1})\zeta - e^{1-\beta}-(1-\beta)e^{\beta-1} + (\beta^2-\beta+1)e^\zeta.
	\end{aligned}
	\]
	Using these functions, Lemma~\ref{lem:distance}, \eqref{veczfp} and \eqref{formulavstar}, one can see immediately that
	\begin{equation}\label{hahahehe1}
	\|\ww^k - \vv^k\| = \frac{|\langle\hzz,\vv^k\rangle|}{\|\hzz\|} = \frac{v^k_y |h_1(v^k_x/v^k_y)|}{\|\hzz\|}\ \ {\rm and}\ \
	\|\ww^k - \uu^k\| = \frac{|\langle\hpp,\vv^k\rangle|}{\|\hpp\|} = \frac{v^k_y |h_2(v^k_x/v^k_y)|}{\|\hpp\|}.
	\end{equation}
	Note that $h_1$ is zero {\bf if and only if} $\zeta = 1 - \beta$.
	Furthermore, we have $h_1'(1-\beta) = 0$ and $h_1''(1-\beta) = -1$.
	Then, considering the Taylor expansion of $h_1$ around $1-\beta$ we have
	\begin{equation*}
	h_1(\zeta) = 1 + (\zeta + \beta-1) - e^{\beta+\zeta-1} = -\frac{(\zeta+\beta-1)^2}{2} + O(|\zeta+\beta-1|^3)\ \ \ {\rm as}\ \ \zeta\to 1-\beta.
	\end{equation*}
	Also, one can check that $h_2(1-\beta)=0$ and that
	\[
	h_2'(1-\beta) = \beta e^{1-\beta} + e^{\beta-1} + (\beta^2-\beta+1)e^{1-\beta} = e^{\beta-1} + (\beta^2+1)e^{1-\beta} > 0.
	\]
	Therefore, we have the following Taylor expansion of $h_2$ around $1-\beta$:
	\begin{equation}\label{newlyadded}
	h_2(\zeta) = (e^{\beta-1} + (\beta^2+1)e^{1-\beta})(\zeta + \beta-1) + O(|\zeta+\beta-1|^2)\ \ \ {\rm as}\ \ \zeta\to 1-\beta.
	\end{equation}
	Thus, using the Taylor expansions of $h_1$ and $h_2$ at $1-\beta$
	we have
	\begin{equation}\label{taylor_limit}
	\lim\limits_{\zeta\to 1-\beta}\frac{|h_1(\zeta)|^\frac12}{|h_2(\zeta)|} = \frac1{\sqrt{2}(e^{\beta-1} + (\beta^2+1)e^{1-\beta})}> 0.
	\end{equation}
	Hence, there exist $C_h > 0$ and $\epsilon>0$ so that
	\begin{equation}\label{maybeusefulrel}
	|h_1(\zeta)|^\frac12 \ge C_h|h_2(\zeta)| \ \ {\rm whenever}\ |\zeta - (1-\beta)| \le \epsilon.
	\end{equation}
	Next, consider the following function\footnote{Notice that this function is well defined because $h_1$ is zero only at $1-\beta$ and thus we will not end up with $\frac00$.}
	\[
	H(\zeta):=\begin{cases}
	\frac{|h_1(\zeta)|}{|h_2(\zeta)|} & {\rm if}\ |\zeta -(1-\beta)|\ge \epsilon, h_2(\zeta)\neq 0,\\
	\infty & {\rm otherwise}.
	\end{cases}
	\]
	Then it is easy to check that $H$ is proper closed and is never zero.
	Moreover, by direct computation, we have $\lim\limits_{\zeta\to \infty}H(\zeta) = \frac{e^{\beta-1}}{\beta^2-\beta+1} > 0$ and
	\[
	\lim\limits_{\zeta\to -\infty}H(\zeta) = \begin{cases}
	|\beta e^{1-\beta} + e^{\beta-1}|^{-1} & {\rm if}\ \beta e^{1-\beta} + e^{\beta-1}\neq 0,\\
	\infty & {\rm otherwise}.
	\end{cases}
	\]
	Thus, we deduce that $\inf H > 0$.
	
	Now, if it happens that $|v^k_x/v^k_y - (1-\beta)|> \epsilon$ for all large $k$, upon letting $\zeta_k:= v^k_x/v^k_y$, we have from \eqref{hahahehe1} that for all large $k$,
	\begin{equation}\label{finalcase2}
	\begin{aligned}
	\frac{\|\ww^k-\vv^k\|}{\|\ww^k-\uu^k\|} &=\frac{\|\hpp\|}{\|\hzz\|}\frac{|h_1(\zeta_k)|}{|h_2(\zeta_k)|} = \frac{\|\hpp\|}{\|\hzz\|}H(\zeta_k)\ge \frac{\|\hpp\|}{\|\hzz\|}\inf H >0,
	\end{aligned}
	\end{equation}
	where the second equality holds because of the definition of $H$ and the facts that $\ww^k\neq \uu^k$ (so that $h_2(\zeta_k)\neq 0$ by \eqref{hahahehe1}) and $|v^k_x/v^k_y - (1-\beta)|> \epsilon$ for all large $k$.
	
	On the other hand, if it holds that $|v^k_x/v^k_y - (1-\beta)|\le \epsilon$ infinitely often, then by extracting a further subsequence, we may assume that $|v^k_x/v^k_y - (1-\beta)|\le \epsilon$ for all $k$. Upon letting $\zeta_k:= v^k_x/v^k_y$, we have from \eqref{hahahehe1} that
	\begin{equation}\label{finalcase1}
	\begin{aligned}
	\frac{\|\ww^k-\uu^k\|}{\|\ww^k-\vv^k\|^\frac12} =\frac{\sqrt{v^k_y\|\hzz\|}}{\|\hpp\|}\frac{|h_2(\zeta_k)|}{|h_1(\zeta_k)|^\frac12} \le \frac{\sqrt{v^k_y\|\hzz\|}}{C_h\|\hpp\|} \le \frac{\sqrt{\eta\|\hzz\|}}{C_h\|\hpp\|},
	\end{aligned}
	\end{equation}
	where the first inequality holds thanks to $|v^k_x/v^k_y - (1-\beta)|\le \epsilon$ for all $k$, \eqref{maybeusefulrel} and the fact that $\ww^k\neq {\uu}^k$ (so that $h_2(\zeta_k)\neq 0$ and hence $h_1(\zeta_k)\neq 0$), and the second inequality holds because $\vv^k\in B(\eta)$.
	
	Using \eqref{finalcase2} and \eqref{finalcase1} together with \eqref{infratiohk_contradiction}, we see that \eqref{infratiohkb} does not hold for $\frakg = |\cdot|^\frac12$. This concludes case~\ref{ebsubcasev0}.
	
	\ref{ebsubcasevneq0}: By passing to a subsequence, we may assume that $\langle \hff,\vv^k\rangle< 0$ for all $k$. Then we see from \eqref{veczfp} and \eqref{formulavstar} that
	\[
	(1-\beta)\frac{v_x^k}{v_y^k} + 1 + e^{1-\beta}e^{v_x^k/v_y^k} = \frac1{v_y^k} \langle \hff,\vv^k\rangle< 0.
	\]
	Using this together with the fact that $(1-\beta)^2+1+e^{2(1-\beta)} > 0$, we deduce that
	there exists $\epsilon> 0$ so that
	\begin{equation}\label{eq:keyrelation}
	\left|\frac{v_x^k}{v_y^k}-(1-\beta)\right|\ge \epsilon\ \ \mbox{for all}\ k.
	\end{equation}
	Now, consider the following function
	\[
	G(\zeta):= \frac{|\zeta+\beta-e^{\beta+\zeta-1}|}{\sqrt{\zeta^2+1+e^{2\zeta}}}.
	\]
	Then $G$ is continuous and is zero {\bf if and only if} $\zeta = 1-\beta$. Moreover, by direct computation, we have $\lim\limits_{\zeta\to\infty} G(\zeta) = e^{\beta-1} > 0$ and $\lim\limits_{\zeta\to-\infty} G(\zeta) = 1 > 0$. Thus, it follows that
	\begin{equation}\label{infG}
	\underline{G}:= \inf_{|\zeta +\beta-1|\ge \epsilon}G(\zeta) > 0.
	\end{equation}
	Finally, since $\langle \hff,\vv^k\rangle< 0$ for all $k$, we see that
	\[
	\begin{aligned}
	\frac{\|\ww^k-\vv^k\|}{\|\ww^k-\uu^k\|}&\overset{\rm (a)}\ge \frac{\|\ww^k-\vv^k\|}{\|\vv^k\|} \overset{\rm (b)}= \frac{|{\widehat{z}}_xv^k_x + {\widehat{z}}_yv^k_y + {\widehat{z}}_zv^k_ye^{v^k_x/v^k_y}|}{\|\hzz\|\sqrt{(v^k_x)^2+(v^k_y)^2+(v^k_y)^2e^{2v^k_x/v^k_y}}}\\
	&\overset{\rm (c)}=\frac{|{\widehat{z}}_x\zeta_k + {\widehat{z}}_y + {\widehat{z}}_ze^{\zeta_k}|}{\|\hzz\|\sqrt{(\zeta_k)^2 + 1 + e^{2\zeta_k}}} \overset{\rm (d)}= \frac{1}{\|\hzz\|}G(\zeta_k)\overset{\rm (e)}\ge \frac{1}{\|\hzz\|}\underline{G} > 0,
	\end{aligned}
	\]
	where (a) follows from $\|\ww^k-\uu^k\| = \|\ww^k\|$ (see Lemma~\ref{lem:distance}) and $\|\ww^k\|\leq \norm{\vv^k}$ (because $\ww^k$ is the projection of ${\vv}^k$ onto a subspace), (b) follows from Lemma~\ref{lem:distance} and \eqref{formulavstar}, (c) holds because $v^k_y > 0$ (see \eqref{formulavstar}) and we defined $\zeta_k:= v_x^k/v_y^k$, (d) follows from \eqref{veczfp} and the definition of $G$, and (e) follows from \eqref{eq:keyrelation} and \eqref{infG}. The above together with \eqref{infratiohk_contradiction} shows that \eqref{infratiohkb} does not hold for $\frakg = |\cdot|^\frac12$, which is what we wanted to show in case~\ref{ebsubcasevneq0}.
	
	Summarizing the above cases, we conclude that there does not exist the sequence $\{\vv^k\}$ and its associates so that \eqref{infratiohkb} holds for $\frakg = |\cdot|^\frac12$. By Lemma~\ref{lem:infratio}, it must then hold that $\gamma_{\zz,\eta}\in (0,\infty]$ and we have the desired error bound in view of Theorem~\ref{thm:1dfacesmain}. This completes the proof.
\end{proof}

Combining Theorem~\ref{thm:nonzerogamma}, Theorem~\ref{thm:1dfacesmain} and Lemma~\ref{lem:facialresidualsbeta}, and using the observation that $\gamma_{\zz,0}=\infty$ (see \eqref{gammabetaeta}), we obtain a {\oneFRF} in the following corollary.
\begin{corollary}[{\oneFRFs} concerning $\Fhat_\beta$, $\beta\in \RR$]\label{col:frf_fb}
	Let $\beta\in \RR$ and $\zz\in K_{\exp}^*$ with $z_x<0$ such that $\Fhat_\beta = \{\zz \}^\perp \cap K_{\exp}$ is a one-dimensional face of $K_{\exp}$.
	Let $\kappa_{\zz,t}$ be defined as in \eqref{haha2} with $\frakg = |\cdot|^\frac12$. Then the function $\psi_{\stdCone,\zz}:\RR_+\times \RR_+\to \RR_+$ given by
	\begin{equation*}
	\psi_{\stdCone,\zz}(\epsilon,t) := \max \left\{\epsilon,\epsilon/\|\zz\| \right\} + \kappa_{\zz,t}(\epsilon +\max \left\{\epsilon,\epsilon/\|\zz\| \right\} )^\frac12
	\end{equation*}
	is a {\oneFRFs}  for $K_{\exp}$. In particular,
	there exist $\kappa > 0$ and a nonnegative monotone nondecreasing function $\rho:\RR_+ \to \RR_+$ such that the function $\hat \psi_{\stdCone,\zz}$ given by $\hat \psi_{\stdCone,\zz}(\epsilon,t) \coloneqq \kappa \epsilon + \rho(t)\sqrt{\epsilon}$ is a {\oneFRFs}   for $K_{\exp}$ and $\zz$.
\end{corollary}

\subsubsection{$\Fhat_{\infty}$: the exceptional one-dimensional face}

Recall the special one-dimensional face of $K_{\exp}$ defined by
$$
\stdFace_{\infty}:= \{(x,0,0) \;|\; x\leq 0 \}.
$$
We first show that we have a Lipschitz error bound for any exposing normal vectors $\zz = (0, z_y,z_z)$ with $z_y > 0$ and $z_z > 0$.
\begin{theorem}[Lipschitz error bound concerning $\Fhat_{\infty}$]\label{thm:nonzerogammasec71}
	Let $\zz\in K_{\exp}^*$ with $z_x=0$, $z_y > 0$ and $z_z>0$ so that $\{\zz \}^\perp \cap K_{\exp}=\Fhat_\infty$.
	Let $\eta > 0$ and let $\gamma_{\zz,\eta}$ be defined as in \eqref{gammabetaeta} with $\frakg = |\cdot|$. Then $\gamma_{\zz,\eta} \in (0,\infty]$ and
	\[
	\dist(\qq,\Fhat_{\infty})\le \max\{2,2\gamma_{\zz,\eta}^{-1}\}\cdot\dist(\qq,K_{\exp})\ \ \ \mbox{whenever $\qq\in \{\zz\}^\perp\cap B(\eta)$.}
	\]
\end{theorem}
\begin{proof}
	Without loss of generality, upon scaling, we may assume that $\zz = (0,a,1)$ for some $a > 0$.
	Similarly as in the proof of Theorem~\ref{thm:nonzerogamma}, we will consider the following vectors:
	\begin{equation*}
	\tzz := \begin{bmatrix}
	0\\a\\1
	\end{bmatrix}, \ \
	\tff := \begin{bmatrix}
	-1\\0\\0
	\end{bmatrix},\ \
	\tpp := \begin{bmatrix}
	0\\1\\-a
	\end{bmatrix}.
	\end{equation*}
	Here, $\Fhat_\infty$ is the conical hull of $\tff$ (see \eqref{Finfinity}), and $\tpp$ is constructed so that $\{\tzz,\tff,\tpp\}$ is orthogonal.
	
	Now, let $\vv\in \bd K_{\exp}\cap B(\eta)\backslash \Fhat_\infty$, $\ww = P_{\{\zz\}^\perp }\vv$ and $\uu=P_{\Fhat_\infty}\ww$ with $\uu\neq \ww$. Then, as in Lemma~\ref{lem:distance}, by decomposing $\vv$ as a linear combination of $\{\tzz,\tff,\tpp\}$, we have
	\begin{equation}\label{haha5}
	\|\ww - \vv\| = \frac{|\langle\tzz,\vv\rangle|}{\|\tzz\|}\ \ {\rm and}\ \ \|\ww-\uu\| = \begin{cases}
	\frac{|\langle\tpp,\vv\rangle|}{\|\tpp\|} & {\rm if}\ \langle\tff,\vv\rangle\ge 0,\\
	\|\ww\| & {\rm otherwise}.
	\end{cases}
	\end{equation}
	We consider the following cases for estimating $\gamma_{\zz,\eta}$.
	\begin{enumerate}[(I)]
		\item\label{inftyebcase1} $\vv\in \Fhat_{-\infty}\backslash \Fhat_\infty$;
		\item\label{inftyebcase2} $\vv\notin \Fhat_{-\infty}$ with $v_x \le 0$;
		\item\label{inftyebcase3} $\vv\notin \Fhat_{-\infty}$ with $v_x > 0$.
	\end{enumerate}
	
	\ref{inftyebcase1}: In this case, $\vv = (v_x,0,v_z)$ with $v_x\le 0\le v_z$; see \eqref{Fneginfinity}. Then $\langle\tff,\vv\rangle = -v_x\ge 0$ and $|\langle\tzz,\vv\rangle| = |v_z| = \frac1a|\langle\tpp,\vv\rangle|$. Consequently, we have from \eqref{haha5} that
	\[
	\|\ww - \vv\| = \frac{|\langle\tzz,\vv\rangle|}{\|\tzz\|} = \frac{|\langle\tpp,\vv\rangle|}{a\|\tzz\|} = \frac{\|\tpp\|}{a\|\tzz\|}\|\ww - \uu\|.
	\]
	
	\ref{inftyebcase2}: In this case, in view of \eqref{d:Kexp} and \eqref{Fneginfinity}, we have $\vv = (v_x,v_y,v_ye^{v_x/v_y})$ with $v_x\le 0$ and $v_y > 0$. Then $\langle\tff,\vv\rangle = -v_x\ge 0$. Moreover, since $v_y>0$, we have
	\[
	\begin{aligned}
	|\langle\tzz,\vv\rangle| &= |av_y+v_ye^{v_x/v_y}| = av_y+v_ye^{v_x/v_y} \ge \min\{1,a\}(v_y+v_ye^{v_x/v_y})\\
	& \ge \frac{\min\{1,a\}}{\max\{1,a\}}(v_y+av_ye^{v_x/v_y})\ge \frac{\min\{1,a\}}{\max\{1,a\}}|v_y-av_ye^{v_x/v_y}| = \frac{\min\{1,a\}}{\max\{1,a\}}|\langle\tpp,\vv\rangle|.
	\end{aligned}
	\]
	Using \eqref{haha5}, we then obtain that $\|\ww - \vv\|\ge \frac{\min\{1,a\}\|\tpp\|}{\max\{1,a\}\|\tzz\|}\|\ww - \uu\|$.
	
	\ref{inftyebcase3}: In this case, in view of \eqref{d:Kexp} and \eqref{Fneginfinity}, $\vv = (v_x,v_y,v_ye^{v_x/v_y})$ with $v_x> 0$ and $v_y > 0$. Then $\langle\tff,\vv\rangle = -v_x< 0$ and hence $\|\ww - \uu\|= \|\ww\|\le \|\vv\|$, where the equality follows from \eqref{haha5} and the inequality holds because $\ww$ is the projection of $\vv$ onto a subspace. Since $v_y > 0$, we have
	\[
	\begin{aligned}
	|\langle\tzz,\vv\rangle| &= |av_y+v_ye^{v_x/v_y}| = av_y+0.5v_ye^{v_x/v_y} + 0.5v_ye^{v_x/v_y}\\
	& \overset{\rm (a)}\ge av_y+0.5v_y(1 + v_x/v_y) + 0.5v_ye^{v_x/v_y} \ge 0.5v_y + 0.5v_x + 0.5v_ye^{v_x/v_y}= 0.5\|\vv\|_1,
	\end{aligned}
	\]
	where we used $v_y > 0$ and $e^t\ge 1+t$ for all $t$ in (a) and $\|\vv\|_1$ denotes the $1$-norm of $\vv$. Combining this with \eqref{haha5} and the fact that $\|\ww\|\le \|\vv\|$, we see that
	\[
	\|\ww - \vv\| = \frac{|\langle\tzz,\vv\rangle|}{\|\tzz\|}\ge\frac{\|\vv\|_1}{2\|\tzz\|}\ge \frac{\|\vv\|}{2\|\tzz\|} \ge \frac{\|\ww\|}{2\|\tzz\|} = \frac{\|\ww-\uu\|}{2\|\tzz\|}.
	\]
	
	Summarizing the three cases, we conclude that $\gamma_{\zz,\eta} \ge \min\left\{\frac{\|\tpp\|}{a\|\tzz\|},\frac{\min\{1,a\}\|\tpp\|}{\max\{1,a\}\|\tzz\|},\frac{1}{2\|\tzz\|}\right\} > 0$. In view of Theorem~\ref{thm:1dfacesmain}, we have the desired error bound. This completes the proof.
\end{proof}

We next turn to the supporting hyperplane defined by $\zz = (0,0,z_z)$ for some $z_z>0$ and so $\{\zz \}^\perp$ is the $xy$-plane. The following lemma demonstrates that the H\"{o}lderian-type error bound in the form of \eqref{haha2} with $\frakg = |\cdot|^\alpha$ for some $\alpha\in (0,1]$ no longer works in this case.

\begin{lemma}[Nonexistence of H\"{o}lderian error bounds]\label{lem:non_hold}
	Let $\zz\in K_{\exp}^*$ with $z_x= z_y = 0$ and $z_z>0$ so that $\{\zz \}^\perp \cap K_{\exp}=\Fhat_\infty$. Let $\alpha\in (0,1]$ and $\eta > 0$. Then
	\[
	\inf_\qq\left\{\frac{\dist(\qq,K_{\exp})^\alpha}{\dist(\qq,\Fhat_\infty)}\;\bigg|\;\qq\in \{\zz\}^\perp\cap B(\eta)\backslash\Fhat_\infty\right\} = 0.
	\]
\end{lemma}
\begin{proof}
	For each $k\in \mathbb{N}$, let $\qq^k := (-\frac\eta2,\frac\eta{2k},0)$. Then $\qq^k\in \{\zz\}^\perp\cap B(\eta)\backslash\Fhat_\infty$ and we have $\dist(\qq^k,\Fhat_\infty) = \frac\eta{2k}$. Moreover, since $(q^k_x,q^k_y,q^k_ye^{q^k_x/q^k_y})\in K_{\exp}$, we have $\dist(\qq^k,K_{\exp})\le q^k_ye^{q^k_x/q^k_y} = \frac\eta{2k}e^{-k}$. Then it holds that
	\[
	\frac{\dist(\qq^k,K_{\exp})^\alpha}{\dist(\qq^k,\Fhat_\infty)} \le \frac{\eta^{\alpha-1}}{2^{\alpha-1}}k^{1-\alpha}e^{-\alpha k} \to 0 \ \ \ {\rm as}\ \ k\to \infty
	\]
	since $\alpha\in (0,1]$. This completes the proof.
\end{proof}

Since a zero-at-zero monotone nondecreasing function of the form $(\cdot)^\alpha$ no longer works, we opt for the following function $\frakg_\infty:\RP\to \RP$ that grows faster around $t=0$:
\begin{equation}\label{def:frakg}
\frakg_\infty(t) :=  \begin{cases}
0 &\text{if}\;t=0,\\
-\frac{1}{\ln(t)} & \text{if}\;0 < t\leq \frac{1}{e^2},\\
\frac{1}{4}+\frac{1}{4}e^2t & \text{if}\;t>\frac{1}{e^2}.
\end{cases}
\end{equation}
Similar to $\frakg_{-\infty}$ in \eqref{d:entropy}, $\frakg_{\infty}$ is monotone increasing and there exists a constant $\widehat L \ge 1$ such that the following inequalities hold for every $t \in \RP$ and $M > 0$:
\begin{equation}\label{def:frakg_p}
|t| \leq \frakg_{\infty}(t), \quad \frakg_{\infty}(2t) \leq \widehat L\frakg_{\infty}(t),\quad \frakg_{\infty}(Mt) \leq {\widehat L}^{1+|\log_2(M)|}\frakg_{\infty}(t).
\end{equation}
We next show that error bounds in the form of \eqref{haha2} holds for $\zz = (0,0,z_z)$, $z_z>0$, if we use $\frakg_\infty$.
\begin{theorem}[Log-type error bound concerning $\Fhat_{\infty}$]\label{thm:nonzerogammasec72}
	Let $\zz\in \expCone^*$ with $z_x=z_y = 0$ and $z_z>0$ so that $\{\zz \}^\perp \cap \expCone=\Fhat_\infty$.
	Let $\eta > 0$ and let $\gamma_{\zz,\eta}$ be defined as in \eqref{gammabetaeta} with $\frakg = \frakg_\infty$ in \eqref{def:frakg}. Then $\gamma_{\zz,\eta} \in (0,\infty]$ and
	\[
	\dist(\qq,\Fhat_{\infty})\le \max\{2,2\gamma_{\zz,\eta}^{-1}\}\cdot\frakg_\infty(\dist(\qq,K_{\exp}))\ \ \ \mbox{whenever $\qq\in \{\zz\}^\perp\cap B(\eta)$.}
	\]
\end{theorem}
\begin{proof}
	Take $\bv \in \Fhat_\infty$ and a sequence $\{\vv^k\}\subset \bd K_{\exp}\cap B(\eta) \backslash \Fhat_\infty$ such that
	\begin{equation*}
	\underset{k \rightarrow \infty}{\lim}\vv^k = \underset{k \rightarrow \infty}{\lim}\ww^k =\bv,
	\end{equation*}
	where $\ww^k = P_{\{\zz\}^\perp}\vv^k$, $\uu^k = P_{\Fhat_\infty}\ww^k$, and $\ww^k\neq \uu^k$.
	Since $\ww^k\neq \uu^k$, in view of \eqref{Finfinity} and \eqref{Fneginfinity}, we must have $\vv^k\notin \Fhat_{-\infty}$. Then, from \eqref{d:Kexp} and \eqref{Finfinity}, we have
	\begin{equation}\label{vwurepre}
	\vv^k = (v^k_x,v^k_y,v^k_ye^{v^k_x/v^k_y}) \mbox{ with } v^k_y>0,\ \ \ww^k = (v^k_x,v^k_y,0)\ \ {\rm and}\ \ \uu^k = (\min\{v^k_x,0\},0,0).
	\end{equation}
	Since $\ww^k\to \bv$ and $\vv^k\to \bv$, without loss of generality, by passing to a subsequence if necessary, we assume in addition that $\|\ww^k - \vv^k\|\le e^{-2}$ for all $k$.
	From \eqref{vwurepre} we conclude that $\vv^k \neq \ww^k$, hence $\frakg_\infty(\|\ww^k - \vv^k\|) = -(\ln\|\ww^k - \vv^k\|)^{-1}$.

	We consider the following two cases in order to show that \eqref{infratiohkb} does not hold for $\frakg = \frakg_\infty$:
	\begin{enumerate}[(I)]
		\item\label{finalebcasevneq0} $\bv\neq \mathbf{0}$;
		\item\label{finalebcasev0} $\bv = \mathbf{0}$.
	\end{enumerate}
	
	\ref{finalebcasevneq0}: In this case, we have $\bv = ({\bar{v}}_x,0,0)$ for some ${\bar{v}}_x < 0$. This implies that $v^k_x<0$ for all large $k$.  Hence, we have from \eqref{vwurepre} that for all large $k$,
	\begin{equation*}
	\frac{\frakg_\infty(\|\ww^k - \vv^k\|) }{\|\ww^k - \uu^k \|} = \frac{-(\ln\|\ww^k-\vv^k\|)^{-1}}{\|\ww^k - \uu^k\|} = -\frac{1}{v^k_y\left(v^k_x/v^k_y + \ln v^k_y \right)} = -\frac{1}{v^k_y\ln v^k_y + v^k_x} \to -\frac1{{\bar{v}}_x} > 0
	\end{equation*}
	since $v^k_y\to 0$ and $v^k_x\to {\bar{v}}_x < 0$. This shows that \eqref{infratiohkb} does not hold for $\frakg = \frakg_\infty$.
	
	\ref{finalebcasev0}: If $v_x^k\le 0$ infinitely often, by extracting a subsequence, we assume that $v_x^k\le 0$ for all $k$. Since $\ww^k\neq \uu^k$ (and $\ww^k\neq \vv^k$), we note from \eqref{vwurepre} that
	\[
	-\frac{1}{v^k_y\ln v^k_y + v^k_x} = \frac{-(\ln\|\ww^k-\vv^k\|)^{-1}}{\|\ww^k - \uu^k\|} \in (0,\infty)\ \ \mbox{for all}\ k.
	\]
	Since $\{-(v^k_y\ln v^k_y + v^k_x)\}$ is a positive sequence and it converges to zero as $(v^k_x,v^k_y)\to 0$, it follows that $\lim\limits_{k\to\infty}\frac{-(\ln\|\ww^k-\vv^k\|)^{-1}}{\|\ww^k - \uu^k\|}=\infty$. This shows that \eqref{infratiohkb} does not hold for $\frakg = \frakg_\infty$.
	
	Now, it remains to consider the case that $v_x^k>0$ for all large $k$. By passing to a subsequence if necessary, we assume that $v_x^k>0$ for all $k$. By solving for $v_x^k$ from $v^k_z=v^k_y e^{v^k_x/v^k_y} > 0$ and noting \eqref{vwurepre}, we obtain that
	\begin{equation}\label{vwurepre2}
	\vv^k = (v^k_y\ln(v_z^k/v_y^k),v^k_y,v^k_z) \mbox{ with } v^k_y>0,\ \ \ww^k = (v^k_y\ln(v_z^k/v_y^k),v^k_y,0)\ \ {\rm and}\ \ \uu^k = (0,0,0).
	\end{equation}
	Also, we note from $v_x^k=v^k_y\ln(v_z^k/v_y^k)>0$, $v_y^k>0$ and the monotonicity of $\ln(\cdot)$ that for all $k$,
	\begin{equation}\label{bugfix}
	v^k_z>v^k_y>0.
	\end{equation}
	Next consider the function $h(t) := \frac{1}t\sqrt{1+(\ln t)^2}$ on $[1,\infty)$. Then $h$ is continuous and positive. Since $h(1)=1$ and $\lim_{t\to \infty}h(t) = 0$, there exists $M_h$ such that $h(t)\le M_h$ for all $t\ge 1$. Now, using \eqref{vwurepre2}, we have, upon defining $t_k:= v_z^k/v_y^k$ that
	\[
	\begin{aligned}
	\frac{\|\ww^k - \uu^k\|}{-(\ln\|\ww^k-\vv^k\|)^{-1}} &= \frac{v_y^k\sqrt{1+[\ln(v_z^k/v_y^k)]^2}}{-(\ln v_z^k)^{-1}} = -v_y^k\sqrt{1+[\ln(v_z^k/v_y^k)]^2}\ln v_z^k\\
	& \overset{\rm (a)}= -\frac{v_y^k}{v_z^k}\sqrt{1+\left[\ln\left(\frac{v_z^k}{v_y^k}\right)\right]^2}v_z^k\ln v_z^k \overset{\rm (b)}= - h(t_k)v_z^k\ln v_z^k \overset{\rm (c)}\le -M_hv_z^k\ln v_z^k,
	\end{aligned}
	\]
	where the division by $v_z^k$ in (a) is legitimate because $v_z^k>0$, (b) follows from the definition of $h$ and the fact that $t_k > 1$ (see \eqref{bugfix}), and (c) holds because of the definition of $M_h$ and the fact that $-\ln v_z^k > 0$ (thanks to $v_z^k = \|\ww^k - \vv^k\|\le e^{-2}$). Since $v^k_z\to 0$, it then follows that $\left\{\frac{\|\ww^k - \uu^k\|}{-(\ln\|\ww^k-\vv^k\|)^{-1}}\right\}$ is a positive sequence that converges to zero. Thus, $\lim\limits_{k\to\infty}\frac{-(\ln\|\ww^k-\vv^k\|)^{-1}}{\|\ww^k - \uu^k\|}=\infty$, which again shows that \eqref{infratiohkb} does not hold for $\frakg = \frakg_\infty$.

	Having shown that \eqref{infratiohkb} does not hold for $\frakg = \frakg_\infty$, in view of Lemma~\ref{lem:infratio}, we must have $\gamma_{\zz,\eta} \in (0,\infty]$. Then the result follows from Theorem~\ref{thm:1dfacesmain} and \eqref{def:frakg_p}.
\end{proof}

Combining Theorem~\ref{thm:nonzerogammasec71}, Theorem~\ref{thm:nonzerogammasec72} and Lemma~\ref{lem:facialresidualsbeta}, and noting \eqref{def:frakg_p} and $\gamma_{\zz,0}=\infty$ (see \eqref{gammabetaeta}), we can now summarize the {\oneFRF}s derived in this section in the following corollary.
\begin{corollary}[{\oneFRFs} concerning $\Fhat_{\infty}$]\label{col:1dfaces_infty}
	Let $\zz\in K_{\exp}^*$ with $z_x=0$ and $\{\zz \}^\perp \cap K_{\exp}=\Fhat_\infty$.
	\begin{enumerate}[{\rm (i)}]
		\item \label{lem:facialresidualsbeta:1}In the case when $z_y > 0$, let $\kappa_{\zz,t}$ be defined as in \eqref{haha2} with $\frakg = |\cdot|$. Then the function $\psi_{\stdCone,\zz}:\RR_+\times\RR_+\to \RR_+$ given by
		\begin{equation*}
		\psi_{\stdCone,\zz}(\epsilon,t):=\max \left\{\epsilon,\epsilon/\|\zz\| \right\} + \kappa_{\zz,t}(\epsilon +\max \left\{\epsilon,\epsilon/\|\zz\| \right\} )
		\end{equation*}
		is a {\oneFRFs} for $K_{\exp}$ and $\zz$.
In particular, there exist $\kappa > 0$ and a nonnegative monotone nondecreasing function $\rho:\RR_+ \to \RR_+$ such that the function $\hat \psi_{\stdCone,\zz}$ given by $\hat \psi_{\stdCone,\zz}(\epsilon,t) \coloneqq \kappa \epsilon + \rho(t)\epsilon$ is a {\oneFRFs} for $K_{\exp}$ and $\zz$.
		\item \label{lem:facialresidualsbeta:2} In the case when $z_y = 0$, let $\kappa_{\zz,t}$ be defined as in \eqref{haha2} with $\frakg = \frakg_\infty$ in \eqref{def:frakg}. Then the function $\psi_{\stdCone,\zz}:\RR_+\times\RR_+\to \RR_+$ given by
		\begin{equation*}
		\psi_{\stdCone,\zz}(\epsilon,t):=\max \left\{\epsilon,\epsilon/\|\zz\| \right\} + \kappa_{\zz,t}\frakg_\infty(\epsilon +\max \left\{\epsilon,\epsilon/\|\zz\| \right\} )
		\end{equation*}
		is a {\oneFRFs} for $K_{\exp}$ and $\zz$. In particular,
	there exist $\kappa > 0$ and a nonnegative monotone nondecreasing function $\rho:\RR_+ \to \RR_+$ such that the function $\hat \psi_{\stdCone,\zz}$ given by $\hat \psi_{\stdCone,\zz}(\epsilon,t) \coloneqq \kappa \epsilon + \rho(t)\frakg_{\infty}(\epsilon)$ is a {\oneFRFs}   for $K_{\exp}$ and $\zz$.
	\end{enumerate}
\end{corollary}

\subsubsection{The non-exposed face $\Fhat_{ne}$}

Recall the unique non-exposed face of $K_{\exp}$:
$$
\stdFace_{ne} := \{(0,0,z) \;|\; z \geq 0\}.
$$
In this subsection, we take a look at $\Fhat_{ne}$. Note that $\Fhat_{ne}$ is an exposed face of $\Fhat_{-\infty}$, which is polyhedral. This observation leads immediately to the following corollary, which also follows from  \cite[Proposition~18]{L17} by letting $\stdFace\coloneqq\stdCone \coloneqq\Fhat_{-\infty}$ therein. We omit the proof for brevity.
\begin{corollary}[{\oneFRFs} for $\Fhat_{ne}$]\label{col:frf_ne}
Let $\zz \in \ \Fhat_{-\infty}^*$ be such that
$\Fhat_{ne} = \Fhat_{-\infty} \cap \{\zz\}^\perp$.
Then there exists $\kappa > 0$ such that
\[
\psi _{\Fhat_{-\infty},\zz }(\epsilon,t) \coloneqq
\kappa\epsilon
\]
is a {\oneFRFs} for $\Fhat_{-\infty}$ and
$\zz$.
\end{corollary}	

\subsection{Error bounds}\label{sec:exp_err}
In this subsection, we return to the feasibility problem \eqref{eq:feas} and consider the case
where $\stdCone = \expCone$. We now have all the
tools for obtaining error bounds.
Recalling Definition~\ref{def:hold}, we can state the following result.

\begin{theorem}[Error bounds for \eqref{eq:feas} with $\stdCone = \expCone$]\label{thm:main_err}
Let $\stdSpace \subseteq \RR^3$ be a subspace and $\abd \in \RR^3$ such that $(\stdSpace + \abd) \cap \expCone \neq \emptyset$.
Then the following items hold.
\begin{enumerate}
	\item The distance to the PPS condition of $\{\expCone, \stdSpace+\abd\}$ satisfies $\dpp(\expCone,\stdSpace+\abd) \leq 1$.
	\item\label{thm:mainii} If $\dpp(\expCone,\stdSpace+\abd)=0 $,
	then $\expCone$ and $\stdSpace+\abd$ satisfy a Lipschitzian error bound.
	\item\label{thm:mainiii} Suppose $\dpp(\expCone,\stdSpace+\abd)=1$ and
	let $\stdFace \subsetneq \expCone$ be a chain of faces of length $2$ satisfying items {\rm(ii)} and {\rm(iii)} of Proposition~\ref{prop:fra_poly}. We have the following possibilities.
	\begin{enumerate}[label={\rm (\alph*)}]		
		\item\label{thm:mainiiia} If $\stdFace = \Fhat_{-\infty}$ then $\expCone$ and $\stdSpace+\abd$ satisfy an entropic error bound as in \eqref{eq:entropic_err}.
		In addition, for all $\alpha \in (0,1)$, a	uniform H\"olderian error bound with exponent $\alpha$ holds.
		\item\label{thm:mainiiib}  If $\stdFace = \Fhat_{\beta}$, with
		$\beta \in \RR$, then  $\expCone$ and $\stdSpace+\abd$ satisfy a uniform H\"olderian error bound with exponent $1/2$.
		\item\label{thm:mainiiic}  Suppose that $\stdFace = \Fhat_{\infty}$.
		If 	there exists $\zz\in \expCone^* \cap \stdSpace^\perp \cap \{a\}^\perp$ with $z_x=0$, $z_y > 0$ and $z_z>0$ 		
		then  $\expCone$ and $\stdSpace+\abd$ satisfy a Lipschitzian error bound. Otherwise, $\expCone$ and $\stdSpace+\abd$ satisfy a log-type error bound as in \eqref{eq:log_err}.
		\item \label{thm:mainivc} If $\stdFace = \{\mathbf{0} \}$, then $\expCone$ and $\stdSpace+\abd$ satisfy a Lipschitzian error bound.
	\end{enumerate}
\end{enumerate}
\end{theorem}
\begin{proof}
(i): All proper faces of $\expCone$ are polyhedral, therefore $\ell_{\text{poly}}(\expCone) = 1$. By item \ref{prop:fra_poly:1} of Proposition~\ref{prop:fra_poly}, there exists a
chain of length $2$ satisfying item \ref{prop:fra_poly:3} of Proposition~\ref{prop:fra_poly}. Therefore, $\dpp(\expCone,\stdSpace+\abd) \leq 1$.

(ii): If  $\dpp(\expCone,\stdSpace+\abd) = 0$, it is because $\{\expCone, \stdSpace+\abd\}$ satisfies the PPS condition, which implies a Lipschitzian error bound
by Proposition~\ref{prop:cq_er}.

(iii): Next,  suppose $\dpp(\expCone,\stdSpace+\abd)=1$ and
let $\stdFace \subsetneq \expCone$ be a chain of faces of length $2$ satisfying items \ref{prop:fra_poly:2} and \ref{prop:fra_poly:3} of Proposition~\ref{prop:fra_poly}, together with  $\zz\in K_{\exp}^* \cap \stdSpace^\perp \cap \{\abd\}^\perp$ such that
\[
\stdFace = \expCone \cap \{\zz\}^\perp.
\]
Since positively scaling $\zz$ does not affect the chain of faces, we may assume that $\norm{\zz} = 1$.
Also, in what follows, for simplicity, we define
\[
\wdist(\xx) \coloneqq \max \{\dist(\xx,\stdSpace+\abd), \dist(\xx, \expCone) \}.
\]
Then, we prove each item by applying Theorem~\ref{theo:err} with the corresponding facial residual function.
\begin{enumerate}[{\rm (a)}]	
\item 	 If $\stdFace = \Fhat_{-\infty}$, the
{\oneFRF}s are given by Corollary~\ref{col:frf_2dface_entropic}. First we consider the case where $\frakg = \frakg_{-\infty}$ and we have
\begin{equation*}
\psi_{\stdCone,\zz}(\epsilon,t):= \epsilon + \kappa_{\zz,t}\frakg_{-\infty}(2\epsilon ),
\end{equation*}
where $\frakg_{-\infty}$ is as in \eqref{d:entropy}.
Then, if
$\psi$ is a positively rescaled shift of
$\psi_{\stdCone,\zz}$, using the monotonicity
of $\frakg_{-\infty}$ and of $\kappa_{\zz,t}$ as a function of $t$, we conclude that
there exists $\widehat M  > 0$ such that
\begin{equation*}
\psi(\epsilon,t) \leq \widehat M \epsilon + \widehat M \kappa_{\zz,  \widehat M t}\frakg_{-\infty}(\widehat M\epsilon ).
\end{equation*}

Invoking Theorem~\ref{theo:err}, using the monotonicity of all functions involved in the definition of $\psi$ and recalling \eqref{d:entropy_p},
we conclude that for every bounded set $B$, there exists $\kappa _B > 0$
\begin{equation}\label{eq:entropic_err}
\dist\left(\xx, (\stdSpace + \abd) \cap \expCone\right) \leq \kappa_{B}\frakg_{-\infty}(\wdist(\xx)), \qquad \forall \xx \in B,
\end{equation}
which shows that an entropic error bound holds.

Next, we consider the case $\frakg = |\cdot|^{\alpha}$.
Given $\alpha \in (0,1)$, we have the following {\oneFRF}:
\begin{equation*}
\psi_{\stdCone,\zz}(\epsilon,t):= \epsilon + \kappa_{\zz,t}2^\alpha \epsilon^\alpha,
\end{equation*}
where  $\kappa_{\zz,t}$ is defined as in \eqref{haha2}.
Invoking Theorem~\ref{theo:err},
we conclude that for
every bounded set $B$, there exists $\kappa _B > 0$ such that
\[
\dist\left(\xx, (\stdSpace + \abd) \cap \expCone\right) \leq \kappa _B \wdist(\xx) +  \kappa_{B} \wdist(\xx)^\alpha, \qquad \forall \xx \in B,
\]
In addition, for $\xx \in B$, we have
$\wdist(\xx) \leq \wdist(\xx)^\alpha  {M}$, where
$M = \sup _{\xx \in B} \wdist(\xx)^{1-\alpha}$.
In conclusion, for  $\kappa = 2\kappa_{B}\max\{M,1\}$, we have
\[
\dist\left(\xx, (\stdSpace + \abd) \cap \expCone\right) \leq \kappa \wdist(\xx)^\alpha, \qquad \forall \xx \in B.
\]
That is, a uniform H\"olderian error bound holds with exponent $\alpha$.
\item  If $\stdFace = \Fhat_{\beta}$, with
$\beta \in \RR$, then the {\oneFRF} is given by Corollary~\ref{col:frf_fb}, that is, we have
\begin{equation*}
\psi_{\stdCone,\zz}(\epsilon,t) := \epsilon + \kappa_{\zz,t}\sqrt{2} \epsilon^{1/2},
\end{equation*}
Then, following the same argument as in the second half of item (a), we conclude that a uniform H\"olderian error bound holds with exponent $1/2$.
\item If $\stdFace = \Fhat_{\infty}$, the {\oneFRF}s are given by Corollary~\ref{col:1dfaces_infty} and they depend on $\zz$. Since $\stdFace = \Fhat_{\infty}$, we must have $z_x = 0$ and $z_z > 0$, see Section~\ref{sec:exposed}.

The deciding factor is whether $z_y$ is positive or zero. If $z_y > 0$, then we have
the following {\oneFRF}:
\begin{equation*}
\psi_{\stdCone,\zz}(\epsilon,t):=
(1+2 \kappa_{\zz,t})\epsilon,
\end{equation*}
where $\kappa_{\zz,t}$ is defined as in \eqref{haha2}.
In this case, analogously to items (a) and (b) we have a Lipschitzian error bound.

If $z_y = 0$,
we have
\begin{equation*}
\psi_{\stdCone,\zz}(\epsilon,t):= \epsilon + \kappa_{\zz,t}\frakg_\infty(2\epsilon ),
\end{equation*}
where $\frakg_\infty$ is as in \eqref{def:frakg}.
Analogous to the proof of item (a) but making use of \eqref{def:frakg_p} in place of \eqref{d:entropy_p},
we conclude that for
every bounded set $B$, there exists $\kappa _B > 0$ such that
\begin{equation}\label{eq:log_err}
\dist\left(\xx, (\stdSpace + \abd) \cap \expCone\right) \leq \kappa_{B}\frakg_\infty(\wdist(\xx))), \qquad \forall \xx \in B.
\end{equation}
\item See \cite[Proposition~27]{L17}.
\end{enumerate}
\end{proof}
\begin{remark}[Tightness of Theorem~\ref{thm:main_err}]\label{rem:opt}
We will argue that Theorem~\ref{thm:main_err} is tight by showing that for every situation described in item (iii), there is a specific choice of $\stdSpace$ and a sequence $\{\ww^k\}$ in $\stdSpace\backslash\expCone$ with $\dist(\ww^k,\expCone) \to 0$ along which the corresponding
error bound for $\expCone$ and $\stdSpace$ is off by at most a multiplicative constant.
	
\begin{enumerate}[{\rm (a)}]
\item\label{opt:a} Let $\stdSpace = \spanVec \Fhat_{-\infty} = \{(x,y,z) \mid y = 0 \}$ (see \eqref{eq:exp_2d}) and consider the sequence $\{\ww^k\}$ where $\ww^k = ((1/(k+1))\ln(k+1),0,1)$, for every $k \in \NN$. Then, $\stdSpace \cap \expCone = \Fhat_{-\infty}$ and we are under the conditions of item~\ref{thm:mainiii}\ref{thm:mainiiia} of Theorem~\ref{thm:main_err}. Since $\{\ww^k\} =: B \subseteq \stdSpace$, there exists $\kappa_B > 0$ such that
\[
\dist\left(\ww^k, \stdSpace \cap \expCone\right) \leq \kappa_{B}\frakg_{-\infty}(\dist(\ww^k, \expCone)), \quad \forall k \in \NN.
\]
Then, the projection of $\ww^k$ onto
$\Fhat_{-\infty}$ is given by $(0,0,1)$.
 Therefore,
\[
\frac{\ln(k+1)}{k+1} = \dist(\ww^k,\stdSpace\cap \expCone) \leq \kappa_{B}\frakg_{-\infty}(\dist(\ww^k, \expCone)).
\]
Let $\vv^k = ((1/(k+1))\ln(k+1),1/(k+1),1)$ for every $k$. Then, we have $\vv^k \in \expCone$.
Therefore, $\dist(\ww^k, \expCone) \leq 1/(k+1)$.
In view of the definition of $\frakg_{-\infty}$ (see \eqref{d:entropy}), we conclude that for large enough $k$ we have
\[
\frac{\ln(k+1)}{k+1} = \dist(\ww^k,\stdSpace\cap \expCone) \leq \kappa_{B}\frakg_{-\infty}(\dist(\ww^k, \expCone)) \leq \kappa_B\frac{\ln(k+1)}{k+1}.
\]
Thus, it holds that for all sufficiently large $k$,
\[
1\le \frac{\dist(\ww^k,\stdSpace\cap \expCone)}{\frakg_{-\infty}(\dist(\ww^k,\expCone))} \le \kappa_B.
\]
Consequently, for any given nonnegative function $\frakg:\RR_+\to \RR_+$ such that $\lim_{t\downarrow 0}\frac{\frakg(t)}{\frakg_{-\infty}(t)}=0$, we have upon noting $\dist(\ww^k,\expCone)\to 0$ that
\[
\frac{\dist(\ww^k,\stdSpace\cap \expCone)}{\frakg(\dist(\ww^k,\expCone))} = \frac{\dist(\ww^k,\stdSpace\cap \expCone)}{\frakg_{-\infty}(\dist(\ww^k,\expCone))}\frac{\frakg_{-\infty}(\dist(\ww^k,\expCone))}{\frakg(\dist(\ww^k,\expCone))} \to \infty,
\]
which shows that the choice of $\frakg_{-\infty}$ in the error bound is tight.
\item Let $\beta \in \RR$ and let $\hzz$, $\hpp$ and
$\hff$ be as in \eqref{veczfp}. Let
$\stdSpace = \{\zz\}^\perp$ with $z_x < 0$ such that
$\expCone \cap \stdSpace = \Fhat_{\beta}$. We are then under the conditions of item~\ref{thm:mainiii}\ref{thm:mainiiib} of Theorem~\ref{thm:main_err}.
We consider the following sequences
\begin{equation*}
\vv^k = \begin{bmatrix}
1-\beta +1/k\\ 1\\ e^{1-\beta + 1/k}
\end{bmatrix},\quad \ww^k = P_{\{\zz\}^\perp}\vv^k,\quad \uu^k = P_{\Fhat_\beta}\ww^k.
\end{equation*}
For every $k$ we have $\vv^k \in \partial \expCone \setminus \Fhat_{\beta}$, and $\vv^k \neq \ww^k$ (because otherwise, we would have $\vv^k \in \expCone \cap \{\zz\}^\perp = \Fhat_{\beta}$).
In addition, we have $\vv^k \to \hff$ and, since $\hff \in \Fhat_{\beta}$, we have  $\ww^k \to \hff$ as well.

Next, notice that we have $\langle \hff, \vv^k \rangle \geq 0$ for $k$ sufficiently large and $|v_x^k/v_y^k - (1-\beta)| \rightarrow 0$. Then, following the computations outlined in case~\ref{ebsubcasev0} of the proof of Theorem~\ref{thm:nonzerogamma} and letting $\zeta_k\coloneqq v_x^k/v_y^k$, we have from \eqref{hahahehe1} and \eqref{newlyadded} that $h_2(\zeta_k)\neq 0$ for all large $k$ (hence, $\ww^k\neq \uu^k$ for all large $k$), and that
\begin{equation}\label{eq:beta_limit}
L_{\beta} \coloneqq \lim_{k \rightarrow \infty}\frac{\|\ww^k-\vv^k\|^{\frac{1}{2}}}{\|\ww^k-\uu^k\|}=\lim_{k \rightarrow \infty}\frac{\|\hpp\|}{\|\hzz\|^{\frac{1}{2}}}\frac{|h_1(\zeta_k)|^{\frac{1}{2}}}{|h_2(\zeta_k)|} = \frac{\|\hpp\|}{\|\hzz\|^{\frac{1}{2}}}\frac1{\sqrt{2}(e^{\beta-1} + (\beta^2+1)e^{1-\beta})} \in(0,\infty),
\end{equation}
where the latter equality is from \eqref{taylor_limit}.
On the other hand, from item~\ref{thm:mainiii}\ref{thm:mainiiib} of Theorem~\ref{thm:main_err}, for $B \coloneqq \{\ww^k\}$, there exists $\kappa_B > 0$ such that for all $k\in \NN$,
\[
\|\ww^k-\uu^k\| = \dist(\ww^k, \stdSpace\cap\expCone) \leq \kappa_B \dist(\ww^k,\expCone)^{\frac{1}{2}} \leq \kappa_B\|\ww^k-\vv^k\|^{\frac{1}{2}}.
\]
However, from \eqref{eq:beta_limit}, for large enough $k$, we have $\|\ww^k-\uu^k\| \geq 1/(2L_{\beta})\|\ww^k-\vv^k\|^{\frac{1}{2}}$.
Therefore, for large enough $k$ we have
\[
\frac{1}{2L_{\beta}}\|\ww^k-\vv^k\|^{\frac{1}{2}} \leq \dist(\ww^k, \stdSpace\cap\expCone)\leq \kappa_B \dist(\ww^k,\expCone)^{\frac{1}{2}} \leq  \kappa_B\|\ww^k-\vv^k\|^{\frac{1}{2}}.
\]
Consequently, it holds that for all large enough $k$,
\[
\frac{1}{2L_\beta}\le \frac{\dist(\ww^k,\stdSpace\cap \expCone)}{\dist(\ww^k,\expCone)^\frac12} \le \kappa_B.
\]
Arguing similarly as in case (a), we can also conclude that the choice of $|\cdot|^\frac12$ in the error bound is tight.
\item 	Let $\zz= (0,0,1)$ and $\stdSpace = \{(x,y,0) \mid x,y \in \RR \} = \{\zz\}^\perp$.
Then, from \eqref{Finfinity}, we have
$\stdSpace \cap \expCone = \Fhat_{\infty}$.
We are then under case \ref{thm:mainiii}\ref{thm:mainiiic} of Theorem~\ref{thm:main_err}.
Because there is no $\hat \zz \in \stdSpace^\perp $ with $\hat z _y > 0$, we have a log-type error bound as in \eqref{eq:log_err}.

We proceed as in item \ref{opt:a} using
sequences such that $\ww^k=(-1,1/k,0)$,
$\vv^k=(-1,1/k,(1/k)e^{-k})$, $\uu^k=(-1,0,0)$, for every $k$.
Note that $\ww^k \in \stdSpace, \vv^k \in \expCone$ and $\proj{\Fhat_\infty}(\ww^k) = \uu^k$, for every $k$. Therefore, there exists $\kappa_B > 0$ such that
\begin{equation*}
\frac{1}{k} = \dist(\ww^k, \stdSpace \cap \expCone) \leq \kappa_B \frakg_{\infty}(\dist(\ww^k,\expCone))\leq \kappa_{B}\frakg_{\infty}\left(\frac{1}{ke^k}\right), \quad \forall k \in \NN.
\end{equation*}
In view of the definition of $\frakg_{\infty}$ (see \eqref{def:frakg}), there exists $L > 0$ such that for large enough $k$ we have
\begin{equation*}
\frac{1}{k} = \dist(\ww^k, \stdSpace \cap \expCone) \le \kappa_B \frakg_{\infty}(\dist(\ww^k,\expCone)) \leq \frac{L}{k}.
\end{equation*}
Consequently, it holds that for all large enough $k$,
\[
\frac{\kappa_B}{L}\le \frac{\dist(\ww^k,\stdSpace\cap \expCone)}{\frakg_{\infty}(\dist(\ww^k,\expCone))} \le \kappa_B.
\]
Arguing similarly as in case (a), we conclude that the choice of $\frakg_{\infty}$ is tight.
\end{enumerate}
Note that a Lipschitz error bound is always tight  up to a constant, because
$\dist(\xx,\stdCone\cap (\stdSpace+\abd)) \geq \max\{\dist(\xx,\stdCone),\dist(\xx,\stdSpace+\abd)\}$. Therefore, the error bounds in items~\ref{thm:mainii}, \ref{thm:mainiii}\ref{thm:mainivc} and in the first half of \ref{thm:mainiii}\ref{thm:mainiiic} are tight.
\end{remark}

Sometimes we may need to consider direct products of multiple copies of $\expCone$ in order to model certain problems, i.e., our problem of interest could have the following shape:
\begin{equation*}
\text{find} \quad \xx \in (\stdSpace + \abd) \cap \stdCone, \label{eq:multiple_exp}
\end{equation*}
where $\stdCone = \expCone \times \cdots \times \expCone$ is a direct product of $m$ exponential cones.

Fortunately, we already have all the tools required to extend Theorem~\ref{thm:main_err} and compute error bounds for this case too. We recall that the faces of a direct product of cones are direct products of the faces of the individual cones.\footnote{Here is a sketch of the proof. If $\stdFace^1 \face \stdCone^1, \stdFace^{2} \face \stdCone^2$, then the definition of face implies that $\stdFace^1 \times \stdFace^{2} \face \stdCone^1 \times \stdCone^2$. For the converse, let $\stdFace \face \stdCone^1 \times \stdCone^2$ and let
	$\stdFace ^1, \stdFace^2$ be  the projection of $\stdFace$ onto the first variable and second variables, respectively. Suppose that
	$\xx,\yy \in \stdCone^1$ are such that $\xx+\yy \in \stdFace^1$. Then, $(\xx+\yy,\zz) \in \stdFace$ for some $\zz \in \stdCone^2$. Since $(\xx+\yy,\zz) = (\xx,\zz/2) + (\yy,\zz/2)$ and $\stdFace$ is a face, we conclude that   $(\xx,\zz/2), (\yy,\zz/2) \in \stdFace$ and $\xx, \yy \in \stdFace^1$. Therefore $\stdFace^1 \face \stdCone^1$ and, similarly, $\stdFace^2 \face \stdCone^2$. Then, the equality $\stdFace = \stdFace^1 \times \stdFace^2$ is proven using the definition of face and the fact that $(\xx,\zz) = (\xx,0) + (0,\zz)$.
} Therefore, using Proposition \ref{prop:frf_prod}, we are able to compute all the necessary {\oneFRF}s for $\stdCone$. Once they are obtained we can invoke Theorem~\ref{theo:err}.
Unfortunately, there is quite a number of different cases one must consider, so we cannot give a concise statement of an all-encompassing tight error bound result.

We will, however, given an error bound result under the following \emph{simplifying assumption of non-exceptionality} or SANE.
\begin{assumption}[SANE: simplifying assumption of non-exceptionality]
Suppose \eqref{eq:feas} is feasible with $\stdCone = \expCone \times \cdots \times \expCone$ being a direct product of $m$ exponential cones. We say that $\stdCone$ and $\stdSpace+\abd$ satisfy  the \emph{simplifying assumption of non-exceptionality} (SANE) if there exists a chain of faces $\stdFace _{\ell}  \subsetneq \cdots \subsetneq \stdFace_1 = \stdCone $ as in Proposition~\ref{prop:fra_poly}  with $\ell - 1 = {\dpp(\stdCone,\stdSpace+\abd)}$
such that for all $i$, the exceptional face $\Fhat_{\infty}$ of $\expCone$ never appears as one of the blocks of $\stdFace_{i}$.
\end{assumption}
\begin{remark}[SANE is not unreasonable]\label{rem:sane}
In many modelling applications of the exponential cone presented in \cite[Chapter~5]{MC2020}, translating to our notation, the $\yy$ variable is fixed to be $1$ in \eqref{d:Kexp}. For example, the hypograph of the logarithm function ``$x \leq \ln(z)$'' can be represented as the constraint ``$(x,y,z) \in \expCone \cap (\stdSpace +\abd)$'', where $\stdSpace +\abd = \{(x,y,z) \mid y = 1\}$. Because the $y$ variable is fixed to be $1$, the feasible region does not intersect the 2D face $\Fhat_{-\infty}$ nor its subfaces  $\Fhat_{\infty}$ and $\Fhat_{ne}$. In particular, SANE is satisfied.
More generally, if $\stdCone$ is a direct product of exponential cones and the affine space $\stdSpace +\abd$ is such that the $\yy$ components of each block are fixed positive constants, then $\stdCone$ and $\stdSpace +\abd$ satisfy SANE.

On the other hand, problems involving the relative entropy $D(x,y) \coloneqq x \ln(x/y)$ are often modelled as
``minimize $t$'' subject to ``$(-t,x,y)  \in \expCone$'' and additional constraints. We could also have sums so that the problem is of the form
``minimize $\sum t_i$'' subject to ``$(-t_i,x_i,y_i)  \in \expCone$'' and additional constraints.
In those cases, it seems that it could happen that SANE is not satisfied.
\end{remark}

Under SANE, we can state the following result.
\begin{theorem}[Error bounds for direct products of exponential cones]\label{theo:sane}
Suppose \eqref{eq:feas} is feasible with $\stdCone = \expCone \times \cdots \times \expCone$ being a direct product of $m$ exponential cones. Then the following hold.
\begin{enumerate}
	\item  The distance to the PPS condition of  $\{\stdCone, \stdSpace+\abd\}$ satisfies $\dpp(\stdCone,\stdSpace+\abd) \leq m$.
	\item If SANE is satisfied, then $\stdCone$ and $\stdSpace+\abd$ satisfy a uniform H\"olderian error bound
	with exponent
	$2^{-\dpp(\expCone,\stdSpace+\abd)}$.	
\end{enumerate}
\end{theorem}
\begin{proof}
(i):  All proper faces of $\expCone$ are polyhedral, therefore $\ell_{\text{poly}}(\expCone) = 1$. By item \ref{prop:fra_poly:1} of Proposition~\ref{prop:fra_poly}, there exists a
chain of length $\ell$ satisfying item \ref{prop:fra_poly:3} of Proposition~\ref{prop:fra_poly} such that
$\ell-1 \leq m$. Therefore, $\dpp(\stdCone,\stdSpace+\abd)\leq \ell-1 \leq m$.

(ii): If SANE is satisfied, then there exists a
chain $\stdFace _{\ell}  \subsetneq \cdots \subsetneq \stdFace_1 = \stdCone $ of length $\ell \leq m +1$  as in Proposition~\ref{prop:fra_poly}, together with the corresponding $\zz_{1},\ldots,\zz_{\ell-1}$. Also, the exceptional face $\Fhat_{\infty}$ never appears as one of the blocks of the $\stdFace _i$.

In what follows, for simplicity, we define
\[
\wdist(\xx) \coloneqq \max \{\dist(\xx,\stdSpace+\abd), \dist(\xx, \stdCone) \}.
\]
Then, we invoke Theorem~\ref{theo:err}, which implies that  given a bounded set $B$, there exists a
constant $\kappa _B > 0$ such that
\begin{equation}\label{eq:bound_sane}
\dist\left(\xx, (\stdSpace + \abd) \cap \stdCone\right) \leq \kappa _B (\wdist(\xx)+\varphi(\wdist(\xx),M)),
\end{equation}
where $M = \sup _{\xx\in B} \norm{\xx}$ and
there are two cases for $\varphi$.
If $\ell = 1$, $\varphi$ is the function such
that $\varphi(\epsilon,M) = \epsilon$.
If $\ell \geq 2$, we have $\varphi = \psi _{{\ell-1}}\comp \cdots \comp \psi_{{1}}$, where $\psi _{i}$ is a (suitable positively rescaled shift of a) {\oneFRF} for $\stdFace_{i}$ and $\zz_i$.
In the former case, the PPS condition is satisfied, we have a Lipschitzian error bound and we are done. We therefore assume that the latter case occurs with $\ell - 1 = {\dpp(\stdCone,\stdSpace+\abd)}$.

First, we compute the {\oneFRF}s for each
$\stdFace_i$.
In order to do that, we recall that
each $\stdFace_{i}$ is a direct product $\stdFace_{i}^1\times \cdots \times \stdFace_{i}^m$ where each $\stdFace_{i}^j$ is a face  of $\expCone$, excluding $\Fhat_{\infty}$ by SANE.
Therefore, a {\oneFRF} for $\stdFace_{i}^j$ can be obtained from Corollary~\ref{col:frf_2dface_entropic}, \ref{col:frf_fb} or \ref{col:frf_ne}. In particular, taking the worst\footnote{$\sqrt{\cdot}$ is ``worse'' than $\frakg_{-\infty}$ in that, near zero, $\sqrt{t} \geq \frakg_{-\infty}(t)$. The function $\frakg_{\infty}$ need not be considered because, by SANE, $\stdFace_{\infty}$ never appears.} case in consideration, and taking the maximum of the facial residual functions, there exists a nonnegative monotone nondecreasing function $\rho :\RR_+ \to \RR_+$ such that the function $\psi$ given by
\[
\psi(\epsilon,t) \coloneqq \rho(t) \epsilon +\rho(t)\sqrt{\epsilon }
\]
is a {\oneFRF} for each $\stdFace_{i}^j$. In what follows, in order to simplify the notation, we define $\samf(t) \coloneqq \sqrt{t}$. Also, for every $j$, we use $\samf_j$ to denote the composition of $\samf$ with itself $j$-times, i.e., \begin{equation}\label{eq:hatg}
\samf_j = \underbrace{\samf\circ \cdots \circ \samf}_{j \text{ times}};
\end{equation}
and we set $\samf_0$ to be the identity map.

Using the above notation and Proposition~\ref{prop:frf_prod}, we
conclude the existence of a nonnegative monotone nondecreasing function
$\sigma: \RR_+ \to \RR_+$ such that  the function $\psi _{i}$ given by
\begin{align}
\psi _{i}(\epsilon,t) \coloneqq
\sigma(t)\epsilon + \sigma(t)\samf{(\epsilon)} \notag
\end{align}
is a {\oneFRF} for $\stdFace_i$ and $\zz_i$.
Therefore, for $\xx \in B$, we have
\begin{align}
\psi _{i}(\epsilon,\norm{\xx})  \leq \sigma(M)\epsilon+ \sigma(M)\samf{(\epsilon)} =  \psi _{i}(\epsilon,M), \label{eq:frf_fi}
\end{align}
where $M = \sup _{\xx\in B} \norm{\xx}$.

Next we are going to make a series of arguments related to the following informal principle: over a bounded set only the terms $\samf_j$ with largest $j$ matter.
We start by noting that for any $\xx\in B$ and any $0\le k\le j\le \ell$,
\begin{equation}\label{relationhahaha}
  \samf_k(\wdist(\xx)) = \wdist(\xx)^{2^{-k}} = \wdist(\xx)^{(2^{-k} - 2^{-j})}\wdist(\xx)^{2^{-j}} \le \hat\kappa_{j,k}\wdist(\xx)^{2^{-j}} \le \hat\kappa\samf_j(\wdist(\xx)),
\end{equation}
where $\hat\kappa_{j,k}:= \sup_{x\in B}\wdist(\xx)^{(2^{-k} - 2^{-j})} < \infty$ because $\xx \mapsto \wdist(\xx)^{(2^{-k} - 2^{-j})}$ is continuous, and $\hat\kappa := \max_{0\le k\le j\le \ell}\hat\kappa_{j,k}$.

Now, let $\varphi_j \coloneqq \psi _{{j}}\comp \cdots \comp \psi_{{1}}$, where $\comp$ is the diamond composition defined in \eqref{eq:comp}.
We will show by induction that for every $j \leq \ell-1$ there exists $\kappa _j$ such that
\begin{equation}\label{eq:diamond_bound}
\varphi_j(\wdist(\xx),M) \leq \kappa_j\samf_{j}(\wdist(\xx)), \qquad \forall \xx \in B.
\end{equation}
For $j = 1$, it follows directly from \eqref{eq:frf_fi} and \eqref{relationhahaha}. Now, suppose that the claim is valid for some $j$ such that $j+1 \leq \ell-1$.
By the inductive hypothesis, we have
\begin{align}
\varphi_{j+1}(\wdist(\xx),M) &= \psi _{j+1}(\wdist(\xx)+ \varphi _{j}(\wdist(\xx),M),M) \notag \\
& \leq \psi _{j+1}(\wdist(\xx)+ \kappa_j\samf_{j}(\wdist(\xx)),M) \notag\\
& \leq \psi _{j+1}(\tilde \kappa_j\samf_{j}(\wdist(\xx)),M), \label{eq:varphi_j}
\end{align}
where $\tilde \kappa_j \coloneqq 2\max\{\hat \kappa,\kappa_j\}$ and the last inequality follows from \eqref{relationhahaha}.
Then, we plug $\epsilon = \tilde \kappa_j\samf_{j}(\wdist(\xx))$ in \eqref{eq:frf_fi} to obtain
\begin{align}
\psi _{j+1}(\tilde \kappa_j\samf_{j}(\wdist(\xx)),M) & = \sigma(M)\tilde \kappa_j\samf_{j}(\wdist(\xx)) + \sigma(M)\samf(\tilde \kappa_j\samf_{j}(\wdist(\xx))) \notag\\
& = \sigma(M) \tilde \kappa_j \samf_{j}(\wdist(\xx)) +
\sigma(M) \sqrt{\tilde \kappa_j} \samf_{j+1}(\wdist(\xx))  \notag \\
& \le \sigma(M)(\tilde \kappa_j\hat\kappa + \sqrt{\tilde \kappa_j})\samf_{j+1}(\wdist(\xx)) \label{eq:varphi_j2},
\end{align}
where the last inequality follows from \eqref{relationhahaha}.
Combining \eqref{eq:varphi_j} and \eqref{eq:varphi_j2} concludes the induction proof.
In particular, \eqref{eq:diamond_bound} is
valid for $j = \ell-1$.
Then, taking into account some positive rescaling and shifting (see \eqref{eq:pos_rescale}) and adjusting constants,  from \eqref{eq:bound_sane}, \eqref{eq:diamond_bound} and \eqref{relationhahaha}  we deduce that there exists $\kappa > 0$ such that
\begin{equation*}
\dist\left(\xx, (\stdSpace + \abd) \cap \stdCone\right) \leq \kappa \samf_{\ell-1}(\wdist(\xx)), \qquad \forall \xx \in B
\end{equation*}
with $\samf_{\ell-1}$ as in \eqref{eq:hatg}.
To complete the proof, we recall that  ${\dpp(\stdCone,\stdSpace+\abd)} = \ell-1$.
\end{proof}


\begin{remark}[Variants of Theorem~\ref{theo:sane}]
Theorem~\ref{theo:sane} is not tight and admits variants that are somewhat cumbersome to describe precisely. For example, the $\amf_{-\infty}$ function was not taken into account explicitly but simply ``relaxed" to $t\mapsto \sqrt{t}$.

Going for greater generality, we can also drop the SANE assumption altogether and try to be as tight as our analysis permits when dealing with possibly inSANE instances. Although there are several possibilities one must consider, the overall strategy is the same as outlined in the proof of Theorem~\ref{theo:sane}: invoke Theorem~\ref{theo:err}, fix a bounded set $B$, pick a chain of faces as in Proposition~\ref{prop:fra_poly} and upper bound the diamond composition of facial residual function as in \eqref{eq:diamond_bound}. Intuitively, whenever sums of function compositions appear, only the ``higher'' compositions matter. However, the analysis must consider the possibility of $\frakg_{-\infty}$ or $\frakg_{\infty}$ appearing.
After this is done, it is just a matter to plug this upper bound into \eqref{eq:bound_sane}.
\end{remark}

We conclude this subsection with an application. 
In \cite{BLT17}, among other results, the authors showed that when a H\"olderian error bound holds, it is possible to derive the convergence rate of several algorithms from the exponent of the error bound. As a consequence,  Theorems~\ref{thm:main_err} and \ref{theo:sane} allow us to apply some of their results  (e.g., \cite[Corollary~3.8]{BLT17}) to the conic feasibility problem with exponential cones, \emph{whenever a H\"olderian error bound holds}.
For non-H\"olderian error bounds appearing in Theorem~\ref{thm:main_err}, different techniques are necessary, such as the ones discussed in \cite{LL20} for deriving convergence rates under more general error bounds.


\subsection{Miscellaneous odd behavior and connections to other notions}\label{sec:odd}
In this final subsection, we collect several instances of pathological behaviour that can be found inside the facial structure of the exponential cone.
\begin{example}[H\"olderian bounds and the non-attainment of admissible exponents]\label{ex:exponents}
We recall Definition~\ref{def:hold} and we consider the special case of two closed convex sets $C_1,C_2$ with non-empty intersection.
We say that $\gamma \in (0,1]$ is an \emph{admissible exponent} for $C_1, C_2$ if $C_1$ and $C_2$ satisfy a
uniform H\"olderian error bound with exponent $\gamma$.
It turns out that the supremum of the set of admissible exponents is not itself admissible.
In particular, if $C_1 = \expCone$ and $C_2 = \spanVec \Fhat_{-\infty}$, then we see from Corollary~\ref{col:2d_hold} that
$C_1 \cap C_2 = \Fhat_{-\infty}$ and that $C_1$ and $C_2$ satisfy a uniform H\"olderian error bound for all $\gamma \in (0,1)$; however, in view of the sequence constructed in Remark~\ref{rem:opt}(a), the exponent cannot be chosen to be $\gamma = 1$.

In fact, from Theorem~\ref{thm:main_err} and Remark~\ref{rem:opt}(a), $C_1$ and $C_2$ satisfy an entropic error bound  which is tight and is, in a sense, better than any H\"olderian error bound with $\gamma \in (0,1)$ but worse than a Lipschitzian error bound.
\end{example}

\begin{example}[Non-H\"olderian error bound]\label{ex:non_hold}
The facial structure of $\expCone$ can be used to derive an example of two sets that provably do not have a H\"olderian error bound.	
Let $C_1 = \expCone$ and
$C_2 = \{\zz\}^\perp$, where $z_x=z_y = 0$ and $z_z=1$ so that $C_1\cap C_2=\Fhat_\infty$.
Then, for every $\eta > 0$ and every $\alpha \in (0,1]$, there is no constant $\kappa > 0$ such that
\[
\dist(\xx,\Fhat_\infty) \leq \kappa \max\{\dist(\xx,\expCone)^\alpha, \dist(\xx,\{\zz\}^\perp)^\alpha \}, \qquad \forall \ \xx \in B(\eta).
\]
This is because if there were such a positive $\kappa$, the infimum in Lemma~\ref{lem:non_hold} would be positive, which it is not. This shows that $C_1$ and $C_2$ do not have a H\"olderian error bound.
However, as seen in Theorem~\ref{thm:nonzerogammasec72}, $C_1$ and $C_2$ have a log-type error bound. In particular if
$\qq \in B(\eta)$, using \eqref{proj:p1}, \eqref{proj:p2} and Theorem~\ref{thm:nonzerogammasec72}, we have
\begin{align}
\dist(\qq, \Fhat_\infty) & \leq   \dist(\qq,\{\zz\}^\perp) + \dist(P_{\{\zz\}^\perp}(\qq),\Fhat_\infty) \notag\\
& \leq \dist(\qq,\{\zz\}^\perp) + \max\{2,2\gamma_{\zz,\eta}^{-1}\}\frakg_\infty (\dist(P_{\{\zz\}^\perp}(\qq),\expCone)) \notag\\
&\leq \wdist(\qq) + \max\{2,2\gamma_{\zz,\eta}^{-1}\}\frakg_\infty (2\wdist(\qq)) \label{eq:non_hold},
\end{align}
where $\wdist(\qq) \coloneqq \max\{\dist(\qq,\expCone),\dist(\qq,\{\zz\}^\perp) \}$ and in the last inequality we used the monotonicity of $\frakg_\infty$.
\end{example}
Let $C_1, \cdots, C_m$ be closed convex sets having nonempty intersection and let $C \coloneqq \cap _{i=1}^m C_i$.
Following \cite{LL20}, we say that $\varphi : \RR_+\times \RR_+ \to \RR_+ $  is a \emph{consistent error bound function (CEBF)} for $C_1, \ldots, C_m$ if the following inequality holds
\begin{equation*}
\dist(\xx,\, C) \le \varphi\left(\max_{1 \le i \le m}\dist(\xx, C_i), \, \|\xx\|\right) \ \ \ \forall\ \xx\in\mathcal{E};
\end{equation*}
and the following technical conditions are satisfied for every $a,b\in \RR_+$: $\varphi(\cdot,b)$ is monotone nondecreasing, right-continuous at $0$ and $\varphi(0,b) = 0$; $\varphi(a,\cdot)$ is mononotone nondecreasing. CEBFs are a framework for expressing error bounds and can be used in the convergence analysis of algorithms for convex feasibility problems, see \cite[Sections~3 and 4]{LL20}. For example, $C_1,\ldots,  C_m$ satisfy a H\"olderian error bound (Definition~\ref{def:hold}) if and only if these sets admit a CEBF of the format $\varphi(a,b) \coloneqq \rho(b)\max\{a,a^{\gamma(b)}\}$, where $\rho:\RR_+ \to \RR_+$ and $\gamma:\RR_+ \to (0,1]$ are monotone nondecreasing functions \cite[Theorem~3.4]{LL20}.

We remark that in Example~\ref{ex:non_hold}, although the sets $C_1, C_2$ do not satisfy a H\"olderian error bound, the log-type error bound displayed therein is covered under the framework of consistent error bound functions. This is because
$\frakg_\infty$ is a continuous monotone nondecreasing function and $\gamma_{\zz,\eta}^{-1}$ is monotone nondecreasing as a function of $\eta$ (Remark~\ref{rem:kappa}).
Therefore, in view of \eqref{eq:non_hold}, the function given by $\varphi(a,b) \coloneqq a + \max\{2,2\gamma_{\zz,b}^{-1}\}\frakg_\infty (2a)$ is a CEBF for $C_1$ and $C_2$.

By the way, it seems conceivable that many of our results in Section~\ref{sec:frf_comp} can be adapted to derive CEBFs for arbitrary convex sets.
Specifically,  Lemma~\ref{lem:facialresidualsbeta}, Theorem~\ref{thm:1dfacesmain}, and Lemma~\ref{lem:infratio} only rely on convexity rather than on the more specific structure of cones.

Next, we will see that we can also adapt Examples~\ref{ex:exponents} and \ref{ex:non_hold}
to find instances of odd behavior of the so-called \emph{Kurdyka-{\L}ojasiewicz (KL) property} \cite{BDL07,BDLS07,ABS13,ABRS10,BNPS17,LP18}. First, we
recall some notations and definitions.
Let $f: \RR^n\to \RR \cup \{+\infty\}$ be a proper closed convex extended-real-valued function. We denote by $\dom \partial f$ the set of points for
which the subdifferential $\partial f(\xx)$ is non-empty and by $[a < f < b]$ the set of $\xx$ such that
$a < f(\xx) < b$.
As in \cite[Section~2.3]{BNPS17}, we define for $r_0\in (0,\infty)$ the set
\begin{equation*}
  {\cal K}(0,r_0) := \{\phi\in C[0,r_0)\cap C^1(0,r_0)\;|\; \phi \mbox{ is concave}, \ \phi(0) = 0, \ \phi'(r) > 0\ \forall r\in (0,r_0)\}.
\end{equation*}
Let $B(\xx,\epsilon)$ denote the closed ball of radius $\epsilon > 0$ centered at $\xx$.
With that, we say that $f$ satisfies the KL property at $\xx \in \dom \partial f$ if
there exist $r_0 \in (0,\infty)$, $\epsilon > 0$  and  $\phi \in \mathcal{K}(0,r_0)$ such that for all $\yy \in B(\xx,\epsilon) \cap [f(\xx) < f < f(\xx) + r_0 ]$ we have
\[
\phi'(f(\yy)-f(\xx))\dist(0,\partial f(\yy)) \geq 1.
\]
In particular, as in \cite{LP18}, we say that $f$ satisfies the \emph{KL property with exponent $\alpha\in [0,1)$ at $\xx \in \dom \partial f$},  if $\phi$ can be taken to be $\phi(t) = ct^{1-\alpha}$ for some positive constant $c$.
Next, we need a result which is a corollary of \cite[Theorem~5]{BNPS17}.

\begin{proposition}\label{prop:error_kl}
Let $C_1, C_2 \subseteq \RR^n$ be closed convex sets with $C_1 \cap C_2 \neq \emptyset$. Define $f: \RR^n \to \RR$
as
\[
f(\yy) = \dist(\yy,C_1)^2 + \dist(\yy,C_2)^2.
\]
Let $\xx \in C_1\cap C_2$, $\gamma \in (0,1]$. Then, there exist $\kappa > 0$ and $\epsilon > 0 $ such that
\begin{equation}\label{eq:error}
\dist(\yy, C_1\cap C_2) \leq \kappa \max\{\dist(\yy,C_1),\dist(\yy,C_2)\}^\gamma,\qquad \forall \yy \in B(\xx,\epsilon)
\end{equation}
if and only if $f$
 satisfies the KL property with exponent $1-\gamma/2$ at $\xx$.
\end{proposition}
\begin{proof}
Note that $\inf f = 0$ and $\argmin f = C_1\cap C_2$.
Furthermore, \eqref{eq:error} is equivalent to the existence of $\kappa' > 0$ and $\epsilon > 0$ such that
\[
\dist(\yy, \argmin f) \leq \varphi(f(\yy)),\qquad \forall \yy \in B(\xx,\epsilon),
\]
where $\varphi$ is the function given by
$\varphi(r) = \kappa' r^{\gamma/2}$.
With that, the result follows from \cite[Theorem~5]{BNPS17}.
\end{proof}
%
%

\begin{example}[Examples in the KL world]	
In Example~\ref{ex:exponents}, we have two sets $C_1, C_2$ satisfying a uniform H\"olderian error bound for $\gamma \in (0,1)$ but not for $\gamma = 1$.
Because $C_1$ and $C_2$ are cones and the corresponding distance functions are positively homogeneous, this implies that for $\mathbf{0} \in C_1 \cap C_2$, a Lipschitzian error bound never holds at any neighbourhood of $\mathbf{0}$. That is, given $\eta > 0$, there is no $\kappa > 0$ such that
\[
\dist(\yy, C_1\cap C_2) \leq \kappa \max\{\dist(\yy,C_1),\dist(\yy,C_2)\},\qquad \forall \yy \in B(\eta)
\]
holds.
Consequently, the  function $f$ in Proposition~\ref{prop:error_kl} satisfies the KL property with exponent $\alpha$ for any $\alpha \in (1/2,1)$ at the origin, but not for $\alpha = 1/2$. To the best of our knowledge, this is the first explicitly constructed function in the literature such that the infimum of KL exponents at a point is not itself a KL exponent.

Similarly, from Example~\ref{ex:non_hold} we obtain $C_1,C_2$ for which \eqref{eq:error} does not hold for $\mathbf{0} \in C_1\cap C_2$ with any chosen $\kappa,\varepsilon>0,\;\gamma \in \left(0,1 \right]$. Thus from Proposition~\ref{prop:error_kl} we obtain a function $f$ that does not satisfy the KL property with exponent $\beta \in [1/2,1)$ at the origin. Since a function satisfying the KL property with exponent $\alpha\in [0,1)$ at an $\xx\in \dom \partial f$ necessarily satisfies it with exponent $\beta$ for any $\beta \in [\alpha,1)$ at $\xx$, we see that this $f$ does not satisfy the KL property with any exponent at the origin. On passing, we would like to point out that there are functions known in the literature that fail to satisfy the KL property; e.g., \cite[Example~1]{BDLS07}.


\end{example}

\section{Concluding remarks}\label{sec:conclusion}
In this work, we presented an extension of the results of \cite{L17} and showed how to obtain error bounds for conic linear systems using {\oneFRF}s and facial reduction (Theorem~\ref{theo:err}) even when the underlying cone is not amenable. 
Related to facial residual functions, we also developed techniques that aid in their computation; see Section~\ref{sec:frf_comp}.
Finally, all techniques and results developed in Section~\ref{sec:frf} were used in some shape or form in order to obtain error bounds for the exponential cone in Section~\ref{sec:exp_cone}.
Our new framework unlocks analysis for cones not reachable with the techniques developed in \cite{L17}; these include cones that are not facially exposed, as well as cones for which the projection operator has no simple closed form or is only implicitly specified. These were, until now, significant barriers against error bound analysis for many cones of interest.

As future work, we are planning to use the techniques developed in this paper to analyze and obtain error bounds for some of these other cones that have been previously unapproachable. Potential examples include the cone of $n\times n$ completely positive matrices and its dual, the cone of $n\times n$ copositive matrices.
The former is not facially exposed when $n\geq 5$ (see \cite{Zh18}) and the latter is not facially exposed when $n \geq 2$. It would be interesting to clarify how far error bound problems for these cones can be tackled by our framework. Or, more ambitiously, we could try to obtain some of the facial residual functions and some error bound results.
Of course, a significant challenge is that their facial structure is not completely understood, but we believe that even partial results for general $n$ or complete results for specific values of $n$ would be relevant and, possibly, quite non-trivial. Finally, as suggested by one of the reviewers, our framework may be enriched by investigating further geometric interpretations of the key quantity $\gamma_{\zz,\eta}$ in \eqref{gammabetaeta}, beyond Figure~\ref{fig:uvw}. For instance, it will be interesting to see whether the positivity of $\gamma_{\zz,\eta}$ is related to some generalization of the angle condition in \cite{YangNg02}, which was originally proposed for the study of Lipschitz error bounds.

\bibliographystyle{abbrvurl}
\bibliography{bib_plain}

\end{document}